\documentclass{amsart}
\usepackage{amssymb} 
\usepackage{dsfont}
\usepackage{footmisc}
\usepackage[usenames,dvipsnames,svgnames,table]{xcolor}
\usepackage[utf8]{inputenc}
\usepackage[T1]{fontenc}
\usepackage{ tipa }
\usepackage{helvet}
\usepackage{color}
\usepackage{mathrsfs}  
\usepackage{stmaryrd}  
\usepackage{amsmath,amsxtra,amsthm,amssymb,xr,enumerate,fullpage,comment, graphicx, mathtools}
\usepackage[all]{xy}
\usepackage[bbgreekl]{mathbbol}
\usepackage{bbm}
\usepackage{bm}
\usepackage{tikz-cd}
\makeatletter
\newlength{\doublefracgap}
\DeclareRobustCommand{\doublefrac}[2]{%
  \mathinner{\mathpalette\doublefrac@{{#1}{#2}}}%
}
\newcommand{\doublefrac@}[2]{\doublefrac@@#1#2}
\newcommand{\doublefrac@@}[3]{%
  \ooalign{%
    \raisebox{\doublefracgap}{$\m@th#1\frac{#2}{\phantom{#3}}$}\cr
    \raisebox{-\doublefracgap}{$\m@th#1\frac{\phantom{#2}}{#3}$}\cr
  }%
}

\makeatother

\DeclareMathSymbol{\invques}{\mathord}{operators}{`>}
\DeclareUnicodeCharacter{00BF}{\tmquestiondown}
\DeclareRobustCommand{\tmquestiondown}{%
  \ifmmode\invques\else\textquestiondown\fi
}

\usepackage{verbatim}
\numberwithin{equation}{section}

\makeatletter
\newcommand{\mylabel}[2]{#2\def\@currentlabel{#2}\label{#1}}
\makeatother

\newtheorem{theorem}{Theorem}[section]
\newtheorem{lemma}[theorem]{Lemma}

\newtheorem{proposition}[theorem]{Proposition}

\newtheorem{corollary}[theorem]{Corollary}
\newtheorem{defn}[theorem]{Definition}

\newtheorem{remark}[theorem]{Remark}

\newtheorem{lthm}{Theorem}

\newcommand{\Gal}{\operatorname{Gal}}

\newcommand{\DD}{\mathbb{D}}

\newcommand{\CC}{\mathbb{C}}

\newcommand{\NN}{\mathbb{N}}
\newcommand{\QQ}{\mathbb{Q}}
\newcommand{\Qp}{{\mathbb{Q}_p}}
\newcommand{\Zp}{\mathbb{Z}_p}
\newcommand{\ZZ}{\mathbb{Z}}
\renewcommand{\AA}{\mathbb{A}}
\newcommand{\FF}{\mathbb{F}}

\newcommand{\g}{\mathbf{g}}

\newcommand{\ord}{\mathrm{ord}}

\newcommand{\fp}{\mathfrak{p}}
\newcommand{\fq}{\mathfrak{q}}

\newcommand{\vp}{\varphi}

\newcommand{\cH}{\mathcal{H}}
\newcommand{\cO}{\mathcal{O}}

\newcommand{\GL}{\mathrm{GL}}

\newcommand{\bR}{\mathbf{R}}

\newcommand{\cyc}{\textup{cyc}}

\newcommand{\Hom}{\mathrm{Hom}}

\newcommand{\ac}{\textup{ac}}
\newcommand{\LL}{\Lambda}
\newcommand{\TT}{\mathbb{T}}

\newcommand{\RR}{\mathcal{R}}
\newcommand{\f}{\textup{\bf f}}

\newcommand{\h}{\textup{\bf h}}

\newcommand{\lra}{\longrightarrow}
\newcommand{\ra}{\lra}
\newcommand{\res}{\textup{res}}

\newcommand{\Bj}{\mathbf{j}}

\newcommand{\Bf}{\mathbf{f}}

\newcommand{\Bg}{\mathbf{g}}
\newcommand{\Bk}{\mathbf{k}}

\newcommand{\ur}{\textup{ur}}

\newcommand{\cF}{\mathcal{F}}

\newcommand{\id}{\mathrm{id}}
\newcommand{\etale}{\textup{\'et}}

\newcommand{\Bh}{\mathbf{h}}

\newcommand{\bD}{\mathbf{D}}
\newcommand{\RG}{\mathbf{R\Gamma}}
\newcommand{\Spm}{\mathrm{Spm}}
\newcommand{\cX}{\mathfrak{X}}
\newcommand{\Ddagrigf}{\mathbf{D}^\dagger_{\mathrm{rig},U_\f}}
\newcommand{\Ddagrigg}{\mathbf{D}^\dagger_{\mathrm{rig},U_\g}}
\newcommand{\Ddagrigfg}{\mathbf{D}^\dagger_{\mathrm{rig},U_\f\times U_\g}}
\newcommand{\DdagrigX}{\mathbf{D}^\dagger_{\mathrm{rig},\cX}}

\newcommand{\DdagrigE}{\mathbf D^\dagger_{\mathrm{rig},E}}

\newcommand{\p}{\mathfrak{p}}

\newcommand{\m}{\mathfrak{m}}

\newcommand{\Frac}{\mathrm{Frac}}

\newcommand{\cW}{\mathcal{W}}

\newcommand{\sA}{\mathscr{A}}

\newcommand{\cR}{\mathcal{R}}

\newcommand{\fN}{\mathfrak{N}}

\newcommand{\sW}{\mathscr{W}}

\newcommand{\Aut}{\mathrm{Aut}}
\newcommand{\cA}{\mathcal{A}}

\newcommand{\bal}{{\rm bal}}
\newcommand{\wt}{{\rm wt}}
\newcommand{\Fr}{{\rm Fr}}

\newcommand{\hatotimes}{{\,\widehat\otimes\,}}

\newcommand{\hh}{\h}
\newcommand{\Q}{\QQ}

\usepackage{hyperref}

\definecolor{pinegreen}{rgb}{0.0, 0.47, 0.44}

 \definecolor{pAlgae}{RGB}{87,115,135}
\definecolor{airforceblue}{rgb}{0.36, 0.54, 0.66}
	\definecolor{bondiblue}{rgb}{0.0, 0.58, 0.71}
\definecolor{britishracinggreen}{rgb}{0.0, 0.26, 0.15}
\definecolor{camouflagegreen}{rgb}{0.47, 0.53, 0.42}
\definecolor{darkcyan}{rgb}{0.0, 0.55, 0.55}

\hypersetup{
    colorlinks = true,
    linkcolor=blue,
     filecolor=blue,
     citecolor = darkcyan,      
     urlcolor=cyan,
    linkbordercolor = {white},
}


\begin{document}

\author{K\^az{\i}m B\"uy\"ukboduk}
\address{K\^az\i m B\"uy\"ukboduk\newline UCD School of Mathematics and Statistics\\ University College Dublin\\ Belfield\\Dublin 4\\Ireland}
\email{kazim.buyukboduk@ucd.ie}

\author{Peter Neamti}
\address{Peter Neamti\newline UCD School of Mathematics and Statistics\\ University College Dublin\\ Belfield\\Dublin 4\\Ireland}
\email{peter.neamti@ucdconnect.ie}

\dedicatory{To Henri Darmon on the occasion of his 60th birthday, with admiration.}

\title{{A} \lowercase{proof of} $p$-\lowercase{adic} G\lowercase{ross}--Z\lowercase{agier theorem via} BDP \lowercase{formula}}

\begin{abstract}
This paper provides a new proof of the $p$-adic Gross--Zagier formula for the $p$-adic $L$-function associated with the base change of a normalised cuspidal eigen-newform $f$ of weight $k \geq 2$ (and families of such) to an imaginary quadratic field $K$. Our results encompass both the classical $p$-ordinary cases and non-ordinary scenarios, including new cases where $k > 2$ and $\mathrm{ord}_p(a_p(f)) > 0$. Unlike the traditional approach of comparing geometric and analytic kernels, we employ a ``wall-crossing'' strategy centred on the BDP formula and the theory of Beilinson--Flach elements. 
\end{abstract}

\maketitle

\tableofcontents

\section{Introduction}
\label{sec_2026_01_10_1700}
Let $K$ be an imaginary quadratic field with class number $h_K$ and let $p>3$ be a prime that splits in $K/\QQ$ as $(p)=\p\p^c$, where $c$ is the non-trivial element of $\Gal(K/\QQ)$. Let $$f =\sum_{n=1}^\infty a_n(f)q^n\in S_{k+2}^{\rm new}(\Gamma_0(N_f))$$ 
be a normalised cuspidal eigen-newform, where $\ord_p (N_f)\leq 1$. We assume that $K$ verifies the strong Heegner hypothesis relative to $N_f$ and fix a factorisation $N_f\cO_K=\fN\fN^c$. When $p\nmid N_f$, we let $f^\alpha\in S_{k+2}(\Gamma_0(N_f p))$ denote its $p$-stabilisation, where 
$$\alpha^2-a_p(f)\alpha+p^{k+1}=0\,,\qquad \ord_p(\alpha)<k+1$$
and $\ord_p$ is the valuation on $\CC_p$ normalised so that $\ord_p(p)=1$. When $p \mid N_f$, we then assume that $a_p(f)\neq p^{\frac{k}{2}}$, put $\alpha=a_p(f)$ and $f^\alpha=f$. 

Our goal in the present paper is to give a proof of the $p$-adic Gross--Zagier theorem (Theorem~\ref{thm_main_intro} below) for the $p$-adic $L$-function associated to the base change of $f^\alpha$ to $K$ (cf. \S\ref{eqn_2026_02_04_1051}) under the following hypotheses:
\begin{itemize}
\item[\mylabel{itemH}{\textbf{H}})] Hypotheses \eqref{item_BI}, \eqref{item_Heeg}, and \eqref{item_Disc} given as in \S\ref{subsubsec_2026_02_06_1123} hold true.
\item[\mylabel{itemNV}{\textbf{NV}})] One of the following holds for the given prime $p$ and $K$ with $p\nmid 6h_K$:
\begin{itemize}
       \item[\mylabel{item_NV1}{\textbf{NV1}})] There exists a $p$-old $p$-ordinary form $h\in S_m(\Gamma_0(Np))$ with $p\nmid N$ and $m\geq 2$ that satisfies  \eqref{itemH} such that $L_p'(h_{/K},m/2)\neq 0$.
        \item[\mylabel{item_NV2}{\textbf{NV2}})] There exists a $p$-old $p$-supersingular form $h\in S_2(\Gamma_0(Np))$ with $p\nmid N$ that satisfies  \eqref{itemH} such that $L_p'(h_{/K},1)\neq 0$.
        \item[\mylabel{item_NV3}{\textbf{NV3}})] There exists an elliptic curve $E_{/\QQ}$ that has good supersingular reduction at $p$ satisfying \eqref{itemH} and such that ${\rm ord}_{s=1}L(E,s)=1$. 
\end{itemize}
\end{itemize}
 
\begin{lthm}
    \label{thm_main_intro}
    Assume that \eqref{itemH} holds true, and assume $p\nmid 6h_K$. Let $h_f( \,,\, )$ be the $p$-adic height pairing given as in \S\ref{subsubsec_354_2026_02_06}. Let $\mathscr{A},\,\mathscr{B}\in \overline{\QQ}$ denote the non-zero constants given as in \S\ref{subsubsec_623_2026_02_06}. Then: 
  \begin{equation}
    \label{eqn_thm_main_intro}
        \frac{\mathscr A}{\mathscr{B}}\cdot\frac{d}{ds}L_p(f_\alpha{}_{/K},s)_{\vert_{s=\frac{k}{2}+1}}=(-1)^{\frac{k}{2}}\times 
        \begin{cases}
            \left(1-\frac{p^{\frac{k}{2}}}{\alpha}\right)^{4}\cdot \dfrac{h_f( z_f,z_f)}{(4|D_K|)^{\frac{k}{2}}}\, & \hbox{ if }\, p\nmid N_f\,,\\
          4\cdot \dfrac{h_f( z_f,z_f)}{(4|D_K|)^{\frac{k}{2}}}\,  &\hbox{ if $p\mid \mid N_f$ and $a_p(f)\neq p^{\frac{k}{2}}$ }\,.
        \end{cases}
    \end{equation}
        Here, $z_f$ is the Heegner cycle associated with the base-change of $f$ to $K$ in the first case, and it is as in Definition~\ref{defn_adhoc_Heegner} in the second. Furthermore, ${\mathscr A}/{\mathscr B}=1$ if \eqref{itemNV} holds true. 
\end{lthm}

\begin{remark}
\label{remark_2026_02_06_1016}
    \normalfont If the condition \eqref{itemNV} fails, then the $p$-adic Gross--Zagier formulae proved in all the works \cite{PR,nekovarGZ, kobayashi13, kobayashi2014_GZ} (for the prime $p$ and $K$ as in the statement of the theorem) would read $0=0$ for all $p$-stabilised eigenforms whose Galois representation satisfies\footnote{\label{footnote_BI}In the situation of \eqref{item_NV3}, the big image condition \eqref{item_BI} automatically holds true thanks to \cite[Proposition~2.1]{edixhoven92} combined with \cite[Theorem 6.1]{ContiLangMedvedovsky}, since we assume $p\neq 2,3$. In the situation of \eqref{item_NV1}, one may relax \eqref{item_BI} to the requirement that the residual representation $\overline{\rho}_h$ be absolutely irreducible.} \eqref{item_BI}. In that case, \eqref{eqn_thm_main_intro} recovers the same conclusion. Otherwise, the $p$-adic Gross--Zagier formula proved in \cite{perrinriou87,nekovarGZ, kobayashi13, kobayashi2014_GZ} for \emph{some} eigenform $h$ does \emph{not} read $0=0$, in which case we have ${\mathscr A}/{\mathscr B}=1$. Finally, it seems plausible that an improvement of the main results of \cite{AlexSmith1,AlexSmith2} would imply \eqref{item_NV3}.  \hfill $\triangle$
\end{remark}

Our approach to prove Theorem~\ref{thm_main_intro} is uniform in that it simultaneously treats the known cases (cf. \cite{perrinriou87,nekovarGZ, kobayashi13, kobayashi2014_GZ} as well as \cite{disegni17}, see our review of the main results of these important papers in \S\ref{subsubsec_overview1} below) as well some new cases (e.g. when $k>2$ and $\ord_p(\alpha)>0$, which were announced in \cite{kobayashi_higherweight_nonord_GZ}). It is fundamentally different\footnote{\label{footnote_BDV_GZ} This approach was envisioned by the authors of \cite{BDV}; cf. Remark~1.3 in op. cit. for an outline, except that our methods do not involve Beilinson--Kato elements.} from the earlier approaches: In our proof of Theorem~\ref{thm_main_intro}, the key ingredient is the celebrated ``$p$-adic Waldspurger limit formula'' of Bertolini--Darmon--Prasanna (cf. \S\ref{subsubsec_overview2} for an overview), and it resonates with the ``wall-crossing principle'' (cf. \cite{BRS}, see also \cite{BS, BC, BCPvP}) emphasising the relationship between the $p$-adic variational properties of explicit GGP formulae and their arithmetic counterparts (across different GGP regions).

We remark that Theorem~\ref{thm_main_intro} is deduced from Theorem~\ref{Theorem_Big_GZ}, where we prove a $p$-adic Gross--Zagier formula for families of eigenforms. Note that Theorem~\ref{Theorem_Big_GZ} specialises also to $p$-adic Gross--Zagier formulae for generalised Heegner cycles for $f_{/K}\times \chi$ (where $\chi$ is a Hecke character of $K$ such that $\chi\circ c =\chi^{-1}$, of infinity type $(-\ell,\ell)$ with $|\ell|<\frac{k}{2}$), which reproves and generalises the main results of \cite{shnidman_padic_GZ}. 

\subsection{An overview}
\label{subsec_overview}
We review earlier works\footnote{\label{footnote_GZ_GHC} We focus on the $p$-adic Gross--Zagier theorem in the most basic setting where the strong Heegner hypothesis holds. There are many generalisations in the literature where one may relax this condition (and appeal to Heegner points associated to non-split quaternion algebras), or allow $f$ to have a central character and consider the central critical $L$-value for the appropriate self-dual twist. Due to space considerations, we have not included these in our summary.} related to the $p$-adic Gross--Zagier formulae and compare the methods of these to those in our paper.
\subsubsection{A review of earlier work}
\label{subsubsec_overview1}
In the case when $k=2$, $p\nmid N_f$ and $\ord_p(\alpha)=0$, Perrin--Riou proved \eqref{eqn_thm_main_intro}. This was later extended by \cite{nekovarGZ} to the case $k>2$ (still when $p\nmid N_f$ and $\ord_p(\alpha)=0$). Disegni in \cite{disegni17} further extended these to the non-split semi-stable case (i.e. allowing $\ord_p(N_f)=1$ but in that case requiring\footnote{\label{footnote_Disegni} We remark that if $\ord_p(N_f)=1$ and $k>2$, the form $f$ is non-$p$-ordinary, hence it escapes the treatment of \cite{disegni17}.} $k=2$ and $a_p(f)\neq 1$). 

The methods of all these papers resemble those in the work of Gross--Zagier, and they dwell on the comparison of a \emph{geometric kernel} (a modular form whose Fourier coefficients are constructed in terms of the $p$-adic heights of Heegner cycles) and the \emph{analytic kernel} (another modular form constructed using the Rankin--Selberg method that is related to the derivative of the $p$-adic $L$-function). 

A key ingredient in these works is the vanishing of the local-at-$p$ contribution of the $p$-adic height pairing, which is proved using a trick of Perrin--Riou, which indeed requires $\alpha\neq 1$ when $k=2$ (which, in turn, avoids exceptional zeroes of $p$-adic $L$-functions\footnote{\label{footnote_NekvsSch} This condition also guarantees that Nekov\'a\v{r}'s heights that we consider here (given in terms of Bockstein morphisms) coincide with those of Schneider given in terms of universal norms, or that of Nekov\'a\v{r}--Zarhin given in terms of splittings of the Hodge filtration; cf. \cite[Remark 1.3.2]{disegni17}.}). The proof of this vanishing is considerably more difficult in the non-ordinary case (i.e. when $\ord_p\,a_p(f)>0$), and it has been settled in \cite{kobayashi13,kobayashi2014_GZ} when $p\nmid N_f$ and $k=2$. 

\subsubsection{A sketch of the proof of Theorem~\ref{thm_main_intro}}
\label{subsubsec_overview2}
The key steps in the proof of Theorem~\ref{thm_main_intro} are as follows.
\begin{itemize}
    \item Description of the $p$-adic $L$-function $L_p(f_\alpha{}_{/K},s)$ in terms of Beilinson--Flach elements (via the ``first reciprocity laws''), cf. Corollary~\ref{corollary_reciprocity_for_anticyclotomic_BF}(i).
    \item Relying on the first step, a Rubin-style formula expressing $L_p'(f_\alpha{}_{/K},s)$ in terms of the $p$-adic height of the Beilinson--Flach element; cf. Proposition~\ref{Rubin_formula}.
    \item Utilising the BDP formula for Heegner cycles (Theorem~\ref{reciprocity_Z_inf}) and the ``second reciprocity laws'' for Beilinson--Flach elements (Corollary~\ref{corollary_reciprocity_for_anticyclotomic_BF}(ii)), one can compare their $p$-local images. Moreover, relying on the Euler system machinery (it is precisely this point where we need \eqref{item_BI}, the big image condition), prove that the comparison of the $p$-local images of Heegner points and Beilinson--Flach elements can be refined to a global comparison in the appropriate Selmer group they live in; cf. Proposition~\ref{prop_BF=Z}.
    \item Plug this comparison in the second step to deduce \eqref{eqn_thm_main_intro}.
\end{itemize}

We remark that the case when $\ord_p(N_f)=1$ and $a_p(f)= p^{\frac{k}{2}}$ escapes our treatment, and it is the subject of the work in progress by the first-named author with S. Kobayashi and K. Ota.
\subsubsection*{Acknowledgements} It will be clear to the reader that every single step in our argument relies crucially on the groundbreaking ideas of Henri Darmon and his coauthors. The first-named author (K.B.) wishes to thank Henri for his generosity both in sharing his ideas and also his constant support and encouragement throughout K.B.'s career. The second author (P.N.) wishes to thank the first author for their continuous encouragement and enduring patience in the preparation of this article. He would also like to thank Ra\'ul Alonso Rodr\'iguez, Antonio Cauchi, and Ju-Feng Wu for many stimulating conversations and for answering countless questions. The authors are grateful to the anonymous referee for their constructive and insightful feedback, and to Rodolfo Venerucci for an enlightening exchange.

K.B.’s research in this publication was conducted with the financial support of Taighde \'{E}ireann -- Research Ireland under Grant number IRCLA/2023/849 (HighCritical), and P.N. is supported by a NUI travelling studentship. 
\section{Preliminaries}
\label{subsec_2025_12_08_1449}

\subsection{Basic set-up}

For a profinite abelian group $G$, we let $\LL(G):=\ZZ_p[[G]]$ denote its completed group ring with coefficients in $\ZZ_p$. We denote by $\LL^\sharp(G)$ (resp. $\LL^\iota(G)$) the free $\LL(G)$-module of rank one equipped with the tautological $G$-action (resp. on which $g\in G$ acts by multiplication by the group like element $[g^{-1}]$).  For any topological $G$-module $M$, we denote by $H^i(G,M)$ the continuous cohomology of $G$ with coefficients in $M$.

We fix an embedding $\iota_p: \overline{\QQ}\hookrightarrow \mathbb{C}_p$ and assume that the image of the prime $\fp\subset \cO_K$ falls within the maximal ideal of $\cO_{\mathbb{C}_p}$. We also fix an embedding $\iota_\infty: \overline{\QQ}\hookrightarrow \mathbb{C}$ as well as an isomorphism $\mathfrak{j}:\mathbb{C}\stackrel{\sim}{\ra}  \mathbb{C}_p$ such that $\mathfrak{j}\circ \iota_\infty =\iota_p$. Let's denote the ring of integers of the completion of the maximal unramified extension $\QQ_p^\ur$ of $\Qp$ by $\sW$. We also fix a finite extension of $\QQ_p$, which we will enlarge as befitting our requirements.

Let us put $\mu_{p^\infty}:=\varinjlim \mu_{p^n} \subset \overline{\QQ}^\times$, which we view also as a subgroup of $\mathbb{C}_p^\times$ via $\iota_p$. We let $\chi_\cyc^\QQ$ denote the cyclotomic character giving the action of $G_\QQ$ on $\mu_{p^\infty}$. We denote its restriction to $G_K$ by $\chi_\cyc$. We normalise the  Hodge--Tate weights so that $\chi_\cyc^\QQ$ has weight $-1$.

\subsection{Galois representations attached to modular forms} 
\label{subsec_2026_01_11_1350}
We review the construction of Galois representations attached to modular forms, following \cite[\S6.4]{benois_buyukboduk_interpolation_BK} for the most part.

\subsubsection{} Let $h \in S_{w+2}(\Gamma_1(M),\varepsilon_h)$ be a newform, and let $V_h$ denote Deligne's $p$-adic Galois representation attached to $h$, which is realised as the $h$-isotypic quotient of $H^1_{\textup{\'et}}(Y_1(M)_{\overline{\QQ}},{\rm TSym}^{w}\mathscr{H}_{\QQ_p}(1))\otimes_{\QQ_p}E$ for the action of dual (covariant) Hecke operators, where the \'etale sheaf ${\rm TSym}^{w}\mathscr{H}_{\QQ_p}$ is as in \cite[\S2.3]{KLZ2}. 

One may realise the dual representation $V_h^*$ as the $h$-isotypic direct summand of $H^1_{\textup{\'et},c}(Y_1(M)_{\overline{\QQ}},{\rm Sym}^{w}\mathscr{H}_{\QQ_p}^\vee)$ for the action of usual (contravariant) Hecke action,  where ${\rm Sym}^{w}\mathscr{H}_{\QQ_p}^\vee$ is the dual sheaf. The natural isomorphism $\mathscr{H}_{\QQ_p}^\vee(1)\simeq \mathscr{H}_{\QQ_p}$ then induces a Hecke equivariant morphism
$$H^1_{\textup{\'et},c}(Y_1(M)_{\overline{\QQ}},{\rm Sym}^{w}\mathscr{H}_{\QQ_p}^\vee)(w+1)\lra H^1_{\textup{\'et}}(Y_1(M)_{\overline{\QQ}},{\rm TSym}^{w}\mathscr{H}_{\QQ_p}(1))$$
given by the Atkin--Lehner operator $W_M$ (normalised as in \S 2.5 of \cite{KLZ2}), which in turn yields a natural isomorphism 
\begin{equation}
\label{eqn_2026_01_10_1340}
   \lambda_M(h)^{-1}W_{M}\,:\, V_h(\mathscr{H}_{\QQ_p}^\vee):=V_{h^c}^*(w+1)\simeq V_h=:V_h(\mathscr{H}_{\QQ_p})\,,
\end{equation} 
where $h^c=h\otimes\varepsilon_h^{-1}$ is the dual form and $ \lambda_M(h)$ is the Atkin--Lehner pseudo-eigenvalue, given as in the beginning of \cite[\S 10.1]{KLZ2}.

\subsubsection{} The Poincar\'e duality pairing
$$\{\,,\,\}_{M}\,:\,\, H^1_{\textup{\'et}}(Y_1(M)_{\overline{\QQ}},{\rm Sym}^{w}\mathscr{H}_{\QQ_p}^\vee)\otimes H^1_{\textup{\'et}}(Y_1(M)_{\overline{\QQ}},{\rm TSym}^{w}\mathscr{H}_{\QQ_p})\lra \QQ_p$$
(see \cite{milne_etale_cohomology}) restricts to the perfect pairing 
\begin{equation}
    \label{eqn_2026_01_10_1428}
    \{\,,\,\}_{M}\,:\, V_{h^c}^*\otimes V_h \lra E\,.
\end{equation}
Combined with the isomorphism \eqref{eqn_2026_01_10_1340}, \eqref{eqn_2026_01_10_1428} yields a perfect Hecke equivariant pairing
$$\langle\,,\,\rangle_{M}\,:\, V_h(\mathscr{H}_{\QQ_p}^\vee)\otimes V_h(\mathscr{H}_{\QQ_p}^\vee) \lra E(w+1)\,, \qquad x\otimes y \mapsto \{x,\lambda_M^{-1}W_M(y)\}_M\,,$$
Using the fact that the normalised Atkin--Lehner isomorphism $\lambda_M(h)^{-1}W_{M}$ in  \eqref{eqn_2026_01_10_1340} is a self-adjoint involution with respect to the pairing \eqref{eqn_2026_01_10_1428}, we will denote the pairing 
\begin{equation}
\label{eqn_2026_1_10_1453}
    V_h \otimes V_h\xrightarrow{\,\langle\,,\,\rangle_{M}\,} E(w+1)
\end{equation}
also by $\langle\,,\,\rangle_{M}$.

\subsubsection{} Let us assume until the end of \S\ref{subsec_2026_01_11_1350} that $h\in S_{w+2}^{\rm new}(\Gamma_0(N))$ with $p\nmid N$. In this case, we have $h^c=h$. 

Let $\lambda$ be a root of the $p^{\rm th}$ Hecke polynomial $X^2-a_p(h)X+p^{w+1}$ of $h$. We denote by $V_{h^\lambda}$ the $h^\lambda$-isotypic quotient of the cohomology of $H^1_{\textup{\'et}}(Y_0(Np)_{\overline{\QQ}},{\rm TSym}^{w}\mathscr{H}_{\QQ_p})(1)$. As above, we will denote this representation by $V_{h^\lambda}(\mathscr{H}_{\QQ_p})$ whenever it is prudent to do so. Similarly, we also have the realisation 
\begin{equation}
\label{eqn_2026_1_10_1404}
   H^1_{\textup{\'et}}(Y_0(Np)_{\overline{\QQ}},{\rm Sym}^{w}\mathscr{H}_{\QQ_p}^\vee)(w+1)\supset V_{h^\lambda}^*(w+1)=:V_{h^\lambda}(\mathscr{H}_{\QQ_p}^\vee) \xrightarrow[\sim]{\,W_{Np}\,} V_{h^\lambda}(\mathscr{H}_{\QQ_p})
\end{equation} 
as the $h^\lambda$-isotypic direct summand, where the isomorphism is given by the Atkin--Lehner operator $W_{Np}$.

\subsubsection{} We remark that the Abel--Jacobi images of the classical Heegner cycles land in the cohomology with coefficients in the cohomological Galois representation $V_h(\mathscr{H}_{\QQ_p}^\vee)$. However, following \cite{KLZ2,JLZ}, one interpolates Galois representations $V_{h^\lambda}(\mathscr{H}_{\QQ_p})=:V_{h_\lambda}$ to a big Galois representation, and big Heegner cycles will take coefficients in this big Galois representation. It is therefore important to note how these two spaces are identified.

\subsubsection{} We have the $p$-stabilisation isomorphisms $V_h\xrightarrow{({\rm pr}^{\lambda})^*} V_{h^\lambda}$ (resp., $V_{h^\lambda}\xrightarrow{{\rm pr}^{\lambda}_*} V_h$) induced from the pull-backs (resp. push-
forwards) of the degeneracy morphisms $Y_0(Np)\rightrightarrows Y_0(N)$, cf. \cite[7.3]{KLZ2}. The morphism $({\rm pr}^{\lambda})^*$ is the adjoint of ${\rm pr}^{\lambda}_*$ with respect to the Poincar\'e duality pairing $\{\,,\,\}_{Np}$, which, combined with isomorphism \eqref{eqn_2026_1_10_1404}, induces a perfect pairing 
$$\{\,,\,\}_{Np}'\,:\,V_{h^\lambda} \otimes V_{h^\lambda} \lra E(w+1)\,,\qquad x\otimes y \mapsto \{W_{Np}^{-1}(x),y\}_{Np}\,.$$
We then have the following commutative diagram:
\begin{equation}
    \label{label_eqn_2026_01_09_2051}
    \begin{aligned}
        \xymatrix@C=.1cm{
        V_{h^\lambda} &\otimes& V_{h^\lambda} \ar[d]^-{{\rm pr}^{\lambda}_*}
        \ar[rrrrrrr]^-{\{\,,\,\}_{Np}'} &&&&&&& E(w+1)\ar@{=}[d]\\
        V_h \ar[u]^-{W_{Np}\circ ({\rm pr}^{\lambda})^*}&\otimes & V_h \ar[rrrrrrr]_-{\langle \,,\,\rangle_N} &&&&&&& E(w+1)
        }
    \end{aligned}
\end{equation}

\subsection{Iwasawa algebras}\label{subsec_Iwasawa_Algebras}
Let $K(p^\infty):= \varinjlim K(p^n),$ where $ K(p^n)$ is the ray-class field modulo $p^n$. Then for $\fq \in \{\fp, \fp^c\}$, we have the subfields $K(\mu_{p^\infty})$,  $H_{p^\infty}$, $K(\fq^ \infty)$ of $K(p^\infty)$, where $H_{p^{\infty}}:= \varinjlim H_{p^n}$ and $H_{p^n}$ is the ring class field of $K$ modulo $p^n$, and $K(\fq^\infty):= \varinjlim K(\fq^n)$ where $K(\fq^n)$ is the ray class field of $K$ modulo $\fq^n$. Let us set
\begin{align}
    &\Gamma_{\rm cyc}:= \Gal(K(\mu_{p^\infty})/K), &\Gamma_{\rm ac}:= \Gal(H_{p^\infty}/K), &&\Gamma_{\fq}:= \Gal(K(\fq^\infty)/ K), &&\Gamma:= \Gal(K(p^\infty)/K). \notag 
\end{align}
Let $K_\infty \subset K(p^\infty)$ denote the $\Zp^2$-extension of $K$. We write $K_{\cyc}\subset K_\infty \cap K(\mu_{p^\infty})$ (resp.\ $K_{\rm ac}\subset K_\infty \cap H_{p^\infty}$) for the cyclotomic (resp.\ anticyclotomic) $\Zp$-extension of $K$. For $\fq \in \{\fp, \fp^c\}$, let $K_\infty^{(\fq)}\subset K_\infty \cap K(\fq^{\infty})$ be the $\Zp$-extension of $K$, unramified outside $\fq$. For $?\in\{\cyc,\ac\}$, set 
\begin{align}
    &\Gamma_?^\circ := \Gal(K_?/K), &\Gamma^\circ_{\fq}:=\Gal(K^{(\fq)}_\infty/K), &&\Gamma^\circ := \Gal(K_\infty/K). \notag
\end{align}
Let $h_p \in \NN$ be such that ${\rm im}\left(G_{K_{\fp}}\to \Gamma_{\fp}^\circ\right)=(\Gamma_{\fp}^\circ)^{p^{h_p}}$. We remark that $h_p=0$ if the class number $h_K$ of $K$ is coprime to $p$.

Let $c$ denote the non-trivial element of $\Gal(K/\Q)$. Then $c$ acts on $\Gamma^\circ$ by $g \mapsto \widetilde{c}^{-1}g \widetilde{c}$, where $\widetilde{c}$ denotes any lift of $c$ to $\Gal(K_\infty/\Q)$ and $g\in \Gamma^\circ$. We have the decomposition 
\[\Gamma^\circ = \Gamma^+ \times \Gamma^-\,, \] 
where $\Gamma^\pm$ denotes the rank-one $\Zp$-submodules on which $c$ acts by $\pm1$. Under the natural projections, $\Gamma^+$ (resp.\ $\Gamma^-$) maps isomorphically onto $\Gamma^{\circ}_\cyc$ (resp.\ $\Gamma^{\circ}_{\ac}$). Following the conventions in \cite{burungale2024zetaelementsellipticcurves}, we fix topological generators $\gamma_?\in\Gamma^\circ_?$ for $?\in\{\pm,\cyc,\ac,\fp,\fp^c\}$,  compatible with these identifications and satisfying $c\,\gamma_{\fp}\,c^{-1}=\gamma_{\fp^c}$. Moreover we require that $\gamma_+ \mapsto \gamma_\fq^{p^{h_p}/2}$, and  $\gamma_- \mapsto \gamma_\fp^{1/2}$ and $\gamma_- \mapsto \gamma_{\fp^c}^{-1/2}$.\\\\
For $?\in\{\emptyset,\pm,\cyc,\ac,\fq\}$, and $\mathcal{O}$ some $p$-adically complete $\Zp$-algebra (most commonly for us, the ring of integers of some finite extension of $\Qp$), we define the Iwasawa algebra \[\LL_\mathcal{O}(\Gamma^\circ_?) := \mathcal{O} \llbracket \Gamma^\circ_? \rrbracket.\] Unless there is danger of confusion, we will often drop the $\mathcal{O}$ from the notation and just write $\Lambda(\Gamma^\circ_?)$. If $E$ is some finite extension of $\QQ_p$, let us also set \[\mathcal{H}_E:= \{f(X)\in E[[X]] : f(X) \text{ converges on ann}(0,1)\}\]
where $\text{ann}(0,1) = \{x\in \mathbb{C}_p: |x|_p<1\}$ and define, for $?\in\{{\cyc, \ac, \fq}\}$ \[\mathcal{H}_E(\Gamma^\circ_?):=\{f(\gamma_? -1): f(X) \in \mathcal{H}_E\}. \] If $A$ is some affinoid algebra over $E$, then we also set \[\mathcal{H}_A := \mathcal{H}_E \hatotimes A.\]

The canonical map $\Gamma^+\times \Gamma^- \xrightarrow{\sim} \Gamma^\circ $ gives the isomorphism
\begin{equation}\label{eq_isom_+_-}
\Lambda(\Gamma^+) \hatotimes \Lambda(\Gamma^-) \xrightarrow{\sim} \Lambda(\Gamma^\circ).
\end{equation}
For $(\star, \bullet) \in (\rm{cyc}, \rm ac)$ or $(\rm cyc, \fp)$, the natural projections give $\Gamma^\circ \xrightarrow{\sim}\Gamma^\circ_\star \times \Gamma^\circ_\bullet$ and hence gives the isomorphism 
\begin{equation} \label{eq_isom_star_bullet}
\Lambda(\Gamma^\circ) \xrightarrow{\sim} \Lambda(\Gamma^\circ_\star) \hatotimes \Lambda(\Gamma^\circ_\bullet) .    
\end{equation} 
By our choices of generators, we see that the isomorphism obtained by composing \eqref{eq_isom_+_-} and \eqref{eq_isom_star_bullet} is given explicitly by 
\[\begin{cases}
    \gamma_+ \mapsto \gamma_{\rm cyc}, &\gamma_- \mapsto \gamma_{\rm ac}, \ \ \ \ \text{if } (\star, \bullet) = (\rm cyc, \rm ac)\\
    \gamma_+ \mapsto (\gamma_{\rm cyc}, \gamma_\fp^{p^{h_p}/2}), &\gamma_- \mapsto \gamma_\fp^{1/2},  \ \ \text{if } (\star, \bullet) = (\rm cyc, \fp).
\end{cases}\]

For $? \in \{\cyc, \ac, \fp \}$, we have \begin{equation}
    \Gamma_? \cong \Gamma^\circ_? \times \Delta_?, \notag
\end{equation}
where $\Delta_?$ is a finite abelian group, and hence we have the decomposition 
\begin{equation}
   \Lambda(\Gamma_?) \cong \bigoplus_{\psi \in \widehat{\Delta}_?} \Lambda_\psi(\Gamma^\circ_?):= \bigoplus_{\psi \in \widehat{\Delta}_?} \Lambda(\Gamma_?^\circ)\cdot e_{\psi}, \label{equation_Lambda(Gamma)_decomposition}
\end{equation}
where $\widehat{\Delta}_?:= \Hom(\Delta_?, \CC_p^\times)$ and $e_{\psi} = \frac{1}{|\Delta_?|}\sum_{\sigma \in \Delta_?} \psi^{-1}(\sigma) \cdot \sigma$ is the idempotent associated to $\psi$.

In what follows, it will be understood that $\psi$ belongs to $\Delta_?$  whenever it appears in the expression $\Lambda_\psi(\Gamma^\circ_?)$. 

\begin{defn} \label{definition_fp_to_ac}
    Suppose $\tau=\tau_\p$  is an element of the character group $\widehat{\Delta}_\p$. Let us denote by ${}^{\rm a}{\tau}$ the Hecke character attached to $\tau$ by class field theory. We define the ring class character ${}^{\rm a}{\tau}_{\ac} := {}^{\rm a}\tau\big{/}{{}^{\rm a}\tau^c}$ and we let $\tau_{\ac} \in \widehat{\Delta}_{\ac}$ denote the Galois character associated with ${}^{\rm a}{\tau}_{\ac}$.
\end{defn}
%\orange{\textbf{Q}: Can I say something about the image of this map from $\widehat{\Delta}_\p \rightarrow \widehat{\Delta}_\ac$?}
\subsubsection{Weight space}
Let $\mathcal O$ be the ring of integers of a finite extension $E$ of $\mathbb Q_p$ (which we shall enlarge as our purposes require). Let us denote by $[ \,\cdot\, ] \colon \mathbb Z_p^{\times} \hookrightarrow \LL_{\rm wt}^{\times}$ the tautological injection, where $\LL_{\rm wt}:=\Lambda(\mathbb Z_p^{\times})$. We define the universal weight character $\bbchi$ as the composite map $G_{\mathbb Q} \xrightarrow{\chi_{\cyc}} \mathbb Z_p^{\times} \hookrightarrow \Lambda_\wt^{\times},$  where $\chi_{\cyc}$ is the cyclotomic character giving the Galois action on the $p$-power roots of unity.

We regard integers as elements of the weight space $\cW={\rm Spf}\,\LL_{\rm wt}$ via 
$$\mathbb Z \lra \Hom_{\mathcal O}(\Lambda_{\wt}, \mathcal O)=\cW(\cO)\,, \qquad n \mapsto (\nu_n \colon [x] \mapsto x^n).$$  
We call a ring homomorphism $\Lambda_{\wt} \xrightarrow{\nu} \mathcal O$ an arithmetic specialisation of weight $k \in \mathbb{N}$ if $\nu$  agrees with $\nu_k$ on a finite-index subgroup of $[\ZZ_p^\times]$.

For an integer $k$, we define 
$\Lambda_\wt^{(k)}:=e_{k}\LL_{\rm wt} \simeq \Lambda(1+p\mathbb Z_p)$ as the component determined by the weight $k$.

\subsection{Hida families} 
\label{subsec_Hida_families}
We let $\h=\sum_{n=1}^{\infty} \mathbb{a}_{n}(\h)q^n \in \mathbb{I}_{\h}[[q]]$ denote the branch of the primitive non-Eisenstein Hida family of tame conductor $N_\h$ and nebentype character $\varepsilon_\h=\varepsilon_\h^{(t)}\varepsilon_\h^{(p)}$ (where the conductor of the Dirichlet character $\varepsilon_\h^{(t)}$ divides $N_\h$, and the conductor of $\varepsilon_\h^{(p)}$ divides $p$)\,, which admits a crystalline specialisation $h_{\circ}$ of weight $k$. Here, $\mathbb{I}_{\h}$ is the branch (i.e. the irreducible component) of Hida's universal ordinary Hecke algebra determined by $h_{\circ}$, module-finite flat algebra over $\LL_{\rm wt}^{(k)}$. We denote by $ \bbchi_{\f}$ the compositum
\[
   G_\QQ\xrightarrow{\bbchi} \LL_{\rm wt}^\times \twoheadrightarrow \LL_{\rm wt}^{(k),\times}\lra \mathbb{I}_{\h}^\times\,.
\]
We let $\m_\h < \mathbb{I}_{\h}$ denote the maximal ideal.

\subsubsection{} Let $\Sigma$ be a finite set of places of $\QQ$ that contains all those that divide $pN_{\h}$, as well as the archimedean prime. By the fundamental work of Hida, there exists a unique Galois representation 
$$ \rho_{\h}\,: \,G_{\mathbb Q, \Sigma} \rightarrow \GL_2(\Frac(\mathbb{I}_{\h}))  $$ 
attached to $\h$, that interpolates Deligne's two-dimensional Galois representations attached to classical specialisations of $\h$ (cf. \S\ref{subsubsec_224_2025_12} below).  We denote by $V_{\h} \subset \text{Frac}(\mathbb{I}_{\h})^{\oplus 2}$ the Ohta lattice. More precisely, $V_{\h}$ corresponds to $M(\h)^*\otimes \LL_\mathbf{a}$ in the notation of \cite[\S7.5]{KLZ2}, which realises the Galois representation $\rho_{\h}$ in the \'etale cohomology groups of a tower of modular curves. 

\subsubsection{} 
\label{subsubsec_2025_12_14_1514}
Let us consider the following conditions on $V_{\h}$.
\begin{enumerate}
 \item[\mylabel{item_Irr}{\textbf{Irr}})] The residual $G_\QQ$-representation $V_{\h}\otimes \mathbb{I}_\h/\m_\h=:\overline{V}_\h$ is irreducible.
 \item[\mylabel{item_Dist}{\textbf{Dist}})] The semi-simplification of $\overline{V}_\h\,{\vert_{G_{\QQ_p}}}$ is not scalar.
\end{enumerate}
We remark that (a) ensures that any $G_{\mathbb Q}$-stable lattice in the field of fractions is homothetic to $V_{\h}$, whereas (b) supplies one with an integral $p$-ordinary filtration. When (a) or (b) fails, then the characterisation of a lattice $V_{\h}$ is much more subtle. We shall elaborate on this point later, following \cite{burungale2024zetaelementsellipticcurves, BDV} closely, where $\h$ will be taken as a particular CM Hida family.

\subsubsection{}
\label{subsubsec_213_2025_12_09_1556}
Let us assume henceforth that $\varepsilon_\h=\varepsilon_\h^{(t)}$. Let us denote by $\h^c:=\h\otimes \varepsilon_\h^{-1}$ the conjugate family. As remarked in \cite[Lemma 3.4]{LoefflerCMB}, the family $\h^c$ is also primitive of level $N_\h$. On identifying $\mathbb{I}_{\h^c}$ with $\mathbb{I}_\h$, we have a perfect $G_\QQ$-equivariant Poincar\'e duality pairing
\begin{equation}
    \label{eqn_2025_12_09_1114}
    V_{\h^c} \otimes_{\mathbb{I}_\h} V_{\h}\xrightarrow{\,\langle\,\,,\,\,\rangle_\h\,} \bbchi_\h\chi_\cyc^\QQ\,,
    \end{equation}
The pairing $\langle\,\,,\,\,\rangle_\h$ has the following interpolative properties (cf. \cite[\S1.1.8--9]{BCPvP} and \cite[\S3.1.1]{BL-GHC}). Let $\kappa$ be an $E$-valued classical specialisation of $\mathbb{I}_\h$ of weight $\wt(\kappa)=w$ and trivial wild character, so that $\h_\kappa \in S_{w+2}(\Gamma_1(N_\h)\cap \Gamma_0(p^r))$ for some natural number $r$.

\begin{itemize}
   %\item[\mylabel{item_P0}{\textbf{P0}})] L Then the following diagram commutes:
   %$$\xymatrix{ 
   %V_{\h} \otimes_{\mathbb{I}_\h} V_{\h^c}\ar[rr]^-{\,\langle\,\,,\,\,\rangle_\h\,}\ar[d]_{\kappa}&& \bbchi_\h\chi_\cyc^\QQ \ar[d]^{\kappa}\\
%    V_{\h_\kappa} \otimes_{E} V_{\h^c_\kappa}\ar[rr]_-{\,\kappa\,\circ\,\langle\,\,,\,\,\rangle_\h\,}&& E(w+1)\,.  }$$
 \item[\mylabel{item_P1}{\textbf{P1}})] Assume in addition that either $\h_\kappa$ is a newform with $r>0$. Then the pairing $\langle\,\,,\,\,\rangle_{\h_
    \kappa}$ can be described as follows: 
$$\kappa\circ \langle x,y\rangle=\alpha(\kappa)^{r}\{x,W_{Np^r}^{-1}(y)\}_{Np^r}=\alpha(\kappa)^{r}\lambda_{N_\h p^r}(\h_\kappa)^{-1}\langle x,y\rangle_{Np^r}\,,$$
where $\langle\bullet,\circ\rangle_{Np^r}$ is the pairing \eqref{eqn_2026_1_10_1453}.

  \item[\mylabel{item_P2}{\textbf{P2}})] 
  Suppose that $\h_\kappa\in  S_{w+2}(\Gamma_0(N_\h p))$ is $p$-old (in particular, we implicitly assume that $\varepsilon_\h$ is trivial), and let $\h_\kappa^\circ \in S_{w+2}(\Gamma_0(N_\h))$ denote the newform associated to $\h_\kappa$.   Then, the following diagram commutes:
$$
   \xymatrix@C=0.1cm{ 
    V_{\h_\kappa} &\otimes_{E}& V_{\h_\kappa}\ar[rrrrrrr]^-{\kappa\,\circ\,\langle\,\,,\,\,\rangle\,}\ar[d]^-{{\rm pr}_*^{\alpha(\kappa)}} &&&&&&& E(w+1)
    %^{\times\,\lambda_{\h_\kappa^\circ}\,\alpha(\kappa)\,\mathcal{E}(\h_\kappa)\,\mathcal{E}^*(\h_\kappa)}
    \\
    V_{\h_\kappa^\circ} \ar[u]^{U_p^{-1}W_{Np}\, ({\rm pr}^{\alpha(\kappa)})^*} &\otimes_{E} &V_{\h_\kappa^\circ}\ar[rrrrrrr]_{\,\langle\,\,,\,\,\rangle_{\h_
    \kappa^\circ}\,} &&&&&&& E(w+1)\ar@{=}[u]\,,
    %_{\times\,\alpha(\kappa)}\,,
   }$$
   %where $\mathcal{E}(\h_\kappa)=\left(1-\frac{p^{w}}{\alpha(\kappa)^2}\right)$ and $\mathcal{E}^*(\h_\kappa)=\left(1-\frac{p^{w+1}}{\alpha(\kappa)^2}\right)$.
   cf. \eqref{label_eqn_2026_01_09_2051} and note the identity $U_p'W_p^{-1}U_p^{-1}=W_p^{-1}$ on $p$-old forms of level $\Gamma_0(Np)$. 
 \end{itemize}

\subsubsection{}
\label{subsubsec_213_2025_12_09}
We will sometimes restrict a Hida family $\h$ to a wide open disc $U_\mathbf{h}:={\rm Spm}(\LL_\h)$ (where $\LL_\h\simeq \cO\left[\left[\frac{X-k}{e}\right]\right][1/p]$ for some $e\in E$) about a classical point $k$ and contained in ${\rm Spm}(\mathbb{I}_\h[\frac{1}{p}])$; which we will shrink as our purposes require.

\subsection{Coleman families} 
\label{subsec_Coleman_families_revisited}
Our approach to the proof of $p$-adic Gross--Zagier theorems applies equally well also for eigenforms of positive slope, at the expense of working with non-ordinary components (i.e. Coleman families) of the eigencurve. In this subsection, we briefly review the basic definitions and properties of such families. We remark that, if one restricts a Hida family $\h$ as above to a wide open disc $U_\h$ as in \S\ref{subsubsec_213_2025_12_09}, the discussion in \S\ref{subsec_Coleman_families_revisited} recovers that in \S\ref{subsec_Hida_families} away from the neighbourhoods of non-\'etale points on the eigencurve. 

\subsubsection{}
\label{subsubsec_221_2025_12_13_1225}
For each  $r=p^{v}<p^{\frac{p-1}{p-2}}$ with $v=\ord_p(e)$ for some $e\in E$, we let $B(r)\subset \cW$ denote the closed affinoid disc about $k$ of radius $r$. Let 
$$\mathscr{A}(r)\simeq E\left\langle\left\langle \frac{X-k}{e}\right\rangle\right\rangle:= \left\{\sum_{n=0}^\infty c_n\left(\frac{X-k}{e}\right)^n\in E\left[\left[\frac{X-k}{e}\right]\right]:\lim_{n\rightarrow\infty}c_n=0\right\}
\subset \LL_{k,r}[1/p]$$ 
denote the Noetherian ring of $E$-valued analytic functions on $B(r)$, where $\LL_{(k,r)}:=\cO_L\left[\left[\frac{X-k}{e}\right]\right]$.

\subsubsection{}  
\label{subsubsec_222_2025_12_14}
Let $h^\circ \in S_{k}(\Gamma_1(N^\circ),\varepsilon)$ be a newform of weight $k\geq 2$ and level $N^\circ$. Assume either that $\gcd(N^\circ,p)=1$, or else $p||N^\circ$ and $p \nmid {\rm cond}(\varepsilon)$. Let $h^\lambda=h$ be its $p$-stabilisation with $U_p$-eigenvalue $\lambda$ (if $N^\circ$ is divisible by $p$, then $h=h^\circ$ and ${\rm ord}_p(\lambda)=\frac{k-2}{2}$). Suppose also that $h$ is non-$\theta$-critical in the sense of Coleman (cf. \cite{bellaiche2012}, Definition 2.12).

For a sufficiently small number $r_0$ as above, there exists a unique Coleman family $\h$ (up to Galois conjugation) over the affinoid disc $B(r_0)$ of fixed slope $v(\lambda):=\ord_p(\lambda)$. Let us fix henceforth $r<r_0$ as above, and put $\LL_\h:=\LL_{(k,r)}[1/p]$.  We may then consider $\h$ as a family over the wide open disc $U_\h:={\rm Spm}(\LL_\h)$. That means, $\h$ admits a formal $q$-expansion 
$$\h=\sum_{n=1}^\infty \mathbb{a}_n(\h)q^n\,\in\, \LL_\h[[q]]\,.$$ 
Let $\bbalpha:=\mathbb{a}_p(\h)\in \LL_\h$ denote the eigenvalue with which $U_p$ acts on $\h$. We define the tame level of the Coleman family as $N_\h:=N^\circ/\gcd(N^\circ,p)$ and its tame nebentype $\varepsilon_\h:=\varepsilon$ (which has conductor dividing $N_\h$).

\subsubsection{}  We let $B(r)_{\rm cl}:=\mathbb{N}\cap B(r)$ denote the set of  classical points in the affinoid disc $B(r)$. An element $\kappa\in B(r)_{\rm cl}\subset U_\h (E)= {\rm Hom}(\LL_\h, E)$ is called a classical specialisation, and the corresponding ring homomorphism $\LL_\h\to E$ is denoted by ${\rm sp}_\kappa$. For such $\kappa$, let us fix a generator $P_\kappa\in \LL_\h$ of $\ker({\rm sp}_\kappa)$. We then have 
$$\h(\kappa):=\sum_{n=1}^\infty \mathbb{a}_n(\h)_{\vert_\kappa}\,q^n\in S_{\kappa+2}(\Gamma_1(N_\h)\cap \Gamma_0(p))\,.$$ 
We call $\h(\kappa)$ the specialisation of $\h$ of weight $\kappa$. On shrinking $B(r)$, one can ensure for $k\neq \kappa\in B^\circ(r)_{\rm cl}$ that the $p$-stabilised eigenform $\h(\kappa)$ is $p$-old and it arises as the $p$-stabilisation of a newform $\h(\kappa)^\circ \in S_\kappa(\Gamma_1(N_\h))$ with $U_p$-eigenvalue $\bbalpha(\kappa)=\mathbb{a}_p(\h)_{\vert_\kappa}$.

\subsubsection{}
\label{subsubsec_224_2025_12}
As explained in \cite[Theorem 4.6.6]{LZ0}, there exists a free $\LL_\h$-module $V_\h$ of rank $2$ equipped with a continuous $G_{\QQ,\Sigma}$-action (where, as before, $\Sigma$ is a finite set of places containing prime divisors of $pN^\circ$ and the archimedean place). We note that the ring $\LL_\h$ is denoted by $B_U$ in op. cit., whereas our $V_\h$ coincides with Loeffler and Zerbes's $M_U(\h)^*$. The specialisation morphism ${\rm sp}_\kappa$ induces isomorphisms 
$V_\h\otimes_{{\rm sp}_\kappa}E\simeq V_\h/P_\kappa V_\h \xrightarrow{\sim} V_{\h(\kappa)}$ of $G_{\QQ,\Sigma}$-modules, where $V_{\h(\kappa)}$ is Deligne's Galois representation attached to $\h(\kappa)$. 

\subsubsection{}
\label{subsubsec_213_2025_12_09_1556_bis}
As in \S\ref{subsubsec_213_2025_12_09_1556}, let us denote by $\h^c:=\h\otimes\varepsilon_\h^{-1}$ the dual Coleman family. Similar to $\eqref{eqn_2025_12_09_1114}$, we have a perfect $G_\QQ$-equivariant Poincar\'e duality pairing
\begin{equation}
    \label{eqn_2025_12_09_1603}
    V_\h \otimes_{\LL_\h} V_{\h^c}\xrightarrow{\,\langle\,\,,\,\,\rangle_\h\,} \bbchi_\h\chi_\cyc^\QQ\,,
    \end{equation}
cf. \cite[Theorem 4.6.6]{LZ1}. We remark here a subtle difference with the duality \eqref{eqn_2025_12_09_1114} for $p$-ordinary families: The pairings on Deligne's representations induced from the Poincar\'e duality do not interpolate integrally. In other words, the pairing \eqref{eqn_2025_12_09_1603} does \emph{not} restrict to a perfect pairing on the Galois stable $\LL_{(k,r)}$-lattices (which exist, cf. \cite{BL-GHC}, Remark 1.4).

Finally, we note that the pairing \eqref{eqn_2025_12_09_1603} also enjoys the properties \eqref{item_P1} (with $r=1$) and \eqref{item_P2}.

\subsubsection{Self-dual families of Galois representations}
In the situation of \S\ref{subsubsec_213_2025_12_09_1556} or \S\ref{subsubsec_213_2025_12_09_1556_bis}, suppose that the Dirichlet character $\varepsilon_\h$ admits a square-root. In that case, we set 
$$V_\h^\dagger:=V_\h \otimes \bbchi_\h^{-\frac{1}{2}}\varepsilon_\h^{-\frac{1}{2}}\,,$$ so that the Poincar\'e duality pairing $\langle\,,\,\rangle_\h$ induces a perfect pairing
$$V_\h^\dagger \otimes_{?_\h} V_\h^\dagger \lra \chi_\cyc^\QQ\,,\qquad ?=\mathbb{I},\LL\,.$$
In other words, $V_\h^\dagger$ is a self-dual family of Galois representations. We similarly define $V_h^\dagger$ and $V_{h^\circ}^\dagger$ for specialisations $h$ of $\h$.

\subsubsection{Hypotheses}
\label{subsubsec_2026_02_06_1123}
Let $K/\mathbb{Q}$ be an imaginary quadratic extension.  We shall consider the following hypotheses for $h$ and $\h$ as above, where the first two are denoted $\textbf{E}_1$ and $\textbf{E}_3$ in \cite{BDV}, respectively.
    \begin{itemize}
        \item[\mylabel{item_Reg}{\textbf{Reg}})] $k\geq 2 $ and $h$ is a non-critical $p$-regular eigenform.
        \item[\mylabel{item_Eis}{\textbf{Eis}})] $k = 1$ and $h$ is the $p$-stabilisation of a $p$-irregular Eisenstein series of weight $1$ and level $N_{\mathbf{h}}$.
       % \item[\mylabel{item_rank_1}{\textbf{Rank}})] ${\ord}_{s=\frac{k}{2}}L(h_{/K},s)=1$. 
        \item [\mylabel{item_Disc}{\textbf{Disc}})] $K\neq \QQ(\sqrt{-1}), \QQ(\sqrt{-3})$, and $\rm{disc}(K/\QQ)$ is odd.
        \item[\mylabel{item_Heeg}{\textbf{Heeg}})] All rational primes dividing $N_\h$ are split in $K$.
         \item[\mylabel{item_Sign}{\textbf{Sign}})] The global root number of the base change of $h$ to $K$ equals $-1$.
    \end{itemize}
Of course, \eqref{item_Heeg} implies \eqref{item_Sign}. We will also need the following ``big image'' condition for some of our results.
\begin{itemize}
\item[\mylabel{item_BI}{\textbf{BI}})] $\{M\in \GL_2(\Zp):\det(M)\in (\Zp^\times)^{k-1}\}\subset {\rm im}\left(G_{\QQ}\stackrel{\rho_h}{\longrightarrow} \Aut_{\cO_L}(T_h)\cong \GL_2(\cO_L)\right)$
\end{itemize}
\begin{remark}
\label{remark_regularity_mod_p_rep}
\normalfont
Suppose that the hypothesis \eqref{item_BI} holds. Then the residual representation $\overline{\rho}_h$ is full, in the sense that ${\rm im}(\overline{\rho}_h)$ contains a conjugate of ${\rm SL}_2(\FF_p)$. This implies that the projective image of $\overline{\rho}_h$ contains ${\rm PSL}_2(\FF_p)$ up to conjugation. As explained in \cite[Remark 2.28]{ContiLangMedvedovsky}, it follows that the regularity condition in \cite[Theorem 6.2]{ContiIT2016} for $\overline{\rho}_h$, as well as \cite[Corollary 7.2]{ContiIT2016} for its deformation to the Coleman family $\h$, hold true.   \hfill $\triangle$
\end{remark}

\section{Selmer complexes and height pairings}
\label{sec_selmer_complexes_heights}
In this section, we review the general theory of Selmer complexes (due to Nekov\'a\v{r} and Pottharst; cf. \cite{nekovar06, benoisheights}), following closely \cite[\S1.2]{BBMemoirs}. 

Let $F$ be an algebraic number field and let $S_p$ (resp., $S_\infty$) denote the set of primes of $F$ above $p$ (resp., archimedean places of $F$). Let us fix a finite set of places $S$ of $F$ containing $S_p$. We denote by $G_{F,S}$ the Galois group of the maximal algebraic extension of $F$ unramified outside $S\cup S_\infty$. Let $\cX = \Spm(\cO_{\cX})$ be an affinoid over $E$. We fix a continuous $G_{F,S}$-representation $V$ with coefficients in $\cO_{\cX}$. In applications,  $\cX$ will be chosen as $\Spm\, \mathscr{A}(r)$ as in \S\ref{subsubsec_221_2025_12_13_1225}. When $\cX=\Spm(E)$ is a singleton, we fix a $G_{F,S}$-stable $\cO_E$-lattice $T\subset V$.

\subsection{Unramified local conditions}
\label{subsec_1211_2025_07_06_1328}
Let $\lambda \in S\setminus S_p$ be a prime. Let $I_\lambda\subset G_{F_\lambda}$ denote the inertia subgroup and ${\rm Fr}_\lambda\in G_{F_\lambda}/I_\lambda$ the arithmetic Frobenius. The definition of Selmer complexes will require the following versions of unramified local conditions:

\subsubsection{}
\label{subsubsec311_2025_12_13_1552}
Nekov\'a\v{r}'s unramified local conditions are determined in terms of the complex
\[
C_{\ur}^\bullet (G_{F_\lambda}, V)= \left [ V^{I_\lambda}\xrightarrow{\Fr_\lambda-1} V^{I_\lambda}\right ],
\]
where the terms are concentrated in degrees $0$ and $1$. When $\cO_X=E$, we similarly define $C_{\ur}^\bullet (G_{F_\lambda}, T)$. We let $\RG_{\ur}(F_\lambda, V)$ (resp. $\RG_{\ur}(F_\lambda, T)$ when $\cO_\cX=E$) denote  the 
corresponding complex in the derived category $\mathscr D_{\textrm{ft}}(\cO_{\cX})$ (resp. in $\mathscr D_{\textrm{ft}}(\cO_E)$ if $\cO_\cX=E$). These complexes are equipped with a natural map 
\[
g_\lambda\,:\, C_{\ur}^\bullet (G_{F_\lambda}, M) \lra C^\bullet (G_{F_\lambda}, M)\,, \qquad M=V,\, T\,.
\]
We then have
\[
\bR^0\mathbf \Gamma_{\ur}(F_\lambda, M)\simeq H^0(F_\lambda,M), \qquad \bR^1\mathbf \Gamma_{\ur}(F_\lambda, M)\simeq
\ker \left (H^1(F_\lambda, M) \rightarrow H^1(I_\lambda,M) \right ).
\]
On the level of cohomology, the map $g_\lambda$ induces the isomorphism $\bR^0\mathbf \Gamma_{\ur}(F_\lambda, M)\simeq H^0(F_\lambda,M)$ and the natural inclusion $\bR^1\mathbf \Gamma_{\ur}(F_\lambda, M) \hookrightarrow H^1(F_\lambda,M)$. We refer the reader to \cite[\S7]{nekovar06} and \cite[\S1.2.1.1]{BBMemoirs} for the key properties of this complex that defines the unramified local conditions.

\begin{remark}
\label{rem_2025_12_14_0934}
\normalfont
As a matter of fact, Nekov\'a\v{r} in \cite[\S7]{nekovar06} defines Greenberg local conditions for a Galois representation $T$ over a complete local Noetherian $\ZZ_p$-algebra.  \hfill $\triangle$
\end{remark}

\subsubsection{} When $\cX=\Spm (E)$ is a singleton, these are supplemented by the Bloch--Kato conditions:
\begin{equation*}
H^1_{\rm f}(F_\lambda,M):= 
\begin{cases} \bR^1\mathbf \Gamma_{\ur}(F_\lambda, V), & \qquad\textrm{if $M=V$},\\
\mathrm{im} \left (H^1_{\rm f}(F_\lambda, V) \lra H^1(F_\lambda,V/T)\right ),  &\qquad \textrm{if $M=V/T$},\\
\textrm{$\ker\left(H^1(F_\lambda, T) \lra \dfrac{H^1(F_\lambda,V)}{H^1_{\rm f}(F_\lambda,V)}\right)$},  &\qquad \textrm{if $M=T$}.
\end{cases}
\end{equation*}
We remark that $H^1_{\rm f}(F_\lambda,M)$ need not agree with $\bR^1\mathbf \Gamma_{\ur}(F_\lambda, M)$ for $M=T$, $V/T$: the difference is accounted by Tamagawa numbers, cf. \cite[Lemma I.3.5(iii)]{rubin00} and \cite[\S7.6.9]{nekovar06}.

\subsection{Greenberg local conditions}
\label{subsec_32_2026_06_03}
We introduce local conditions at primes $\fp\in S_p$ that we shall work with. We put $\Gamma_\p^\cyc=\Gal(F_\p(\mu_{p^\infty})/F_\p)$ and fix a topologial generator $\gamma_\p^\cyc\in \Gamma_\p^\cyc$.

\subsubsection{Triangulations} We let $C^\bullet_{\varphi,\gamma^\cyc_\fp}(\DdagrigX (V))$ denote the Fontaine--Herr complex of the $(\varphi,\Gamma_\p^\cyc)$-module $\DdagrigX (V_{\vert_{G_{F_\fp}}})$ (over the relative Robba ring $\cR_\cX$) associated to the $G_{F_\p}$-representation $V$, given as in \cite[\S2.4]{benoisheights}. We note for the convenience of the reader that $K$ in op. cit. coincides with our $F_\fp$ and Benois writes $\gamma_{K}\in \Gamma_{K}$ in place of our $\gamma^\cyc_\fp\in \Gamma^\cyc_\fp$. Recall from \cite[\S2.5]{benoisheights} that there exists a complex $K^\bullet_\fp(V)$ of $\cO_{\cX}$-modules 
together with quasi-isomorphisms
\begin{equation}
\label{eqn_Fontaine_Herr_enhanced_2025_12}
C^{\bullet}(G_{F_\fp},V) \xrightarrow{\xi_\fp} K^\bullet_\fp(V)
\xleftarrow{\alpha_\fp}  C^\bullet_{\varphi,\gamma^\cyc_\fp}(\DdagrigX (V))\,.
\end{equation}

Let $\bD_\fp \subset \DdagrigX (V_{\vert_{G_{F_\fp}}})$ be a saturated $(\varphi,\Gamma^\cyc_\fp)$-submodule.  Such a choice of $\bD_\fp$ is called a triangulation, and it defines the local condition
\[
g_\fp\,:\, C^\bullet_{\varphi ,\gamma^\cyc_\fp}(\bD_\fp) \stackrel{\,\,\iota_{\bD_\fp}\,\,}{\lra} 
C^\bullet_{\varphi,\gamma^\cyc_\fp}(\DdagrigX (V))
\xrightarrow{\alpha_\fp} K^\bullet_\fp(V).
\]
On the level of cohomology, $g_\fp$ induces natural maps
\[
H^i(F_\fp,\bD_\fp) \lra H^i(F_\fp,\DdagrigX (V_{\vert_{G_{F_\fp}}}))\simeq H^i(F_\fp,V)\,.
\]

\subsubsection{\'Etale  triangulations} An important special case is the scenario when the $(\varphi,\Gamma_\p^\cyc)$-module $\bD_\fp= \DdagrigX (Z_\fp)$ for some $G_{F_\fp}$-submodule $Z_\fp\subset V_{\vert_{G_{F_\fp}}}$. These then recover the classical Greenberg local conditions, which are defined by the natural map
\[
g_\fp\,:\, C^\bullet (G_{F_\fp},Z_\fp) \lra C^\bullet (G_{F_\fp},V)\,.
\]
On the level of cohomology, $g_\fp$ induces a natural morphism
\[
H^i(F_\fp,Z_\fp) \lra H^i(F_\fp, V)\,.
\]
When $\cO_\cX=E$ and $T$ is as in \S\ref{subsubsec311_2025_12_13_1552}, one may equally consider Greenberg local conditions on $T$ associated with a $G_{F_\fp}$-submodule $Z_\fp\subset T$. Moreover, in this situation (as in the case of unramified local conditions), we define, following Nekov\'a\v{r} (cf. \cite{nekovar06}, \S7.8), Greenberg local conditions for a Galois representation $T$ with coefficients in a complete local Noetherian $\ZZ_p$-algebra.

\subsubsection{Bloch--Kato conditions at $S_p$.} 
Suppose $\cX=\Spm (E)$. Following Bloch and Kato, we put
\begin{equation*}
H^1_{\rm f}(\Qp,M):= 
\begin{cases} \ker \left (H^1(F_\fp, V) \lra H^1(F_\fp,V\otimes_{\Qp}\mathbf{B}_{\textrm{cris}})\right ), &\quad\textrm{if $M=V$},\\
\mathrm{im} \left (H^1_{\rm f}(F_\fp, V) \lra H^1(F_\fp,V/T)\right ),  &\quad\textrm{if $M=V/T$},\\
\textrm{$\ker\left(H^1(F_\fp, T) \lra \dfrac{H^1(F_\fp,V)}{H^1_{\rm f}(F_\fp,V)}\right)$},  &\quad\textrm{if $M=T$}.
\end{cases}
\end{equation*}
Here, $\mathbf{B}_{\rm cris}$ is Fontaine's crystalline period ring; cf. \cite[\S3]{Fontaine82BasottiTate}, \cite[\S2.4]{Fontaine94Asterisque}, and \cite[\S3.7]{blochkato}.

\subsection{Selmer complexes}
\label{subsec_selmer_complexes_2025_12}
In this paper, we will work with Selmer complexes with coefficients in $V$, defined by both \'etale (in the $p$-ordinary setting) and non-\'etale (in the non-$p$-ordinary case) Greenberg local conditions. We revisit their general constructions following  \cite{benoisheights}. 

\subsubsection{} As above, let $V$ be a $G_{F_\fp}$-representation over $\cO_\cX$. We write $\RG (V)$ (resp., $\RG (F_\lambda,V)$) for the complexes of continuous cochains of $G_{F,S}$ (resp. of $G_{F_\lambda}$) with coefficients in $V$. We let 
\begin{equation*}
\res_S:=(\res_\lambda)_{\lambda\in S} \,:\,C^\bullet (G_{F,S},V) \xrightarrow{(\res_\lambda)_{\lambda\in S}} \bigoplus_{\lambda\in S} C^\bullet (G_{F_\lambda},V)
\end{equation*}
denote the restriction map. For any $\lambda\in S\setminus S_p$, let us put  $K^\bullet_\lambda (V)=C^\bullet (G_{F_\lambda},V)$ and define $K^\bullet (V):=\underset{\lambda\in S}\bigoplus K^\bullet_\lambda (V)$. For $\lambda\in S$, we let $f_\lambda$  denote the map 
\[
f_\lambda=\begin{cases}
\res_\lambda &\text{if $\lambda\in S\setminus S_p$},\\
\xi_\p\circ \res_\p &\text{if $\lambda=\p\mid p,$}
\end{cases}
\]
where $\xi_\p$ is the map that appeared in \eqref{eqn_Fontaine_Herr_enhanced_2025_12}. We then have the following morphism of complexes:
\[
f_S=(f_\lambda)_{\lambda\in S} \,:\,C^\bullet (G_{F,S},V) \lra K^\bullet (V)\,.
\]

\subsubsection{} 
\label{subsubsec_1223_12_11}
Given triangulations $\bD_\p\subset \DdagrigX (V_{\vert_{G_{F_\fp}}})$ (\'etale or non-\'etale) for each $\fp\in S_p$, we put $\bD=\{\bD_\p\}_{\p\in S_p}$. We associate to $\bD$ the  collection $\left\{U_\lambda^\bullet(V,\bD), g_\lambda\right\}_{\lambda\in S}$ of local conditions given as follows: 
\begin{equation*}
U_\lambda^\bullet (V,\bD):= 
\begin{cases}
C^{\bullet}_{\textrm{ur}}(G_{F_\lambda}, V), &\quad\lambda\in S\setminus S_p\,,\\
C^\bullet_{\varphi,\gamma^\cyc_\fp}(\bD_\fp), &\lambda=\p\in S_p.
\end{cases}
\end{equation*}
Let us put $U^\bullet (V,\bD):=\bigoplus_{\lambda\in S} U^\bullet_\lambda (V,\bD)$ and $g_S:\sum_{\lambda\in S}g_\lambda$. 

The Selmer complex associated to local conditions $\left(U^\bullet (V,\bD), g_S\right)$ is defined as the mapping cone
\begin{equation*}
S^\bullet(V,\bD):=\mathrm{cone}
\left [C^\bullet (G_{F,S},V) \oplus U^\bullet (V,\bD)
\xrightarrow{f_S-g_S} K^{\bullet}(V) \right ] [-1].
\end{equation*}
We denote by  $\RG (V,\bD)= [S^\bullet (V,\bD)]$  the corresponding object in the category of bounded perfect complexes of $\cO_{\cX}$-modules (see \cite[\S3]{benoisheights} for details). We write $H^i(V,\bD)$ for the cohomology of $S^\bullet (V,\bD)$. Note that $H^i(V,\bD)=0$  for $i\geqslant 4$. 

\subsubsection{} 
In the setting of \S\ref{subsubsec_1223_12_11}, we have a distinguished triangle
\begin{equation}
\nonumber
K^\bullet (V) \lra S^\bullet (V,\bD)[1] \lra 
\bigl ( C^\bullet (G_{F,S},V)\oplus U^\bullet (V,\bD)\bigr ) [1]\lra 
K^\bullet (V)[1]\,,
\end{equation}
which induces the long exact sequence (cf. \cite{nekovar06}, \S0.8.0)
\begin{align}
\label{eqn_2025_12_13_1653}
\begin{aligned}
 H^0(F,V) \lra &  \underset{\p\in S_p} \bigoplus  H^0(F_\fp, \bD_\fp^-) \lra H^1(V,\bD) \\
 &\lra H^1(G_{F,S},V)\lra \underset{\lambda\in S\setminus S_p} \bigoplus \frac{H^1(F_\lambda, V)}{H^1_{\rm f}(F_\lambda, V)}\oplus \underset{\fp\in S_p} \bigoplus H^1(F_\fp, \widetilde{\bD}_\p) \,,
\end{aligned}
\end{align}
where $\widetilde{\bD}_\p:=\DdagrigX (V_{\vert_{G_{F_\fp}}})\big{/}\bD_\p$. One utilises the exact sequence \eqref{eqn_2025_12_13_1653} to compare the cohomology of Selmer complexes to their classical variants, namely the Bloch--Kato Selmer groups.

%%%%%%%%%%%%%%%%%%%%%%%%%%
%%%%%%%%%%%%%%%%%%%%%%%%%%
%%%%%%%%%%%%%%%%%%%%%%%%%%
%%%%%%%%%%%%%%%%%%%%%%%%%%

\subsection{Height pairings and Rubin's formula}
\label{subsec_2026_01_07_1351}
In this subsection, we review the construction of $p$-adic height pairings. We refer the reader to \cite{nekovar06,benoisheights} (as well as \cite{BBMemoirs}, \S1.3) for further details. 
\subsubsection{} 
\label{subsubsec_341_2025_12_13}
Let $V$ and $V^*$ be a pair of $G_{F,S}$ Galois representations with coefficients in $\cO_{\cX}$ as above, equipped with an $\cO_{\cX}$-bilinear pairing 
\begin{equation}
\label{eqn_2025_12_13_1715} 
\left (\,,\,\right )\,\,:\,\,V^*\times V\lra \cO_{\cX}\,.
\end{equation}
 The pairing \eqref{eqn_2025_12_13_1715} induces an $\cO_{\cX}$-linear pairing
\begin{equation}
\label{eqn_2025_12_13_1716} 
(\,,\,) \,\,:\,\,\DdagrigX (V^*_{\vert_{G_{F_\p}}}) \times \DdagrigX (V_{\vert_{G_{F_\p}}}) \lra \cO_{\cX}\,, \qquad \fp\in S_p.
\end{equation}
We say that $(\varphi, \Gamma^\cyc_\p)$-submodules $\bD_\fp \subseteq \DdagrigX (V_{\vert_{G_{F_\p}}})$ and 
$\bD^*_\fp \subseteq  \DdagrigX (V^*_{\vert_{G_{F_\p}}})$ are orthogonal if 
\[
(x,y)=0, \qquad \forall\,\,  x\in \bD'_\p,\quad\forall\,\,  y\in \bD_\p\,,
\]
and that $\bD=\{\bD_\fp\}_{\p\in S_p}$ is orthogonal to $\bD^*=\{\bD^*_\fp\}_{\p\in S_p}$ if $\bD_\fp$ and $\bD_\fp^*$ are so for each $\p$.

\subsubsection{}
\label{subsubsec_342_2025_12_14}
Suppose that we are in the scenario of Section~\ref{subsubsec_341_2025_12_13}, where $\bD$ is orthogonal to $\bD^*$. We then have the $p$-adic height pairing 
\begin{equation}
\label{eqn_defn_height_cyclo_2025_12_13}
h^\cyc_{\bD^*,\bD}\,:\,\RG (V^*(1), \bD^*(1)) \otimes_{\cO_\cX}^{\mathbf L}\RG (V,\bD) \lra \cO_\cX [-2]
\end{equation}
given as in \cite[\S1.3.1.3]{BBMemoirs}.  For the convenience of the reader, we briefly recall the definition of these $p$-adic height pairings.  

As in \S\ref{subsec_Iwasawa_Algebras}, let $\gamma_\cyc$ denote a topological generator of $\Gamma_\cyc^\circ$, and let $J=(\gamma_\cyc-1)\subset \LL(\Gamma_\cyc)$ denote the augmentation ideal. Let us denote by $\LL^\sharp(\Gamma_\cyc)$ the free $\LL(\Gamma_\cyc)$-module of rank one on which $G_{F,S}$ acts through the projection $G_{F,S}\twoheadrightarrow \Gamma_\cyc$. We consider the cyclotomic deformation $\mathbb{V}:=V\,\widehat\otimes_{\ZZ_p}\,\LL^\sharp(\Gamma_\cyc)$ of $V$ which we endow with the diagonal $G_{F,S}$ action. 

Consider the exact sequence 
\begin{equation}\label{eqn_deform_Wk}
0\lra V \xrightarrow{\times_\gamma} V_{\epsilon} \xrightarrow{\,{\rm pr}_0\,} V\lra 0
\end{equation}
where $\times_\gamma$ is multiplication by $(\gamma_\cyc-1)/\log\chi_\cyc(\gamma_\cyc)$, $V_{\epsilon}:=\mathbb{V}/J^2 \mathbb{V}$, and ${\rm pr}_0$ is the projection morphism\,. The exact sequence \eqref{eqn_deform_Wk} is independent of the choice of $\gamma_\cyc$ thanks to the normalised definition of the multiplication-by-$(\gamma_\cyc-1)$ map. We similarly put 
$$\bD_\epsilon:=\left\{\bD_\p\otimes \LL(\Gamma_\cyc)/J^2=:\bD_\epsilon^{(\p)}\right\}_{\p\in S_p}$$ 
and define the analogous objects for $V^*$ used in place of $V$. The exact sequence \eqref{eqn_deform_Wk}, together with the analogous sequence involving $\bD_\epsilon$, gives rise to a distinguished  triangle of Selmer complexes
\begin{equation}
\label{eqn_bocstein_sequence_cyclo} 
\RG (V,\bD)\lra \RG (V_\epsilon,\bD_\epsilon) \lra \RG (V,\bD)\xrightarrow{\beta_{\bD}^\cyc} \RG (V,\bD) [1] \,,
\end{equation}
see \cite[\S3.2]{benoisheights} for a detailed exposition. The morphism $\beta_{\bD}^{\cyc}$ is called the Bockstein morphism for the infinitesimal cyclotomic variation.

The $p$-adic height pairing \eqref{eqn_defn_height_cyclo_2025_12_13} associated to the data $\{(V,\bD),(V^*,\bD_\epsilon^*)\}$ is defined as the morphism
\begin{equation}
\label{eqn_defn_height_cyclo}
\RG (V^*(1), \bD^*(1)) \otimes_{\cO_\cW}^{\mathbf L}\RG (V,\bD) \xrightarrow{\beta_{\bD^*}^\cyc \otimes \id} 
 \RG (V^*(1), \bD^*(1))[1] \otimes_{\cO_\cX}^{\mathbf L}\RG (V,\bD)
\xrightarrow{\cup_{\bD^*,\bD}} \cO_\cX [-2],
\end{equation}
where $\cup_{\bD^*,\bD}$ is the global duality pairing (cf. \cite[\S1.2.5.3]{BBMemoirs}).

On the level of cohomology, \eqref{eqn_defn_height_cyclo_2025_12_13} induces an $\cO_\cX$-linear pairing 
\begin{equation}
\label{eqn_2025_12_13_1730}
\left < \,,\,\right>^{\cyc}_{\bD^*,\bD}\,:\, H^1 (V^*(1), \bD^*(1)) \otimes H^1(V, \bD) \lra \cO_\cX\,.
\end{equation}
When $\bD$ and $\bD^*$ are understood, we shall write $\left < \,,\,\right>$ in place of $\left < \,,\,\right>^{\cyc}_{\bD^*,\bD}$.

\subsubsection{}
\label{subsubsec_343_2025_12_14}
Suppose that $R$ is a complete local Noetherian $\ZZ_p$-algebra and $T$ (resp. $T^*$) is a free $R$ module equipped with a continuous $G_{F,S}$-action, as well as a perfect $G_{F,S}$-equivariant pairing $T\otimes T^*\to R$. Suppose also that, for each $\fp\in S_p$, we are given $G_{F_\p}$-stable submodules $Z_\p\subset T_{\vert_{G_{F_\p}}}$ (resp. $Z_\p^*\subset T^*_{\vert_{G_{F_\p}}}$).  With this data at hand (and combined with the unramified local conditions; cf. Remark~\ref{rem_2025_12_14_0934}), we define, following Nekov\'a\v{r} \cite[\S6]{nekovar06}, Greenberg local conditions $U^\bullet(T,\bD)$ (resp. $U^\bullet(T^*,\bD^*)$). Here, $\bD=\{Z_\fp\}_{\fp\in S_p}$ is the (\'etale) triangulation of $\{T_{G_{F_{\p}}}\}_{\p \in S_p}$ (similarly $\bD^*$). These in turn determine Selmer complexes 
$$\RG (T,\bD), \quad \RG (T^*(1),\bD^*(1))\in \mathscr{D}_{\rm ft}^{[0,3]}(R)\,.$$
Let us assume in addition that $\bD$ and $\bD^*$ are orthogonal with respect to the local cup product pairing (cf. \cite{nekovar06}, \S5.2). In that case, similar to \S\ref{subsubsec_342_2025_12_14}, we have the $p$-adic height pairing
\begin{equation}
\label{eqn_defn_height_cyclo_2025_12_14_R}
h^\cyc_{\bD^*,\bD}\,:\,\RG (T^*(1), \bD^*(1)) \otimes_{R}^{\mathbf L}\RG (T,\bD) \lra R [-2]
\end{equation}
given as in \cite[\S11]{nekovar06}. On the level of cohomology, \eqref{eqn_defn_height_cyclo_2025_12_14_R} induces an $R$-linear pairing 
\begin{equation}
\label{eqn_2025_12_14_1030_R}
\left < \,,\,\right>^{\cyc}_{\bD^*,\bD}\,:\, H^1 (T^*(1), \bD^*(1)) \otimes H^1(T, \bD) \lra R\,.
\end{equation}
As above, when $\bD$ and $\bD^*$ are understood, we shall write $\left < \,,\,\right>$ in place of $\left < \,,\,\right>^{\cyc}_{\bD^*,\bD}$.

\begin{remark}
    \label{remark_2026_07_01_1017} 
\normalfont
    Let $f$ be the newform we have fixed at the start of the paper. We recall our running no-exceptional-zero hypothesis: $a_p(f)\neq p^{\frac{k}{2}}$ if $p\mid N_f$.  As explained in \cite[\S11.3]{nekovar06} (especially, Theorem 11.3.9 in op. cit.; see also \cite{benoisheights}, Theorem 5.2.1 and Corollary 6.3.4), the $p$-adic height pairings we have introduced in \S\ref{subsubsec_342_2025_12_14} and \S\ref{subsubsec_343_2025_12_14} above coincide (up to sign) with earlier constructions of $p$-adic heights, due to Schneider~\cite{schneider82} (using universal norms) and Zarhin~\cite{Zarhin} (using splittings of the Hodge filtration), later generalised by Perrin-Riou~\cite{PerrinRiou92} and Nekov\'a\v{r}~\cite[\S4 \& \S6]{nekovar93},  on the Bloch--Kato Selmer group associated to the newform $f$. When $a_p(f)= p^{\frac{k}{2}}$, they do \emph{not} coincide;  see \cite[Theorem 11.4.6]{nekovar06}. \hfill $\triangle$

\end{remark}

\subsubsection{Rubin's formula} 
One may compute the height pairings in terms of local duality pairings, which will be a key ingredient in our leading term formulae. We review treatment of \cite{BBMemoirs} (see also \cite{nekovar06, kbbMTT, benoisbuyukboduk}), which are the generalisations of the works of Rubin~\cite{Rubin92, Rubin_RubinsFormula}. 

Suppose that $(V,\bD)$ and $(V^*,\bD^*)$ are as in \S\ref{subsubsec_342_2025_12_14}. For $(W,D)=(V,\bD)$ or $(V_\epsilon, \bD_\epsilon)$, let us put
$$ \widetilde{U}^\bullet(W,D)=\bigoplus_{\lambda\in S} \widetilde{U}^\bullet_\lambda(W,D):={\rm cone}[U^\bullet(W,D)\xrightarrow{-g_S}K^\bullet(W)]\,.$$ 
We then have the following exact triangle in the derived category:
\begin{equation}
\label{eqn_selmersequence_2025_12_14}
 \widetilde{U}^\bullet(W,D)[-1]\xrightarrow{\,\partial\,}\RG (W,D)\lra 
 {\RG(W)} \stackrel{\widetilde{\res}_S}{\lra}  \widetilde{U}^\bullet(W,D)\,,
\end{equation}
where $\RG(W)$ (resp. $\RG (W,D)$) denotes the image of the complex $C^\bullet(G_{F,S},W)$ (resp. the complex $S^\bullet(W,D)$ given as in \S\ref{subsec_selmer_complexes_2025_12}) in the derived category, cf. \cite[Section~6.1.3]{nekovar06}.  Moreover, for $\lambda \in S$, we have the tautological exact triangle
\begin{equation}\label{eqn:localsequence}
{U}_\lambda^\bullet (W,D)\lra K_\lambda^\bullet(W)\lra \widetilde{U}^\bullet_\lambda(W,D)\lra {U}_\lambda^\bullet(W,D)[1]\,.
\end{equation}

Let us set $\widetilde{D}_\p:= \DdagrigE (Z_{\vert_{G_{F_\p}}})/D_\p$. Recall that  $\RG(F_\p,D_\p)$ (resp. $\RG(F_\p,\widetilde{D}_\p)$) denotes the class of the complex $C_{\varphi,\gamma^\cyc_\fp}^{\bullet}(D_\p)$ (resp., $C_{\varphi,\gamma^\cyc_\fp}^{\bullet}(\widetilde{D}_\p)$) in the corresponding derived category. We have a distinguished triangle
\begin{equation}
\label{eqn:localsequencetilde1}
\RG(F_\fp,D_\fp)\lra \RG(F_\fp,\DdagrigE (Z_{\vert_{G_{F_\fp}}}))\stackrel{\mathfrak{s}}{\lra} \RG(F_\fp,\widetilde{D}_\p)\lra\RG(F_\fp,D_\fp)[1]\,.
\end{equation}
The quasi-isomorphism $\alpha$ in \eqref{eqn_Fontaine_Herr_enhanced_2025_12} together with \eqref{eqn:localsequence} and \eqref{eqn:localsequencetilde1} induce a functorial quasi-isomorphism 
\begin{equation}
\label{eqn_identifyingsungularquotients_2025_12}
\RG(F_\fp,\widetilde{D}_\p)\stackrel{\rm qis}{\lra}  
\widetilde{U}_\p^{ \bullet}(Z_{\vert_{G_{F_\p}}}, D_\p)\,.\end{equation}

\begin{theorem}[Rubin's formula]
\label{thm_RSformula_cyclo_height_2025_12_14}
Suppose $[z_{\rm f}]=[(z,(z_\lambda^+)_{\lambda \in S},(\mu_\lambda)_{\lambda \in S})] \in H^1 (V^*(1), \bD^*(1))$ has the property that 
$$\exists\,[\widetilde{z}]\in H^1(G_{F,S},V^*_\epsilon(1))\,:\qquad  {\rm pr}_0\left([\widetilde{z}]\right)=[z] \in   H^1(G_{F,S}, V^*(1))\,$$
\item[i)] There exists a class $(``$Bockstein normalised cyclotomic derivative of $[\widetilde{z}]")$ 
$$[D_\cyc\widetilde{z}]=\left([D_\cyc\widetilde{z}]_{\lambda}\right)_{\lambda \in S}\,\in \,\bigoplus_{\lambda\in S} H^1(\widetilde{U}_\lambda^{ \bullet}(V^*(1), { \bD^*(1)})=: H^1(\widetilde{U}(V^*(1),  \bD^*(1)))$$
 such that  
 \begin{equation}
\label{eqn_2025_12_14_1203}     
\times_\gamma\,([D_\cyc\widetilde{z}])=\widetilde{\res}_S\left([\widetilde{z}]\right)\in H^1(\widetilde{U}(V_\epsilon^*(1),  \bD_\epsilon^*(1)))\,,\qquad \beta^1([z_{\rm f}])=-\partial([D_\cyc\widetilde{z}])\,,
 \end{equation}
where the maps $\partial$ and $\widetilde{\res}_S$ are as in \eqref{eqn_selmersequence_2025_12_14}, and 
$$H^1(V^*(1),D^*(1))\xrightarrow{\beta^1}H^2(V^*(1),D^*(1))$$ 
is the Bockstein morphism $($cf. \cite{BBMemoirs}, \S1.3.1.2$)$. If $H^0(\widetilde{U}(V^*(1),  \bD^*(1)))=\{0\}$, then $[D_\cyc\widetilde{z}]$ is uniquely determined by the first identity in \eqref{eqn_2025_12_14_1203}.
\item[ii)] Let $D_\cyc\widetilde{z}\in \widetilde{U}^1(V^*(1),\bD^*(1))$ be any cocycle representing $[D_\cyc\widetilde{z}]$ and let $[y_{\rm f}]=[(y,(y_\lambda^+)_{\lambda \in S},(\nu_\lambda)_{\lambda \in S})] \in  H^1(V, \bD)$ be any class. Then,
\begin{align}
\label{eqn_2026_05_26_1905}
\langle [z_{\rm f}], [y_{\rm f}]\rangle&=-\sum_{\lambda \in S}\textup{inv}_\lambda\left((D_\cyc\widetilde{z})_\lambda\cup g_\lambda(y_\lambda^+)\right)\,.
\end{align}
\end{theorem}

\begin{proof}
    When $H^0(\widetilde{U}(V^*(1),\bD^*(1)))=\{0\}$, this is \cite[Proposition 11.3.15]{nekovar06}; for the general case, both parts are \cite[Theorem 1.15]{BBMemoirs}. See also \cite[Proposition A.2]{kbbMTT} and \cite[Theorem 4.13]{benoisbuyukboduk} for similar statements at various levels of generality.
\end{proof}

\begin{corollary}
    \label{cor_thm_RSformula_cyclo_height_2025_12_14}
    In the situation of Theorem~\ref{thm_RSformula_cyclo_height_2025_12_14}, assume that $H^1_{\rm f}(F_\lambda,V_\kappa)=\{0\}$ for all $\lambda\not\in S_p$ and a dense set of specialisations $\kappa\in \cX$. Then,
$$\langle [z_{\rm f}], [y_{\rm f}]\rangle=-\left<\mathfrak{d}_\cyc\left[\widetilde{z}\right], \res_p([y])\right>_{\rm Tate}\,,$$ 
where $\left<\,,\,\right>_{\rm Tate}$ denotes the local cup-product $($Tate$)$ pairing, and 
\begin{equation*}
\mathfrak{d}_\cyc\left[\widetilde{z}\right]=\left(\mathfrak{d}_\cyc\left[\widetilde{z}\right]\right)_{\p\in S_p} \in \bigoplus_{\p\in S_p} H^1(F_\p, \widetilde{\bD}^*_\fp(1))
\end{equation*}
is the image of $([{D}_\cyc\widetilde{z}]_\p)_{\p\in S_p}$ under the isomorphism 
$$\bigoplus_{\fp \in S_p} H^1(\widetilde{U}_\p(V^*(1),\bD^*(1)))\xrightarrow{\eqref{eqn_identifyingsungularquotients_2025_12}} \bigoplus_{\fp \in S_p} H^1(F_\fp,\widetilde {\bD}_\p^*(1)).$$
\end{corollary}
\begin{proof}
    Since $H^1_{\rm f}(F_\lambda,V_\kappa)=\{0\}$, the specialisation of the sum over $\lambda \in S$ in \eqref{eqn_2026_05_26_1905} collapses to a sum over $\lambda \in S_p$. By the density of such $\kappa$, the same conclusion also holds over $\cX$, as claimed.
\end{proof}

\begin{remark}
    \label{rem_2025_12_14_1222}
    \normalfont
    When $\mathfrak{X}=\Spm(E)$, the condition that $H^1_{\rm f}(F_\lambda,V)=\{0\}$ for $\lambda\not\in S_p$ is equivalent to the vanishing of $H^0(F_\lambda,V)$. Indeed, this is immediate by the rank-nullity theorem and the definition of the unramified local condition. In the more general case, assuming that $H^0(F_\lambda,V_\kappa)=\{0\}$ for a dense set of specialisations $\kappa$, it follows that the terms in the sum in \eqref{eqn_2026_05_26_1905} corresponding to $\lambda\nmid p$ vanish because their specialisations to a dense set of points do so by the discussion above. In our main applications, the required vanishing statement at classical specialisations of Coleman families is verified using the local-global compatibility in the Langlands correspondence for $\GL_2$ (cf. \cite{Car86}). \hfill $\triangle$
\end{remark}

\begin{corollary}
    \label{cor_thm_RSformula_cyclo_height_2025_12_15}
    In the situation of Corollary~\ref{cor_thm_RSformula_cyclo_height_2025_12_14}, assume in addition that 
    $$\res_\fq [\widetilde z] \in {\rm im}\left(H^1(F_\fq,\bD^{(\fq)}_\epsilon)\to H^1(F_\fq,V_\epsilon) \right)\,, \quad \forall\, \fq \in S_{\rm cris}\subset S_p\,.$$
    Then,
$$\langle [z_{\rm f}], [y_{\rm f}]\rangle=-\sum_{\p\in S_p\setminus S_{\rm cris}}\left<\mathfrak{d}_\cyc^{(\p)}\left[\widetilde{z}\right], \res_\p([y])\right>_{\rm Tate}\,.$$ 
\end{corollary}

\begin{proof}
    When $S_{\rm cris}=\emptyset$, then this is a restatement of Corollary~\ref{cor_thm_RSformula_cyclo_height_2025_12_14}. The general case also follows from Corollary~\ref{cor_thm_RSformula_cyclo_height_2025_12_14}, on noting that 
    $\mathfrak{d}_\cyc^{(\fq)}\left[\widetilde{z}\right]=0$ whenever $\fq\in S_{\rm cris}$.
\end{proof}

%%%%%%%%%%%%%%%%%%%%%%%%%%
%%%%%%%%%%%%%%%%%%%%%%%%%%
%%%%%%%%%%%%%%%%%%%%%%%%%%
%%%%%%%%%%%%%%%%%%%%%%%%%%

\subsection{Examples: families of modular forms} 
We will be interested in the following triangulations of the Galois representations (cf. \S\ref{subsubsec_2025_12_14_1514} and \S\ref{subsubsec_224_2025_12}) attached to families of eigenforms.
\subsubsection{} 
\label{subsubsec_351_2025_12_14}
Let $\h$ be a Hida family as in \S\ref{subsec_Hida_families} for which the hypotheses \eqref{item_Irr} and \eqref{item_Dist} hold true. Thanks to Wiles (see also \cite{ohta00}), there exist a $G_{\Qp}$-stable free $\LL_\h$-direct summand 
$$\mathcal{F}^+V_{\h}\subset V_{\h}{}_{\vert_{G_{\Qp}}}$$ 
of rank one. Let us define the free $\LL_\h$-module $\mathcal{F}^-V_{\h}:=V_{\h}/\mathcal{F}^+V_{\h}$ of rank one. 

The $G_{\QQ_p}$-modules $\cF^\pm V_{\h}$ can be explicitly described as follows. Let $\widetilde\alpha_\h: G_{\QQ_p}\to \cR_\hh^\times$ be the unramified character given by $\widetilde\alpha_\hh(\Fr_p)=a_p(\hh)$, and put $\alpha_\h:=\bbchi_{\h}\, \chi_\cyc\, \widetilde\alpha_\hh^{-1}\,\varepsilon_\hh$. Then the $G_{\QQ_p}$-action on $\cF^+V_{\h}$ (resp. $ \cF^-V_{\h}$) is given by $\alpha_\h$ (resp. $\widetilde\alpha_\h$). 

\subsubsection{}
\label{subsubsec_352_2025_12_17}
Assume now that $\h$ is a Coleman family as in \S\ref{subsec_Coleman_families_revisited}. The main results of \cite{liu-CMH} equip one with a saturated triangulation
\begin{equation}
\label{eqn_filtration_2025_12_14_1551}
0\lra \DD^+_\h \lra \DdagrigX(V_\h) \lra \DD^-_\h\lra 0
\end{equation}
of $(\vp,\Gamma_\cyc)$-modules (where we are slightly abusing notation in identifying $\Gamma_\cyc$ with $\Gamma^\cyc_\p$ from \S\ref{subsec_32_2026_06_03}) over the relative Robba ring $\RR_{\cX}$, where $\cX=\Spm\,\sA(r_0)$ for sufficiently small $r_0$. The $(\vp,\Gamma_\cyc)$-modules $\DD^+_\h\simeq\cR_\cX(\delta_\h)$ and $\DD^-_\h\simeq\cR_\cX(\widetilde\delta_\h)$ of rank one (where $\delta_\h,\widetilde{\delta}_\h:\QQ_p^\times\to \cO_\cX^\times$ are characters) can be described explicitly as follows:
$$\delta(u)=u^{\kappa(X)-1} \,,\quad \delta(p)=p^{k+1}a_p(\h)^{-1}\varepsilon_\h(p)\,,\qquad\qquad \widetilde\delta_\h(\ZZ_p^\times)=\,1,\quad\widetilde\delta_\h(p)=a_p(\h)\,,$$
where $\kappa(X):=k+\dfrac{\log(1+X)}{\log(1+p)}\in \cO_\cX$ is the unique function with $X=(1+p)^{\kappa(X)-k}-1$

\subsubsection{}
\label{subsubsec_2025_12_17_1139}
Let $\f$ and $\g$ be a pair of Hida families. For each $?=\f,\g$, we assume that we have a $G_{\QQ_p}$-stable $\LL_?$-direct summand $\cF^+T_?\subset T_?$. We consider the Greenberg local conditions associated to the (\'etale) triangulation $\bD_{\rm bal}$ given by the submodule
$$\cF^+_{\rm bal}\, T_\f\,\widehat\otimes\,T_\g:= \cF^+ T_\f\,\widehat\otimes\,T_\g+ T_\f\,\widehat\otimes\, \cF^+ T_\g\,.$$
We also put 
$$\cF^+_{\rm bal}\, T_\f\,\widehat\otimes\,T_\g\,\widehat\otimes\,\LL^\sharp(\Gamma_\cyc):=\left(\cF^+_{\rm bal}\, T_\f\,\widehat\otimes\,T_\g \right)\widehat\otimes\,\LL^\sharp(\Gamma_\cyc)\,.$$
We then define the balanced Selmer groups
$$H^1_{{\rm Iw},{\rm bal}}\left(\mathbb{Q}(\mu_{p^\infty}), T_\Bf\,\widehat\otimes T_\g\right):=H^1(T_\Bf\,\widehat\otimes T_\g\,\widehat\otimes\,\LL^\sharp(\Gamma_\cyc),\bD_{\rm bal})\,.$$
We similarly define $H^1_{{\rm Iw},{\rm bal}}\left(\mathbb{Q}(\mu_{p^\infty}), T_f\,\widehat\otimes T_g\right)$ for specialisations $f\times g$ of the family $\f \times \g$ of Rankin--Selberg products. 

Let us now assume that $\f$ and $\g$ are a pair of Coleman families, which come equipped with triangulations as in \S\ref{subsubsec_352_2025_12_17}. We then consider the saturated $(\varphi,\Gamma_\cyc)$-submodule
$${\bD}_{\bal}:=\bD_\f\,\widehat\otimes\,\Ddagrigg (V_\g) + \Ddagrigf(V_\f)\,\widehat\otimes\,\bD_\g \,\subset\, \Ddagrigfg (V_\f\,\widehat{\otimes}\,V_\g)$$
of rank 3. With this data, we also define the Iwasawa theoretic balanced Selmer group
$$H^1_{{\rm Iw},{\rm bal}}\left(\mathbb{Q}(\mu_{p^\infty}), V_\Bf\,\widehat\otimes V_\g\right):=H^1(V_\Bf\,\widehat\otimes V_\g\,\widehat\otimes\,\cH^\sharp(\Gamma_\cyc),\bD_{\rm bal})\,.$$
We similarly define the balanced Selmer group for specialisations $f\times g$ of the family $\f \times \g$.

\subsubsection{}
\label{subsubsec_354_2026_02_06}
Suppose that $\f$ is as above (a primitive family of finite slope) such that $\varepsilon_\f=\mathds{1}$. The triangulations described in \S\ref{subsubsec_351_2025_12_14} and \S\ref{subsubsec_352_2025_12_17} together with the duality \eqref{eqn_2025_12_09_1114} and \eqref{eqn_2025_12_09_1603} give rise to a $p$-adic height pairing
\begin{equation}
    \label{eqn_2026_01_09_1351}
    \langle\,,\,\rangle\,: \, H^1_{\rm f}(G_{K,S}, V_\f^\dagger) \otimes H^1_{\rm f}(G_{K,S}, V_\f^\dagger) \lra \LL_\f\,,
\end{equation}
where $ H^1_{\rm f}(G_{K,S}, V_\f^\dagger):= H^1(V_\f^\dagger,\DD_\f^{\dagger+})$ is the cohomology of the Selmer complex associated with $V_\f^\dagger$ considered as a $G_K$-representation. It follows from our discussion in \S\ref{subsubsec_213_2025_12_09_1556} that this height pairing has the following interpolative properties. We describe these separately according to whether the specialisation is new or $p$-old. Let $\f_\kappa\in S_{w+2}(\Gamma_0(N_\f\, p))$ denote an $E$-valued specialisation of the family of $\f$. 

If $\f_\kappa$ is a newform, then it follows from \eqref{item_P1} that the following diagram commutes, where the $p$-adic height pairing $h_{\f_\kappa}$ on the second row is induced from the Poincar\'e duality pairing $\langle\,,\,\rangle_{Np^r}$ with $r=1$; cf. \eqref{eqn_2026_1_10_1453}.
\begin{equation}
    \label{eqn_2026_01_09_1406}
    \begin{aligned}
         \xymatrix{
    H^1_{\rm f}(G_{K,S}, V_\f^\dagger) \otimes H^1_{\rm f}(G_{K,S}, V_\f^\dagger) \ar[rr]^-{\langle\,,\,\rangle}\ar[d]_{\kappa}&& \LL_\f\ar[d]^{\alpha(\kappa)^{-1}\lambda_{\f_\kappa}^{-1}\,\cdot\, \kappa}\\
    H^1_{\rm f}(G_{K,S}, V_{\f_\kappa}^\dagger) \otimes H^1_{\rm f}(G_{K,S}, V_{\f_\kappa}^\dagger) \ar[rr]_-{h_{\f_\kappa}}&& E
    }
    \end{aligned}
\end{equation}

If, on the other hand, $\f_\kappa$ is $p$-old and $f=\f_\kappa^\circ \in S_{w+2}(\Gamma_0(N_\f))$ is the associated newform, then it follows from \eqref{item_P2} that the following diagram commutes, where $\langle\,,\,\rangle_\kappa:=\kappa\,\circ\,\langle\,,\,\rangle$ and the $p$-adic height pairing $h_{\f_\kappa^\circ}$ on the second row is induced from the Poincar\'e duality pairing $\langle\,,\,\rangle_{N}$; cf. \eqref{eqn_2026_1_10_1453}.

\begin{equation}
    \label{eqn_2026_01_09_1534}
    \begin{aligned}
         \xymatrix@C=0.1cm{
    H^1_{\rm f}(G_{K,S}, V_{\f_\kappa}^\dagger) &\otimes& H^1_{\rm f}(G_{K,S}, V_{\f_\kappa}^\dagger) \ar[rrrrrrr]^-{\langle\,,\,\rangle_\kappa} \ar[d]^-{{\rm pr}_*^{\alpha(\kappa)}} &&&&&&& E \ar@{=}[d] \\
    H^1_{\rm f}(G_{K,S}, V_{\f_\kappa^\circ}^\dagger)  \ar[u]^{U_p^{-1}W_{Np}\,\circ\, ({\rm pr}^{\alpha(\kappa)})^*}  &\otimes& H^1_{\rm f}(G_{K,S}, V_{\f_\kappa^\circ}^\dagger) \ar[rrrrrrr]_-{h_{\f_\kappa^\circ}}  &&&&&&&E\,. 
    }
    \end{aligned}
\end{equation}
\subsubsection{}  
In this subsection, we denote by $\Gamma$ any one of the groups $\{\Gamma_{\rm cyc} ,  \Gamma_{\rm ac}, \Gamma_{\rm \fp} \}$. Let us also denote by $\Delta$ any one of the groups $\{\Delta_{\rm cyc} ,  \Delta_{\rm ac}, \Delta_{\rm \fp} \}$, so that $\Gamma\simeq \Delta\times \Gamma^\circ$. Let $\mathfrak{m}$ denote the maximal ideal of $\Lambda(\Gamma^\circ)$, and for each positive integer $n$, let us set $\mathcal{H}_n(\Gamma^\circ):=({\Lambda(\Gamma^\circ)}[\frac{\mathfrak{m}^n}{p}]^\wedge)[\frac{1}{p}]$, where $(\,\cdot\,)^\wedge$ means $p$-adic completion. As explained in \cite{pottharst}, these are affinoid algebras over $\QQ_p$. Moreover, we have that 
$$    \mathcal{H}(\Gamma^\circ) \simeq \varprojlim_n \mathcal{H}_n(\Gamma^\circ)\,,\qquad \mathcal{H}(\Gamma) \simeq \QQ_p[\Delta] \otimes_{\QQ_p}\mathcal{H}(\Gamma^\circ) \,. $$
As before, we write $\mathcal{H}_n^\sharp(\Gamma)$ for $\mathcal{H}_n(\Gamma)$ viewed as a $G_{\QQ_p}$-module with tautological action through $G_{\QQ_p} \rightarrow \Gamma \hookrightarrow \mathcal{H}_n(\Gamma)^\times$, and similarly define $\mathcal{H}_n^\iota(\Gamma)$ when the $G_{\QQ_p}$-action factors through the involution $\gamma \mapsto \gamma^{-1}$ on $\Gamma$.
For $? \in \{\sharp, \iota\}$ and positive integers $n$, we may consider (by \cite{KPX}, Theorem 2.2.17) the  $(\varphi,\Gamma_{\rm cyc})$-modules 
$$\bD^\dagger_{{\rm rig}, \mathcal{H}_n(\Gamma^\circ)}(\mathcal{H}_n^?(\Gamma))=: \bD^\dagger_{{\rm rig}}(\mathcal{H}_n^?(\Gamma))\,, $$
and similarly with $\Gamma$ replaced by $\Gamma^\circ$. %Finally, let us set $\mathcal{H}(\Gamma^\circ)^\square\mathcal{H}_n(\Gamma^\circ)^\square^\square\mathcal{H}_n(\Gamma)^\square$

If $D$ is a $(\varphi, \Gamma_{\rm cyc})$--module over $\mathcal{R}_A$, where $A$ is an affinoid algebra over a finite extension $E/\QQ_p$, then we define the Iwasawa cohomology groups 
\begin{align}
\begin{aligned}
\label{eqn_2026_01_20_1037}
    H^1\left(\QQ_{p}, D \hatotimes_{\QQ_p} \bD^\dagger_{\rm rig}(\mathcal{H}^?(\Gamma^\bullet))\right) &:= \varprojlim_{n} \,  H^1\left(\QQ_{p}, D \hatotimes_{\QQ_p} \bD^\dagger_{\rm rig}(\mathcal{H}^?_n(\Gamma^\bullet))\right)\,,\qquad \bullet=\circ,\{\,\}\,. %\\
    %H^1\left(\QQ_{p}, D \hatotimes_{\QQ_p} \bD^\dagger_{\rm rig}(\mathcal{H}^?(\Gamma))\right) &:= \varprojlim_{n} \,  H^1\left(\QQ_{p}, D \hatotimes_{\QQ_p} \bD^\dagger_{\rm rig}(\mathcal{H}^?_n(\Gamma))\right)\,. 
\end{aligned}    
\end{align}
We remark that we need to appeal to the description above of the Iwasawa cohomology groups in terms of (families) of $(\varphi,\Gamma_\cyc)$-modules only when the family $\f$ has positive slope. Otherwise, one can directly work with a $G_{\QQ_p}$-representation $X$ with coefficients in a complete local Noetherian $\ZZ_p$-algebra $R$, and define its Iwasawa cohomology $H^1(\QQ_{p}, X \hatotimes_{\ZZ_p} \LL^?(\Gamma^\bullet))$ as usual.

\subsubsection{} By \cite[Theorem 4.4.8]{KPX}, we have that 
$$H^1_{\rm Iw}(\QQ_p(\mu_{p^\infty}),D):=D^{\psi=1} \simeq H^1\left(\QQ_p, D\hatotimes_{\QQ_p} \bD^\dagger_{\rm rig}(\mathcal{H}^\iota(\Gamma_{\rm cyc}))\right)\,.$$
Moreover, when the $(\varphi,\Gamma_\cyc)$-module $D=\bD^\dagger_{\rm rig}(X)$ is \'etale, then this is consistent with the usual definition of the Iwasawa cohomology:
$$H^1_{\rm Iw}(\QQ_p(\mu_{p^\infty}),\bD^\dagger_{\rm rig}(X))\simeq H^1(\QQ_p,X\hatotimes \mathcal{H}^\iota(\Gamma_{\rm cyc})) \simeq H^1_{\rm Iw}(\QQ_p(\mu_{p^\infty}),X)\otimes_{\LL(\Gamma_{\rm cyc})} \mathcal{H}(\Gamma_{\rm cyc})\,.$$
\subsubsection{}
\label{subsubsec_356_2026_1_19_1714}
Let $\tau \in \widehat{\Delta}_\p$ (cf. \S \ref{subsec_Iwasawa_Algebras}). To make use of results from \cite{LLZ2,BDV, burungale2024zetaelementsellipticcurves}, we will assume that we are in one of the situations below.%: if $\tau$ is not primitive, then the primitive character $\tau'$ that it comes from,  satisfies one of:
\begin{enumerate}
    \item[\mylabel{item_Reg'}{\textbf{Reg}$'$})] Either $\tau$ is primitive, or $\tau$ is not primitive and comes from some character $\tau'$ satisfying $\tau'(\p) \neq \tau'(\p^c)$.
    \item[\mylabel{item_Eis'}{\textbf{Eis}$'$})]  $\tau = \mathds{1}$ is the trivial character mod $\fp$. 
\end{enumerate}
For the rest of this section, let $\Bf$ denote a Coleman or Hida family satisfying the condition \eqref{item_Reg}. Let $\g_\tau$ denote the CM Hida family associated to the tautological character
\[\Theta^\tau_\fp\,:\, \AA_K^\times \twoheadrightarrow \Gamma^\circ_{\fp} \hookrightarrow \Lambda_{\tau}(\Gamma^\circ_{\fp})^\times, \]
where the first arrow is the Artin reciprocity map composed with the natural projection from $G_K^{\rm ab} \twoheadrightarrow \Gamma^\circ_{\fp}$. If \eqref{item_Eis'} holds, we write $\Bg$ in place of $\Bg^{\mathds{1}}$, and we note that $\Bg$ satisfies \eqref{item_Eis}. Setting $\Lambda_{\g_\tau}$ to be the branch algebra of $\g_\tau$, we have the natural identification of rings
\begin{align}
    \label{equation_CM_wt_space}&\Lambda_{\g_\tau} \simeq \Lambda_\tau(\Gamma^\circ_{\fp})\simeq \Lambda(\Gamma^\circ_\fp), & \Lambda_{\rm wt} \simeq \Lambda((\Gamma^\circ_{\fp})^{p^{h_p}}), 
\end{align}
(for more details on CM Hida families, we refer the reader to \cite{Hida_Elementary} \S $7.6$ and \cite{burungale2024zetaelementsellipticcurves} \S $4$). 

When \eqref{item_Eis'} holds, we will make use of the results in \cite{burungale2024zetaelementsellipticcurves} and thus we will need to consider the Galois module $$V_\g^1 := V_\g/ (V_\g)_{\rm tors},$$ where $V_\g$ denotes the $\g$-isotypic part of $V(U_\g)$ as defined in \S \ref{subsubsection_413}. On the other hand, when \eqref{item_Reg'} holds, it suffices for us just to consider $V_{\g_\tau}$. To ease notation, we will write $V_{\g_\tau}$ in situations in which we are speaking of both cases simultaneously and leave it implicit that when \eqref{item_Eis'} holds, we actually mean $V_\g^1$.

\subsubsection{} Let us write \begin{align}
\rho:G_{\QQ_p} \cong G_{K_\fp}. \label{isom_Qp=Kp} \end{align} for the isomorphism determined by our fixed embedding. By Lemma $4.4$ and Theorem $5.19$ of \cite{burungale2024zetaelementsellipticcurves} when \eqref{item_Eis'} holds (see also \cite[Proposition 4.1]{BDV} after possibly shrinking $U_\g$), and by \cite[Corollary 5.2.6]{LLZ2} when \eqref{item_Reg'} holds, we have isomorphisms of $\Lambda(\Gamma_\fp^\circ)[G_{\QQ}]$-modules \begin{equation}
   V_{\g_\tau} \cong \rm{Ind}^\QQ_K(\Lambda_\tau^\sharp(\Gamma^\circ_\fp)), \label{equation_Induced_Rep_isom}
\end{equation} together with isomorphisms of $\Lambda(\Gamma_\fp^\circ)[G_{\QQ_p}]$-modules: 
\begin{align}\label{equation_identifying_F+-_with_Lambda_p}
    \mathcal{F}^+V_{\g_\tau} \simeq \Lambda_\tau^\sharp(\Gamma^\circ_{\fp})\,, \qquad \mathcal{F}^-V_{\g_\tau} \simeq \Lambda_\tau^{\sharp}(\Gamma^\circ_{\fp})^c. 
\end{align}
Here, $G_{\Qp}$ acts on $\Lambda_\tau^{\sharp}(\Gamma^\circ_\fp)$ via the isomorphism $\rho$ composed with $\tau\cdot\Psi_\fp$, where \[\Psi_\fp:G_K\twoheadrightarrow\Gamma_\fp^\circ\hookrightarrow\Lambda(\Gamma_\fp^\circ)^\times\]
is the tautological character. On the other hand, $G_{\Qp}$ acts on $\Lambda_\tau^{\sharp}(\Gamma^\circ_{\fp})^c$ by composing $\rho$ with $(\tau\cdot\Psi_{\fp})^c$, where $(\tau \cdot \Psi_\fp)^c$ is obtained from $\tau\cdot \Psi_\fp$ by pre-composing with conjugation by (any lift of) $c$. Of course, these isomorphisms depend on a choice of $\Lambda(\Gamma_\fp^\circ)[G_{\QQ_p}]$-basis for $\mathcal{F}^\pm V_{\g_\tau}$, which we fix in the next subsection. 

\subsubsection{} \label{subsubsec_Fixing_Basis} 
We explain how to determine a basis for $\mathcal{F}^\pm V_{\g_\tau}$. We first note that \cite[Proposition 10.1.1]{KLZ2} gives us an isomorphism 
\begin{align}
    &\omega_{\g_\tau}^\vee: \bD\left(\mathcal{F}^+V_{{\g_\tau}}(1- \Bk' -\varepsilon_{{\g_\tau}}) \right) \xrightarrow{\,\sim\,} \Lambda_{{\g_\tau}}, \notag 
\end{align}
where $\bD(M) = \left(M \hatotimes_{\ZZ_p} W(\overline{\FF}_p)\right)^{G_{\QQ_p}}$
for an unramified $G_{\QQ_p}$-module $M$.
\begin{remark} \normalfont
We note that our $\Lambda_{{\g_\tau}}$ corresponds to a branch of a Hida family. Moreover, when $\tau=\mathds{1}$, we have that $\varepsilon_{\Bg}{}_{\vert_{G_{\QQ_p}}}$ is trivial because $p$ splits in $K$. Let $\Bh$ denote a cuspidal Hida family. Then \cite[Theorem 9.5.2]{KLZ2} shows that $\omega_\Bh^\vee$ is compatible with the Faltings--Tsuji comparison isomorphism for specialisations of $\Bh$ at arithmetic points.  \hfill $\triangle$
\end{remark}
We have a functorial isomorphism of pro-$p$ abelian groups 
\begin{equation}
    \bD(\mathcal{F}^+V_{\g_\tau}(-1-\Bk'-\varepsilon_{\g_\tau})) \simeq \mathcal{F}^+V_{\g_\tau} \label{equatiom_Fukaya_Kato}
\end{equation} 
by \cite[Proposition 1.7.6]{Fukaya-Kato_Conj_of_Sharifi}, and we let $\lambda_{\g_\tau} \in \mathcal{F}^+V_{\g_\tau}$ denote the basis element corresponding to $1 \in \Lambda(\Gamma_\fp^\circ)$ under $\omega^\vee_{\g_\tau}$ and \eqref{equatiom_Fukaya_Kato}. Since $\mathcal{F}^- V_{\g_\tau} = (\mathcal{F}^+V_{\g_\tau})^c $, we also obtain a basis element $\lambda_{\g_\tau}^c$ of $\mathcal{F}^-V_{\g_\tau}$ by conjugating. These bases then determine the isomorphisms of \eqref{equation_identifying_F+-_with_Lambda_p}.

\subsubsection{} \label{Identifying_Iwasawa_Cohomologies_via_Covering}Let  $\{\rm{Spm}(\mathscr{O}_{\g_{\tau, n}})\}_n$ denote the affinoid cover of the generic fiber of the formal scheme $\rm{Spf}(\Lambda_{\Bg_\tau})$, corresponding to the affinoid algebras $\{\mathcal{H}_n(\Gamma^\circ_\fp)\}_n$ from the previous subsection under \eqref{equation_CM_wt_space}. For $? =\pm$, we set $$\DD^?_{\g_{\tau, n}} := \bD^\dagger_{\rm rig, \mathscr{O}_{\g_\tau, n}}(\mathcal{F}^? V_{\g_{\tau, n}}),$$ where $V_{\g_{\tau,n}}$ denotes the big Galois representation attached to $\g_{\tau,n}$ over the affinoid $\rm{Spm}(\mathscr{O}_{\g_{\tau,n}})$. Also, for any $(\varphi, \Gamma_{\rm cyc})$-module $D$ we define 
$$H^1(\QQ_p, D \hatotimes_{\QQ_p} \DD^?_{\Bg_\tau}) := \varprojlim_n \, H^1(\QQ_p, D \hatotimes_{\QQ_p} \DD^?_{\Bg_\tau, n})\,.$$
The isomorphisms \eqref{equation_identifying_F+-_with_Lambda_p} (determined by the choice of the basis of $\mathcal{F}^\pm$ as in \S\ref{subsubsec_Fixing_Basis}) induce the following isomorphisms: 
\begin{align}
    H^1\left(\QQ_{p}, \DD^{\dagger-}_{\Bf}\hatotimes \DD^+_{\g_\tau} \hatotimes \bD^\dagger_{\rm rig}(\mathcal{H}^\iota(\Gamma^\circ_{\rm cyc}))\right) & \simeq H^1\left(K_\fp, \DD^{\dagger-}_\Bf \hatotimes \bD^\dagger_{\rm rig}(\mathcal{H}^\sharp_{\tau}(\Gamma^\circ_\fp))\hatotimes \bD^\dagger_{\rm rig}(\mathcal{H}^\iota(\Gamma^\circ_{\rm cyc})) \right) \,,\label{+=fp}\\
    H^1\left(\QQ_{p}, \DD^{\dagger+}_\Bf \hatotimes \DD^-_{\g_\tau} \hatotimes \bD^\dagger_{\rm rig}(\mathcal{H}^\iota(\Gamma^\circ_{\rm cyc})) \right) & \simeq H^1\left(K_{\fp}, \DD^{\dagger+}_\Bf \hatotimes\bD^\dagger_{\rm rig}(\mathcal{H}^{\sharp}_{\tau}(\Gamma^\circ_\fp)^c)\hatotimes \bD^\dagger_{\rm rig}(\mathcal{H}^\iota(\Gamma^\circ_{\rm cyc})) \right)\,, \label{-=fp-}\\
    H^1\left(\QQ_{p}, \DD^{\dagger-}_\Bf \hatotimes \DD^-_{\g_\tau} \hatotimes \bD^\dagger_{\rm rig}(\mathcal{H}^\iota(\Gamma^\circ_{\rm cyc}))\right) & \simeq H^1\left(K_{\fp}, \DD^{\dagger-}_\Bf \hatotimes \bD^\dagger_{\rm rig}(\mathcal{H}^{\sharp}_{\tau}(\Gamma^\circ_\fp)^c)\hatotimes \bD^\dagger_{\rm rig}(\mathcal{H}^\iota(\Gamma^\circ_{\rm cyc})) \right). \label{--=fp--}
\end{align}

\subsubsection{} When the slope of the family $\f$ is zero, then the isomorphisms \eqref{+=fp}--\eqref{--=fp--} can be simplified (avoiding the passage to $(\varphi,\Gamma_\cyc)$-modules) and stated with integral coefficients:
\begin{align}
\label{eqn_2026_01_20_1204}
\begin{aligned}
    H^1\left(\QQ_{p}, \mathcal{F}^-V_{\f}^\dagger\hatotimes \mathcal{F}^+V_{\g_\tau} \hatotimes \LL^\iota(\Gamma^\circ_{\rm cyc})\right) & \simeq H^1\left(K_\fp, \mathcal{F}^-V_{\f}^\dagger \hatotimes \LL^\sharp_{\tau}(\Gamma^\circ_\fp)\hatotimes \LL^\iota(\Gamma^\circ_{\rm cyc}) \right) \,,\\
    H^1\left(\QQ_{p}, \mathcal{F}^+V_{\f}^\dagger \hatotimes \mathcal{F}^-V_{\g_\tau} \hatotimes \LL^\iota(\Gamma^\circ_{\rm cyc}) \right) & \simeq H^1\left(K_{\fp}, \mathcal{F}^+V_{\f}^\dagger \hatotimes \LL^{\sharp}_{\tau}(\Gamma^\circ_\fp)^c\hatotimes \LL^\iota(\Gamma^\circ_{\rm cyc}) \right)\,, \\
    H^1\left(\QQ_{p}, \mathcal{F}^-V_{\f}^\dagger \hatotimes \mathcal{F}^-V_{\g_\tau} \hatotimes \LL^\iota(\Gamma^\circ_{\rm cyc})\right) & \simeq H^1\left(K_{\fp}, \mathcal{F}^-V_{\f}^\dagger \hatotimes \LL^{\sharp}_{\tau}(\Gamma^\circ_\fp)^c\hatotimes \LL^\iota(\Gamma^\circ_{\rm cyc}) \right)\,. 
\end{aligned}
\end{align}

\subsubsection{}\label{subsubsec_432_2026_01_19}
Let us denote by $\mathbf{k}$ the universal cyclotomic character on the weight space of $\f$, and $\mathbf{k}/2$ its unique (since $p>2$) square root. We similarly define $\mathbf{k}'$ for $\g_{\tau}$ for each $\tau$. Let us set
    \begin{align}
    \begin{aligned}
        &\DD^+(V_\Bf^\dagger):= \DD_\Bf^{\dagger+}(-1-\Bk/2)^{\Gamma_{\rm cyc}=1}, &&\DD^-(V_\Bf^\dagger):= \DD_\Bf^{\dagger-}(\Bk/2)^{\Gamma_{\rm cyc}=1}\,,  \\
         &\DD^+(V_{\g_\tau}):= \varprojlim_n \,\DD_{\g_{\tau, n}}^{+}(-1-\Bk')^{\Gamma_{\rm cyc}=1}, &&\DD^-(V_{\g_\tau}):= \varprojlim_n \, (\DD_{\g_{\tau,n}}^{-})^{\Gamma_{\rm cyc}=1}\,, \label{equation_D^pm_g}
    \end{aligned}
    \end{align}
    where we put a non-trivial $\Gamma_{\rm cyc}$-action on $\DD^+(V_\Bf^\dagger)$, $\DD^-(V_\Bf^\dagger)$, and $\DD^+(V_{\g_\tau})$ via the respective characters $1+\mathbf{k}/2$, $-\mathbf{k}/2$, and $ 1+\mathbf{k}'$. We also set
    \begin{align}
        \DD^{+-}(V^\dagger_\Bf \hatotimes V_{\g_\tau}) &:= \DD^{+}(V^\dagger_\Bf) \hatotimes \DD^-(V_{\g_\tau}) \simeq \DD^{+}(V^\dagger_\Bf) \hatotimes \mathcal{H}(\Gamma_\fp^\circ), \notag \\
        \DD^{-+}(V^\dagger_\Bf \hatotimes V_{\g_\tau}) &:= \DD^{-}(V^\dagger_\Bf) \hatotimes \DD^+(V_{\g_\tau}) \simeq \DD^{-}(V^\dagger_\Bf) \hatotimes \mathcal{H}(\Gamma^\circ_\fp), \notag
    \end{align}
    where our choice of bases determines the final isomorphism in each case for $\mathcal{F}^\pm V_{\g_\tau}$.

\subsubsection{} For $\invques? \in \{+-, -+\}$, let 
\begin{equation}
    \mathcal{L}^{\invques ?}_n:H^1\left(\QQ_{p}, \DD_\Bf^{\invques,\dagger}\hatotimes \DD^?_{\g_{\tau, n}} \hatotimes \bD^\dagger_{\rm rig}(\mathcal{H}^\iota(\Gamma^\circ_{\rm cyc})) \right)\lra \DD^{\invques?}(V^\dagger_\Bf \hatotimes V_{\g_{\tau, n}}) \hatotimes \mathcal{H}(\Gamma^\circ_{\rm cyc}) \label{equation_Logarithm_Nakamura}
\end{equation}
denote the big Perrin--Riou logarithm induced by restricting the map from \cite[Theorem 7.1.4]{LZ1}, which itself was constructed for $(\varphi,\Gamma_\cyc)$-modules of rank one over the relative Robba ring $\cR_{\cA}$ (where $\cA$ denotes a reduced affinoid algebra) by extending the work of Nakamura~\cite{nakamura2014Jussieu} as in \cite[Equation 6.2.1]{LZ1}. We remark that in the statement of the theorem in \cite{LZ1}, $\QQ_{p,\infty} = \QQ_p(\mu_{p^\infty})$, whereas in the present paper $\QQ_{p,\infty} \subset \QQ_{p}(\mu_{p^\infty})$ denotes the $\ZZ_p$-subextension of $\QQ_{p,\infty}/\QQ_p$. As stated in Theorem 7.1.4 of op. cit., these maps are injective, and on passing to inverse limits and setting 
$$\DD^{\invques?}(V^\dagger_\Bf \hatotimes V_{\g_\tau}) := \varprojlim_n \DD^{\invques?}(V^\dagger_\Bf \hatotimes V_{\g_{\tau,n}})\,,$$ 
we obtain an injective map 
$$\mathcal{L}^{\invques?}: H^1\left(\QQ_{p}, \DD_\Bf^{\invques,\dagger}\hatotimes \DD^{?}_{\g_\tau} \hatotimes \bD^\dagger_{\rm rig}(\mathcal{H}^\iota(\Gamma^\circ_{\rm cyc})) \right)\lra \DD^{\invques?}(V^\dagger_\Bf \hatotimes V_{\g_\tau}) \hatotimes \mathcal{H}(\Gamma^\circ_{\rm cyc})\,. $$

\subsubsection{}
We recall Perrin-Riou and Nakamura's big logarithm maps relevant to our work.
\begin{defn} \label{defn_2_var_log} 
We let 
\begin{align}
    &\mathfrak{Log}^\Gamma_{-+}:H^1\left(K_\fp, \DD^{\dagger-}_\Bf \hatotimes \bD^\dagger_{\rm rig}(\mathcal{H}^\sharp_{\tau}(\Gamma^\circ_\fp))\hatotimes \bD^\dagger_{\rm rig}(\mathcal{H}^\iota(\Gamma^\circ_{\rm cyc})) \right)  \hookrightarrow \DD^{-}(V^\dagger_\Bf) \hatotimes \mathcal{H}(\Gamma^\circ_\fp) \hatotimes \mathcal{H}(\Gamma^\circ_{\rm cyc}), \label{equation_-+_Logarithm}\\
    &\mathfrak{Log}^\Gamma_{+-}:H^1\left(K_{\fp}, \DD^{\dagger+}_\Bf \hatotimes\bD^\dagger_{\rm rig}(\mathcal{H}^{\sharp}_{\tau}(\Gamma^\circ_\fp)^c)\hatotimes \bD^\dagger_{\rm rig}(\mathcal{H}^\iota(\Gamma^\circ_{\rm cyc})) \right)  \hookrightarrow \DD^+(V^\dagger_\Bf) \hatotimes \mathcal{H}(\Gamma^\circ_\fp) \hatotimes \mathcal{H}(\Gamma^\circ_{\rm cyc}), \label{equation_+-_Logarithm}
\end{align}  
denote the big logarithm maps obtained from \eqref{equation_Logarithm_Nakamura} via the isomorphisms \eqref{+=fp} and \eqref{-=fp-}. 
\end{defn}
We recall the decompositions \eqref{equation_Lambda(Gamma)_decomposition}, and let  $e_{\tau_\p}$ be the idempotent associated to $\tau_\p$ (where we now write $\tau_\p$ instead of $\tau$ to emphasise that it is an element of $\widehat{\Delta}_\p$). Let $e_{\tau_{\rm ac}}$ denote the idempotent of the character $\tau_{\ac}$ given as in Definition~\ref{definition_fp_to_ac}. We have the following short exact sequences:
\begin{align}
    &0 \lra  \mathcal{H}_{\tau_\p}^{\sharp}(\Gamma^\circ_\fp) \hatotimes \mathcal{H}^\iota(\Gamma^\circ_{\rm cyc})\xrightarrow{\,\times (\gamma_+ -1)\,} \mathcal{H}_{\tau_\p}^{\sharp}(\Gamma^\circ_\fp) \hatotimes \mathcal{H}^\iota(\Gamma^\circ_{\rm cyc}) \lra \mathcal{H}_{\tau_{\ac}}^{\iota} (\Gamma_{\rm ac}^\circ ) \lra 0, \label{eq_ses_gamma_gamma_gammaac_first} \\ 
    &0 \lra \mathcal{H}_{\tau_\p}^{\sharp}(\Gamma^\circ_\fp)^c \hatotimes \mathcal{H}^\iota(\Gamma^\circ_{\rm cyc})\xrightarrow{\,\times (\gamma_+ -1)\,}  \mathcal{H}_{\tau_\p}^{\sharp}(\Gamma^\circ_\fp)^c \hatotimes \mathcal{H}^\iota(\Gamma^\circ_{\rm cyc}) \lra \mathcal{H}_{\tau^{-1}_{\ac}}^{\iota} (\Gamma_{\rm ac}^\circ ) \lra 0, \label{eq_ses_gamma_gamma_gammaac}
\end{align}
Here, the third arrow in \eqref{eq_ses_gamma_gamma_gammaac_first} is given by 
 $$\mathcal{H}_{\tau_\p}^{\sharp}(\Gamma^\circ_\fp) \hatotimes \mathcal{H}^\iota(\Gamma^\circ_{\rm cyc}) \xrightarrow{\gamma_{\fp}\cdot e_{\tau_\p} \otimes 1 \mapsto \gamma^{-1}_{\fp}\cdot e_{\tau_\p} \otimes 1}  \mathcal{H}_{\tau_\p}^{\iota} (\Gamma^\circ_\fp) \hatotimes \mathcal{H}^\iota(\Gamma^\circ_{\rm cyc}) \xrightarrow{a \cdot e_{\tau_\p} \otimes b \mapsto(a \otimes b \, \rm{mod} \, (\gamma^+-1))\cdot e_{\tau_\ac}} \mathcal{H}_{\tau_\ac}^\iota(\Gamma^\circ_{\rm ac})\,,$$
whereas in \eqref{eq_ses_gamma_gamma_gammaac} by
   $$\mathcal{H}_{\tau_\p}^{\sharp}(\Gamma^\circ_\fp)^c \hatotimes \mathcal{H}^\iota(\Gamma^\circ_{\rm cyc}) \xrightarrow{\gamma_{\fp^c}\cdot e_{\tau_{\p}^c}\otimes 1 \mapsto \gamma^{-1}_{\fp^c}\cdot e_{\tau_{\p}^c}\otimes 1} \mathcal{H}_{\tau_\p}^{\iota} (\Gamma^\circ_\fp)^c \hatotimes \mathcal{H}^\iota(\Gamma^\circ_{\rm cyc}) \xrightarrow{a \cdot e_{\tau_\p^c} \otimes b \mapsto (a \otimes b \, \rm{mod} \, (\gamma^+-1))\cdot e_{\tau^{-1}_\ac}} \mathcal{H}_{\tau^{-1}_{\ac}}^\iota(\Gamma^\circ_{\rm ac})\,. $$
To ease notation, we henceforth write $\mathcal{H}_\tau(\Gamma^\circ_?)$ for $? \in \{ \p, \, {\ac}\}$, dropping the subscript on $\tau$ and leaving it implicit by the fact that it appears alongside $\Gamma_?^\circ$.
\begin{defn}
\label{def_2026_06_03_1051}
    We define the restriction of the logarithm maps to the anticyclotomic tower on setting
    \begin{align}
        &\mathfrak{Log}^{\rm ac}_{-+}\,:\, H^1\left(K_\fp, \DD^{\dagger-}_\Bf \hatotimes \bD^\dagger_{\rm rig}(\mathcal{H}^\sharp_{\tau}(\Gamma^\circ_\fp))\hatotimes \bD^\dagger_{\rm rig}(\mathcal{H}^\iota(\Gamma^\circ_{\rm cyc})) \right)/(\gamma_+-1) \hookrightarrow \DD^{\dagger-}_\Bf \hatotimes \mathcal{H}(\Gamma^\circ_{\rm ac}), \notag\\
        &\mathfrak{Log}^{\rm ac}_{+-}\,:\, H^1\left(K_{\fp}, \DD^{\dagger+}_\Bf \hatotimes\bD^\dagger_{\rm rig}(\mathcal{H}^{\sharp}_{\tau}(\Gamma^\circ_\fp)^c)\hatotimes \bD^\dagger_{\rm rig}(\mathcal{H}^\iota(\Gamma^\circ_{\rm cyc})) \right)/(\gamma_+-1) \hookrightarrow \DD^{\dagger+}_\Bf \hatotimes \mathcal{H}(\Gamma^\circ_{\rm ac}) \notag
    \end{align}
    to be the maps obtained by considering \eqref{equation_-+_Logarithm} and \eqref{equation_+-_Logarithm} modulo $\gamma_{+}-1$.
\end{defn}

\begin{remark} \label{remark_embedding_into_ac}
\normalfont
    From the long exact sequences induced by \eqref{eq_ses_gamma_gamma_gammaac_first} and \eqref{eq_ses_gamma_gamma_gammaac}, we have embeddings 
    \begin{align}
    \begin{aligned}
        &H^1\left(K_\fp, \DD^{\dagger-}_\Bf \hatotimes \bD^\dagger_{\rm rig}(\mathcal{H}^\sharp_{\tau}(\Gamma^\circ_\fp))\hatotimes \bD^\dagger_{\rm rig}(\mathcal{H}^\iota(\Gamma^\circ_{\rm cyc})) \right)/(\gamma_+-1) \hookrightarrow   H^1\left(K_{\fp}, \DD^{\dagger-}_\Bf \hatotimes \bD^{\dagger}_{\rm rig}(\mathcal{H}^\iota_\tau(\Gamma^\circ_{{\rm ac}}))\right) \notag\\
        &H^1\left(K_\fp, \DD^{\dagger+}_\Bf \hatotimes \bD^\dagger_{\rm rig}(\mathcal{H}^\sharp_{\tau}(\Gamma^\circ_\fp)^c)\hatotimes \bD^\dagger_{\rm rig}(\mathcal{H}^\iota(\Gamma^\circ_{\rm cyc})) \right)/(\gamma_+-1) \hookrightarrow   H^1\left(K_{\fp}, \DD^{\dagger+}_\Bf \hatotimes \bD^\dagger_{\rm rig}(\mathcal{H}^\iota_{\tau^{-1}}(\Gamma^\circ_{{\rm ac}}))\right). \label{embedding}
    \end{aligned}
    \end{align}
   Hence, we may view $\mathfrak{Log}^{\rm ac}_{\invques ?}$ as maps on these submodules of $H^1_{\rm{Iw}}\left(K_{\fp}, \DD^{\invques,\dagger}_\Bf \hatotimes \bD^\dagger_{\rm rig}(\mathcal{H}^\iota_\tau(\Gamma^\circ_{{\rm ac}}))\right)$ for $\invques= \pm$ and $\tau \in \widehat{\Delta}_\ac$. Moreover, by comparing interpolation properties, we see that $\mathfrak{Log^{\ac}_{+-}}$ readily agrees with the anticyclotomic logarithm map constructed in \cite{JLZ}, and hence, we consider it as a map on all of the module $H^1\left(K_{\fp}, \DD^{\dagger+}_\Bf \hatotimes \bD^\dagger_{\rm rig}(\mathcal{H}_\tau^\iota(\Gamma^\circ_{\rm ac})) \right)$.  \hfill $\triangle$
\end{remark}

\subsubsection{}
As a consequence of the overconvergent Eichler--Shimura isomorphism, we are equipped with canonical elements 
$$\omega_{\Bg_\tau}:= (\omega_{\g_{\tau,n}})_n \in \DD^{-}_{\Bg_\tau^c}(1+\mathbf{k}'-\varepsilon_{\Bg_\tau^c})^{\Gamma_{\rm cyc}=1} \quad \hbox{ and }\quad  \eta_{\Bg_\tau}:= (\eta_{\g_{\tau,n}})_n \in (\DD^+_{\Bg_\tau^c})^{\Gamma_{\rm cyc}=1}\,,$$ 
where the definition of these modules is similar to that of the modules $\DD^\pm(V_{\g_\tau})$ in \eqref{equation_D^pm_g}. We also obtain elements $ \omega_\Bf \in \DD^{\dagger-}_{\Bf^c}(\Bk/2 -\varepsilon_{\Bf^c}) ^{\Gamma_{\rm cyc}=1}$ and $\eta_\Bf \in \DD^{\dagger+}_{\Bf^c}(-1-\Bk/2)^{\Gamma_{\rm cyc}=1}$.
\begin{remark}
\normalfont
     Under the pairing 
    \[\langle\,,\,\rangle_\DD\,:\, \DD^+_{\g_{\tau,n}}(1-\Bk' - \varepsilon_{\g_{\tau,n}})^{\Gamma_{\rm cyc}=1} \times \DD^-_{\Bg^c_{\tau,n}}(1+\Bk' -\varepsilon_{\Bg^c_{\tau,n}})^{\Gamma_{\rm cyc}=1}\lra \mathscr{O}_{\g_{\tau,n}}\simeq \mathcal{H}_n(\Gamma^\circ_\fp) \]
    induced by Poincar\'e duality, $\langle \,\cdot\,, \omega_{\g_{\tau,n}}\rangle$ agrees with the map induced by $\omega_{\g_\tau}^\vee$.  \hfill $\triangle$
\end{remark}

\begin{defn} \label{definition_L_omega_and_L_eta}
    \item[i)] We define $\mathscr{L}^{\tau}_{\eta_\Bf}$ as the morphism
$$ \mathscr{L}^{\tau}_{\eta_\Bf} :   H^1\left(K_\fp, \DD^{\dagger-}_\Bf \hatotimes \bD^\dagger_{\rm rig}(\mathcal{H}^\sharp_{\tau}(\Gamma^\circ_\fp))\hatotimes \bD^\dagger_{\rm rig}(\mathcal{H}^\iota(\Gamma^\circ_{\rm cyc})) \right)/(\gamma_+-1) \xrightarrow{\langle\mathfrak{Log}^{\rm ac}_{-+} , \, \eta_\Bf \rangle_\DD} \Lambda_\Bf \hatotimes \mathcal{H}(\Gamma^\circ_{\rm ac})\,. $$
    \item[ii)] We put $\mathscr{L}^{\tau^c}_{\omega_\Bf}$ to denote the morphism 
   $$\mathscr{L}^{\tau^c}_{\omega_\Bf} :  H^1\left(K_{\fp}, \DD^{\dagger+}_\Bf \hatotimes \bD^\dagger_{\rm rig}(\mathcal{H}_{\tau^{-1}}^\iota(\Gamma^\circ_{\rm ac})) \right) \xrightarrow{\langle\mathfrak{Log}^{\rm ac}_{+-} , \, \omega_\Bf \rangle_\DD} \Lambda_\Bf \hatotimes \mathcal{H}(\Gamma^\circ_{\rm ac})\,.$$
\end{defn}
\section{A review of Beilinson--Flach elements}
Let $\Bf = \sum_{n\geq 1}a_n(\Bf)\, q^n$ (resp. $\Bg$) be a primitive Coleman family of tame level $N_{\Bf}$ (resp. $N_{\Bg}$) and tame nebentype characters $\varepsilon_\Bf$ (resp. $ \varepsilon_\Bg)$, cf. \S\ref{subsec_Coleman_families_revisited}.  We recall that these are parameterised by the connected affinoid discs $U_\Bf$ and $U_\Bg$, centred at classical weights $k+2$ and $k'+2$ in weight space, so that $a_n(\Bf)\in \LL_\f:=\cO(U_\f)$ (and similarly, $a_n(\Bg)\in \LL_\g$). Let $f$ and $g$ be newforms of respective weights and levels $(k+2, N_{\Bf} p^{r})$ and $(k'+2,N_{\Bg} p^{s})$ such that $r,s\in \{0,1\}$ and (a $p$-stabilisation of) $f$ and $g$ arise as an $E$-valued specialisation of $\Bf$ and $\Bg$, respectively (cf. \S \ref{subsubsection_414}). 
Shrinking $U_\Bf$ as necessary, we assume that $v_p(a_p(\Bf))=: \mathfrak{s}_\Bf\geq 0$ is constant on $U_\Bf$ (similarly for the family $\g$). 

In this section, we shall recall the ``big'' Beilinson--Flach element $_c{\rm{\bf BF}}(\Bf \otimes \Bg)$ and its interpolative properties; our exposition closely follows the proof of \cite[Proposition 2.3]{BDV}.

\subsection{} Let $j \in [0,\min(k,k')]$ be an integer. We consider the Rankin-Eisenstein class
\[{\rm Eis}_{{\textup{\'{e}t}},b,N}^{[k,k',j]} \in H^3_{{\textup{\'{e}t}}}\left(Y_1(N)^2, \mathscr{L}_{k}\boxtimes \mathscr{L}_{k'}(2-j)\right),\] 
defined in \cite[\S3.3.1]{KLZ2}, where the modular curves $Y(M,N)$ (for positive integers $M$ and $N$ where $M\mid N$) are as in op. cit, and finally, $\mathscr{L}_{w} = \rm{TSym}^w \mathscr{H}_{\QQ_p}$, where $\mathscr{H}_{\QQ_p}$ is the \'etale sheaf defined in \cite[\S2.3]{KLZ2}. 
\begin{defn}
    Let $r$ be a non-negative integer. We denote by
    \[{\rm Eis}(k,k',j)\in H^3_{{\textup{\'{e}t}}}\left(Y(p^r,Np^{r+1})^2, (\mathscr{L}_{k}\boxtimes \mathscr{L}_{k'}(2-j)\right)\] 
    the pull-back of ${\rm Eis}_{{\textup{\'{e}t}},1,Np^{r+1}}^{[k,k',j]}$ induced by the natural degeneracy map $Y(p^r,Np^{r+1})\rightarrow Y(Np^{r+1})$. 
\end{defn}
\subsubsection{} \label{subsubsection_412} 
Let $\Bh \in \{\Bf, \Bg\}$ and $w \geq \mathfrak{s}_{\h}$. As in the proof of \cite[Proposition 2.3]{BDV}, we let 
\[{V}(w)^{\leq \mathfrak{s}_{\h}}:= H^1_{\rm{\textup{\'{e}t}, par}}(Y_{\Bh}, \mathscr{L}_{w})^{\leq \mathfrak{s}_{\h}}\otimes_{\ZZ_p} E(1)\,,\] 
where $Y_{\Bh}=Y_1(N_{\Bh}p)_{\overline{\QQ}}$. The superscript $(\cdot)^{\leq \mathfrak{s}_{\h}}$ means the subspace of $(\cdot)$ on which the dual Hecke operator $U_p'$ acts with slope $\leq \mathfrak{s}_{\h}$. We also define 
\[\widetilde{V}(w)^{\leq \mathfrak{s}_{\h}}:= H^1_{\rm{\textup{\'{e}t}}}(Y_{\Bh}, \mathscr{L}_{w})^{\leq \mathfrak{s}_{\h}} \otimes_{\ZZ_p} E(1)\,.\] 
We define $\widetilde{\rm BF}_r(k,k',j)$ as be the image of ${\rm Eis}(k,k',j)$ under the composition of following morphisms: 
\begin{align*}
    &H^3_{{\textup{\'{e}t}}}\left(Y(p^r,Np^{r+1})^2, \mathscr{L}_{k}\boxtimes \mathscr{L}_{k'}(2-j)\right)  \\
    &\quad \xrightarrow{(t_r\times t_r)_*}H^3_{{\textup{\'{e}t}}}\left((Y_1(Np)\otimes\mathbb{Z}[\mu_{p^r}])^2, \mathscr{L}_{k}\boxtimes \mathscr{L}_{k'}(2-j)\right) \\
     &\quad \quad \longrightarrow H^1_{\textup{\'{e}t}}\left(\mathbb{Z}[\mu_{p^r}, \frac{1}{Np}], H^2_{\textup{\'{e}t}}\left(Y_1(Np)^2_{\bar{\mathbb{Q}}}, \mathscr{L}_{k}\boxtimes \mathscr{L}_{k'}(2-j) \right)\right) \\
     &\quad\quad\quad \longrightarrow H^1_{\textup{\'{e}t}}\left(\mathbb{Z}[\mu_{p^r}, \frac{1}{Np}], \widetilde{V}(k)^{\leq\mathfrak{s}_{\Bf}}\otimes \widetilde{V}(k')^{\leq \mathfrak{s}_{\Bg}}(-j) \right)\,. 
\end{align*}
Here, $N={\rm lcm}(N_\f,N_\g)$ and $Y(p^r,Np^{r+1})\xrightarrow{t_r} Y_1(Np)\otimes_\mathbb{Z}\mathbb{Z}[\mu_{p^r}]$ is the map denoted by $s_M$ in \cite[\S6.1]{KLZ2}. Furthermore, the second arrow is given by the Hochschild-Serre spectral sequence, and the last arrow is the pushforward along the natural projection $Y_1(Np)^2\rightarrow Y_1(N_\f p)\times Y_1(N_\g p)$ combined with the K{\"u}nneth formula.
\begin{defn}\label{defn_BFtilde_k_kp_j}
For an integer $c\geq 2$ coprime to $6Np$, we define 
\[_c\widetilde{{\rm BF}}_r(k,k',j) := ((c^2-c^{2j-k-k'+2}) \ \langle c \rangle_f\otimes \langle c \rangle_g)\cdot\widetilde{{\rm BF}}_r(k,k',j)\]
where $\langle c\rangle_*$ is the diamond operator acting on $\widetilde{V}(\,*\,)^{\leq \mathfrak{s}_*}$ for $* \in \{f,g\}$
\end{defn}

\subsubsection{} \label{subsubsection_413} 
For the rest of this section, we assume that the Coleman family $\Bf$ satisfies \eqref{item_Reg} and $\Bg$ satisfies \eqref{item_Eis}. For $\Bh \in \{\f, \Bg\}$, let $\Lambda^{\rm b}_\Bh$ denote the ring of bounded analytic functions on $U_\Bh$. We set $\Gamma_{\Bh} := \Gamma_1(N_\Bh) \cap \Gamma_0(p)$ and we write $\bm{\mathcal{L}}_\Bh$ for the $\Lambda^{\rm b}_{\Bh}[\Gamma_\Bh]$-module of locally $m$-analytic distributions on $T':= p\ZZ_p \times \ZZ_p^*$ (see, for example \cite[\S4.1]{BSV}, for more details), for a sufficiently large $m\in \ZZ$ (which we suppress from the notation). We then set \begin{align*}
V(U_{\Bh}) := H^1_{\rm par}(\Gamma_{\Bh}, \bm{\mathcal{L}_{\Bh}})(1)\otimes_{\Lambda^{\rm b}_{\h}} \Lambda_{\Bh}, \qquad 
\widetilde{V}(U_{\Bh}) := H^1(\Gamma_{\Bh}, \bm{\mathcal{L}_{\Bh}})(1)\otimes_{\Lambda^{\rm b}_{\h}} \Lambda_{\Bh}\,,
\end{align*}
and we write $V_{\Bh}$ (resp. $\widetilde{V}_{\Bh}$) for the $\Bh$-isotypic quotient of $V(U_{\Bh})^{\leq\mathfrak{s}_{\h}}$ (resp. $\widetilde{V}(U_{\Bh})^{\leq \mathfrak{s}_{\h}}$). 
If we write 
\begin{align*}
{\rm pr}_{\Bh}: V(U_{\Bh})^{\leq\mathfrak{s}_{\h}} \twoheadrightarrow V_{\Bh}\,,\qquad
{\rm pr}_{\Bh}: \widetilde{V}(U_{\Bh}))^{\leq\mathfrak{s}_{\h}} \twoheadrightarrow \widetilde{V}_{\Bh}
\end{align*}
for the natural projection maps, then we have the isomorphism 
\begin{equation}
\label{cuspidal_isom}
V_{\h_\kappa}\simeq \widetilde{V}_{\h_\kappa}\,, \qquad \forall\,\kappa\in U_\mathbf{h} \cap \mathbb{Z}^{\geq 2}\,, 
\end{equation}
since such specialisations are cuspidal. Here $V_{\h_\kappa}$ denotes the specialisation of $V_{\Bh}$ at $\kappa$ and similarly for $\widetilde{V}_{\h_\kappa}$. We refer the reader to \cite[\S2.5]{BDV} for more details.

\subsection{} \label{subsubsection_414}
Let $k \in \ZZ^{\geq 2} $ and consider $f \in S_k(\Gamma_0(N_\Bf\, p^r))$, where we recall that we assume $r \in \{0,1\}$ for the rest of the paper. If $r=0$, we let $\alpha, \beta$ denote the roots of the Hecke polynomial of $f$ at $p$, where we assume that $\rm{ord}_p(\alpha)\leq  {\rm ord}_p(\beta)$, and let $f_\alpha := f(q) - \beta \cdot f(q^p)$ denote the $p$-stabilisation of $f$ at $\alpha$. If $r=0$ (resp. $r = 1$) we let $f_\alpha = f(q) - \beta \cdot f(q^p)$ denote the $p$-stabilisation of $f$ with respect to $\alpha$. If $r=1$, we simply put $f_\alpha := f$, where $\alpha=a_p(f)$. We remark that, thanks to the works of Hida, Coleman, and Bella\"iche--Chenevier, it follows that $\Bf$ is the unique (up to conjugation) Coleman family passing through $f_\alpha$. We denote by $\LL_\f\xrightarrow{\kappa}E$ the corresponding specialisation of $\f$.
\begin{defn}
    For an integer $c\geq 2$ coprime to $6Np$, we let 
    $$_c\widetilde{{\rm BF}}_r(f_\alpha,k',j)\in H^1_{\textup{\'{e}t}}(\mathbb{Z}[\mu_{p^r}, \frac{1}{Np}], V_{f_\alpha} \otimes \widetilde{V}(k')^{\leq \mathfrak{s}_{\Bg}}(-j))$$ 
    denote the element given by \[_c\widetilde{{\rm BF}}_r(f_\alpha,k',j) =({\rm pr}_{f_\alpha}\otimes {\rm id}) (_c\widetilde{{\rm BF}}_r(k,k',j))\,, \]
    where $\widetilde{V}(k)^{\leq \mathfrak{s}_{\Bf}}\xrightarrow{{\rm pr}_{f_\alpha}} \widetilde{V}_{f_\alpha} \simeq V_{f_\alpha}$ is the projection onto the $f_\alpha$-isotypic quotient; cf. \eqref{cuspidal_isom} for the final isomorphism.
\end{defn}

\subsubsection{} For a fixed $j \in \mathbb{Z}^{\geq 0}$, it follows from \cite{KLZ2,LZ1} (especially, \S5.3 of \cite{LZ1}) that there exists an element 
\[_c\widetilde{{\rm{\bf BF}}}_{j,r}(\Bf\otimes U_\Bg) \in H^1_{{\textup{\'{e}t}}}\left(\mathbb{Z}[\mu_{p^r},\frac{1}{Np}], V_\f\hat{\otimes}_{\mathbb{Q}_p}\widetilde{V}(U_\Bg)^{\leq \mathfrak{s}_{\Bg}}(-j) \right)\]
satisfying the following interpolation property:
\[{{k}\choose{j}}{{k'}\choose{j}}\cdot ({\rm sp}_\kappa\,{\otimes}\,{\rm  sp}_{k'})(_c\widetilde{{\rm{\bf BF}}}_{j,r}(\Bf\otimes U_\Bg)) =   {_c\widetilde{{\rm BF}}(f_\alpha,k',j)}\,.\]
Here, for $X \in{\{\widetilde{V},V \}}$, we denote by 
$$X_\f^{\leq \mathfrak{s}_{\Bf}}\otimes_{\Lambda_\Bf, {\rm{sp}}_\kappa} E\xrightarrow[\,{\rm sp}_\kappa\,]{\,\sim\,} X(f_\alpha)$$ 
the isomorphism induced by the $E$-valued specialisation $\kappa$. Similarly, 
$$X(U_\Bg)^{\leq \mathfrak{s}_{\Bg}}\otimes_{\Lambda_\Bg, {\rm{sp}}_{k'}} E\xrightarrow[\,{\rm sp}_{k'}\,]{\,\sim\,} X(k')^{\leq \mathfrak{s}_{\Bg}}\,.$$
We assume $\mathfrak{s}_{\Bg} = 0$ for the rest of the paper.
\begin{lemma}[\cite{BDV}, Lemma 2.4]\label{lifting_lemma}
    There exists a unique element 
    \[_c{\rm{\bf BF}}_{j,r}(\Bf \otimes U_\Bg) \in H^1_{\textup{\'{e}t}}\left(\mathbb{Z}[\mu_{p^r}, \frac{1}{Np}], V_\f\,\widehat{\otimes}_E\, V(U_\Bg)^{\leq 0}(-j)\right)\,,\] 
    whose image under the natural map induced on cohomology by $V(U_\Bg)^{\leq 0}\hookrightarrow \widetilde{V}(U_\Bg)^{\leq 0}$ coincides with the class $_c\widetilde{{\rm{\bf BF}}}_{j,r}(\Bf\otimes U_\Bg)$.
\end{lemma}
\begin{defn}
    Let 
    \[_c{\rm{\bf BF}}_{j,r}(\Bf\otimes \Bg) \in H^1_{\textup{\'{e}t}}\left(\mathbb{Z}[\mu_{p^r}, \frac{1}{Np}], V_\f\,\widehat\otimes\,V_\g(-j) \right)\]
    denote the image of $_c{\rm{\bf BF}}_{j,r}(\Bf\otimes U_\Bg)$ under the map induced by the projection $V(U_\Bg)^{\leq 0}\xrightarrow{{\rm pr}_\Bg} V_\Bg$.
\end{defn}
\subsubsection{} Let $\mathcal{H}:= \mathcal{H}(\ZZ_p^\times)$ denote Perrin--Riou's extended Iwasawa algebra. The following proposition is a direct consequence of Lemma~\ref{lifting_lemma} and the results of \cite{LZ1}.
\begin{proposition}\label{proposition_BF_properties} 
    \item[(a)] There exists an element $_c{\rm{\bf BF}} (\Bf\otimes \Bg) \in H^1_{\rm Iw}\left(\mathbb{Q}(\mu_{p^\infty}), V_\f\,\widehat\otimes\,V_\g \otimes \mathcal{H} \right)$ satisfying the equality 
    \[{\rm ev}_j(_c{\rm{\bf BF}} (\Bf\otimes \Bg)) = (a_p(\Bf)\cdot a_p(\Bg))^{-r}\cdot (j!)^{-1}\cdot {_c{\rm{\bf BF}}_{j,r}(\Bf\otimes \Bg)}\] 
    for all natural numbers $j$, where the map ${\rm ev}_j$ induced by evaluation on $\mathcal{H}$ at $j$. Furthermore, if $\mathfrak{s}_\Bf =0$, then $_c{\mathbf{BF}(\Bf \otimes \Bg)} \in H^1_{\rm Iw}\left(\mathbb{Q}(\mu_{p^\infty}), V_\f\,\widehat\otimes\,V_\g \right)$.
    \item[(b)] Let $\kappa$ (resp. $\kappa'$) denote specialisations of $\LL_\f$ (resp. of $\LL_\g$) of weight $k+2$ (resp. $k'+2$). Then the element $_c{\rm{\bf BF}} (\Bf\otimes \Bg)$  satisfies
    \[(\kappa\hatotimes \kappa')\circ {\rm ev}_j (_c{\rm{\bf BF}} (\Bf\otimes \Bg)) = \frac{1}{j! {k\choose j}{k' \choose j} }\left(1-\frac{p^j}{a_p(\Bf_\kappa)\cdot a_p(\Bg_{\kappa'})}\right){}_c{\rm{\bf BF}}(\Bf_\kappa, \Bg_{\kappa'},j)\] 
    for integers $0\leq j\leq\min(k,k')$, where 
    $_c{\rm{\bf BF}}(\Bf_\kappa, \Bg_{\kappa'},j) \in H^1\left(\mathbb{Q}, V_{\Bf_k}\otimes V_{\Bg_{k'}}(-j)\right) $
    is the image of the class $_c\widetilde{{\rm BF}}_0(k,k',j)$ given as in Definition~\ref{defn_BFtilde_k_kp_j}, under the map induced by the projection onto $\Bf_\kappa\times\Bg_{\kappa'}$-isotypic components. 
\item[(c)] The class $_c{\rm{\bf BF}} (\Bf\otimes \Bg)$ belongs to the balanced Selmer group $H^1_{{\rm Iw},{\rm bal}}\left(\mathbb{Q}(\mu_{p^\infty}), V_\Bf\,\widehat\otimes V_\g\,)\right)$ given as in \S\ref{subsubsec_2025_12_17_1139}.
\end{proposition}
\begin{proof}
    This is explained as part of the discussion immediately following the proof of \cite[Lemma 2.4]{BDV}.
\end{proof}
We assume until the end of this paper that $\varepsilon_\f=\mathds{1}$ is the trivial character.
\begin{defn}
   We let $_c{\mathbf{BF}^\dagger(\Bf \otimes \Bg)} \in H^1_{\rm Iw}\left(\mathbb{Q}(\mu_{p^\infty}), V_\f^\dagger\,\widehat\otimes\,V_\g \otimes \mathcal{H} \right)$ $($resp. in $H^1_{\rm Iw}\left(\mathbb{Q}(\mu_{p^\infty}), V_\f^\dagger\,\widehat\otimes\,V_\g\right)$, if $\mathfrak{s}_\Bf=0$$)$ denote the central critical twist of the Beilinson--Flach element. Explicitly, 
   \[ _c{\mathbf{BF}^\dagger(\Bf \otimes \Bg)} = \rm{Tw}_{\Theta^{-1}}\left(_c{\mathbf{BF} (\Bf \otimes \Bg)}\right)\,, \] 
   where $\Theta$ is given as in \cite[Definition 2.1.3]{howard2007Inventiones} and $\rm{Tw}_{\Theta^{-1}}$ is the map induced from the twisting morphism
   \[V_\Bf \hatotimes \Lambda(\Gamma_{\rm cyc}) \,\xrightarrow{\,x\otimes \gamma \,\mapsto\, \Theta^{-1}(\gamma)\,x \otimes \gamma \,}\, V_\Bf^\dagger \hatotimes \Lambda(\Gamma_{\rm cyc})\,.\] 
\end{defn}

\subsection{Beilinson-Flach Elements over $K$} \label{subsubsection_433}
We now specialise to the setting where $\g=\g_\tau$ is a CM family with branch character $\tau$ given as in \S\ref{subsubsec_356_2026_1_19_1714}. Recall the affinoids $\rm{Spm} (\mathscr{O}_{\Bg_{\tau,n}})$ and the associated big Galois representations $V_{\Bg_{\tau, n}}$ introduced in \S \ref{Identifying_Iwasawa_Cohomologies_via_Covering}.
\subsubsection{}
Let us denote by 
$$_c{\rm{\bf BF}}^\dagger(\Bf\otimes {\g_\tau}) \in \varprojlim_n H^1\left(\QQ, V_\Bf^\dagger \hatotimes V_{\g_{\tau, n}} \hatotimes \mathcal{H}(\Gamma^\circ_{\rm cyc})\right)$$
the image of $({_c{\mathbf{BF}^\dagger(\Bf \otimes \g_{\tau,n})}})_n$ under the map induced by $\Gamma_{\rm cyc} \twoheadrightarrow \Gamma^\circ_{\rm cyc}$ and recall the isomorphism \eqref{equation_Induced_Rep_isom}, which we fix in a compatible way with our earlier choice of the basis $\{\lambda_{\g_\tau}, \lambda_{\g_\tau}^c\}$. We also recall our convention that $\g_{\mathds{1}} = \g$ and we note that in this case, we actually view $_c\mathbf{BF}^\dagger(\f \otimes \g_n)$ as an element of $  H^1\left(\Q, V^\dagger_\f \hatotimes V^1_{\g_n} \hatotimes \mathcal{H} \right)$ (cf. the discussion at the end of \S \ref{subsubsec_356_2026_1_19_1714}), where we recall that $(\cdot)^1 = (\cdot) / (\cdot)_{\rm tors}$.
\begin{remark}
\normalfont
    When $\tau=\mathds{1}$, we only know that we have an element $_c\mathbf{BF}^\dagger(\f\otimes \g_{n}) \in H^1\left(\Q, V^\dagger_\f \hatotimes \widetilde{V}^1_{\g_n} \hatotimes \mathcal{H} \right)$ a-priori. However Lemma \ref{lifting_lemma} allows us to lift this to an element of $ H^1\left(\Q, V^\dagger_\f \hatotimes {V}^\mathds{1}_{\g_n} \hatotimes \mathcal{H} \right)$. We henceforth dispense with the superscript $(\cdot)^\mathds{1}$ and write $V_{\g_n}$. \hfill $\triangle$
\end{remark}
Let us set
\[_c{\rm{\bf BF}}_\Bf^\tau \in H^1\left(K, V_\Bf^\dagger \hatotimes \mathcal{H}^{\sharp}_\tau(\Gamma_\fp^\circ) \hatotimes \mathcal{H}^\iota(\Gamma^\circ_{\rm cyc})\right)\] 
to be the image of $_c{\mathbf{BF}^\dagger(\Bf \otimes \Bg_\tau)}$ under Shapiro's isomorphism, and let 
\[{}_c{\rm BF}^{\tau}_{\rm ac, \Bf} \in H^1\left(K, V^\dagger_{\Bf} \hatotimes \mathcal{H}_\tau^\iota(\Gamma^\circ_{\rm ac})\right )\,,\] 
denote the image of $_c{\rm{\bf BF}}_\Bf^\tau$ under the map induced by \eqref{eq_ses_gamma_gamma_gammaac_first}. 
\subsubsection{}
By \cite[Corollary 4.4.11]{KPX} we have that 
$$ \varprojlim_n \, H^1 \left(\QQ_{p}, V_\Bf^\dagger \hatotimes V_{\Bg_{\tau,n}} \hatotimes \mathcal{H}(\Gamma^\circ_{\rm cyc})\right) \simeq \varprojlim_n \, H^1 \left(\QQ_{p}, \bD^\dagger_{\rm rig}(V_\Bf \hatotimes V_{\Bg_{\tau,n}}) \hatotimes \bD^\dagger_{\rm rig}(\mathcal{H}^\iota(\Gamma^\circ_{\rm cyc}))\right).$$
Consider the image of $\rm{res}_p({_c\mathbf{BF}^\dagger(\f \otimes \g_{\tau})})$ under the map on cohomology induced by the maps $\bD^\dagger_{\rm rig}(V_\Bf \hatotimes V_{\Bg_{\tau,n}}) \twoheadrightarrow \bD^\dagger_{\rm rig}(V_\Bf) \hatotimes \DD^-_{\g_{\tau,n}} $. By Proposition~\ref{proposition_BF_properties}(c), this class corresponds to a uniquely determined class 
\begin{align*}
\begin{aligned}
    \label{eqn_2026_02_04_1422}
     {_c{\rm BF}^{\tau^c,+}_{\fp}}\,\in\, H^1\left(K_{\fp}, \DD^{\dagger+}_\Bf \hatotimes\bD^\dagger_{\rm rig}(\mathcal{H}^{\sharp}_{\tau}(\Gamma^\circ_\fp)^c)\hatotimes \bD^\dagger_{\rm rig}(\mathcal{H}^\iota(\Gamma^\circ_{\rm cyc})) \right)\,\stackrel{\eqref{-=fp-}}{\simeq}\, H^1 \left(\QQ_{p}, \DD^{\dagger+}_\Bf \hatotimes \DD^{-}_{\g_\tau} \hatotimes \bD^\dagger_{\rm rig}(\mathcal{H}^\iota(\Gamma^\circ_{\rm cyc})) \right)\,.
\end{aligned}
\end{align*}
 We also define 
$$_c{\rm BF}^{\tau, -}_{\fp} \in H^1 \left(K_\fp, \DD^{\dagger-}_\f \hatotimes \bD^\dagger_{\rm rig}(\mathcal{H}^\sharp_{\tau}(\Gamma^\circ_\p)) \hatotimes \bD^\dagger_{\rm rig}(\mathcal{H}^\iota(\Gamma^\circ_{\rm cyc}))\right) \stackrel{\eqref{+=fp}}{\simeq} H^1 \left(\Qp, \DD^{\dagger-}_\Bf \hatotimes \DD^+_{\g_\tau} \hatotimes \bD^\dagger_{\rm rig}( \mathcal{H}^{\iota}(\Gamma^\circ_{\rm cyc})) \right)$$ 
as the unique element that corresponds to the image of ${\rm res}_p({_c{\mathbf{BF}}^\dagger(\Bf\otimes {\g_\tau})})$ under the map induced by $\Ddagrigf(V_\Bf^\dagger) \twoheadrightarrow \DD^{\dagger-}_\Bf$.

\begin{defn}
    \label{defn_2026_02_04_1426}
    Let us denote by $_c{\rm BF}^{\tau^{-1}, +}_{\rm ac,\fp}$ the image of $_c{\rm BF}^{\tau^c ,+}_\fp$ inside
\[{\rm im}\left(H^1\left(K_{\fp}, \DD^{\dagger+}_\Bf \hatotimes\bD^\dagger_{\rm rig}(\mathcal{H}^{\sharp}_{\tau}(\Gamma^\circ_\fp)^c)\hatotimes \bD^\dagger_{\rm rig}(\mathcal{H}^\iota(\Gamma^\circ_{\rm cyc})) \right)/(\gamma^+-1) \hookrightarrow H^1 \left (K_\fp, \DD^{\dagger+}_\Bf \hatotimes \bD^\dagger_{\rm rig}( \mathcal{H}_{\tau^{-1}}^\iota(\Gamma^\circ_{\rm ac}))\right) \right)\,.\] 
Similarly, we let $_c{\rm BF}^{\tau,-}_{\rm ac, \fp}$ denote the image of $_c{\rm BF}^{\tau, -}_{\fp}$ inside
\[ {\rm im}\left(H^1 \left(K_\fp, \DD^{\dagger-}_\f \hatotimes \bD^\dagger_{\rm rig}(\mathcal{H}^\sharp_{\tau}(\Gamma^\circ_\p)) \hatotimes \bD^\dagger_{\rm rig}(\mathcal{H}^\iota(\Gamma^\circ_{\rm cyc}))\right)/(\gamma^+-1) \hookrightarrow H^1 \left(K_\fp, \DD^{\dagger-}_\Bf \hatotimes \bD^\dagger_{\rm rig}(\mathcal{H}_\tau^\iota(\Gamma^\circ_{\rm ac}))\right) \right).\]
\end{defn}
See Remark~\ref{remark_simpler_description_of_BF^+_p} below for an equivalent and simpler description of $_c{\rm BF}^{\tau^{-1}, +}_{\rm ac,\fp}$.

\subsubsection{Reciprocity laws for $_c{\rm BF}^\tau_{\rm ac ,\Bf}$}
In this subsubsection, we recall the $p$-adic $L$-functions attached to the Rankin--Selberg convolution of two modular forms. Using the isomorphisms that we fixed in \S\ref{subsec_Iwasawa_Algebras}, we may view these $p$-adic $L$-functions as functions of cyclotomic and anticyclotomic variables. We may thus write a ``Taylor expansion'' of these $p$-adic $L$-functions in the variable $\gamma_\cyc-1$, and it is the coefficient of $\gamma_\cyc -1$ that we will view as the cyclotomic derivative of the $p$-adic $L$-function ``evaluated at the trivial (cyclotomic) character''. Moreover, we recall the relationship between Beilinson-Flach elements and these $p$-adic $L$-functions in the form of explicit reciprocity laws. These results will be leveraged in \S \ref{subsec_GZ_formula} in an essential way to prove the main theorem of this article.

In \cite[Appendix II]{AndreattaIovita_triple_product} and \cite{LoefflerCMB}, the authors construct the aforementioned $p$-adic $L$-functions attached to the pair $(\f,\g)$; one for when the family $\f$ is dominant, and one for when the family $\g$ is dominant. These $p$-adic $L$-functions are elements of $\Lambda_\f \hatotimes \Lambda_\g \hatotimes \mathcal{H}(\Gamma_\cyc)$. If $\h \in \{\f,\g\}$, let us write
$$L^{\Bh}_p(\Bf,{\g_\tau},1+\Bj) \in  \Lambda_\Bf \hatotimes \mathcal{H}(\Gamma^\circ_{\fp}) \hatotimes \mathcal{H}(\Gamma^\circ_{\rm cyc}) \stackrel{\eqref{eq_isom_+_-} + \eqref{eq_isom_star_bullet} }{\simeq} \Lambda_\Bf \hatotimes \mathcal{H}(\Gamma^\circ_{\rm ac}) \hatotimes \mathcal{H}(\Gamma^\circ_{\rm cyc})$$ 
for the image (under our identification $\Lambda_{\g_\tau} \simeq \Lambda(\Gamma_\fp^\circ)$) of the $\Bh$-dominant $p$-adic $L$-function after projecting $\Lambda(\Gamma_{\rm cyc}) \twoheadrightarrow \Lambda(\Gamma^\circ_{\rm cyc})$. We may write 
\begin{equation}
    L_p^{\Bh}(\Bf,{\g_\tau}, \Bj+\frac{\Bk}{2}) = \mathfrak{L}^{\Bh}_{{\rm ac},0}(\Bf,\tau) + \mathfrak{L}^{\Bh}_{{\rm ac},1}(\Bf,\tau)\cdot \left(\frac{\gamma_{\rm cyc}-1}{\log(\chi_{\cyc}(\gamma_\cyc))}\right) + \mathfrak{L}^{\Bh}_{{\rm ac},2}(\Bf,\tau)\cdot \left(\frac{\gamma_{\rm cyc}-1}{\log(\chi_{\cyc}(\gamma_\cyc))}\right)^2 +\dots, \label{eq_expansion_of_l_function}
\end{equation}
with $\mathfrak{L}^{\Bh}_{{\rm ac},i}(\Bf) \in \Lambda_\Bf \hatotimes \mathcal{H}(\Gamma^\circ_{\rm ac} )$. As before, let us denote by $\mathbf{k}$ (resp. $\Bk'$) the universal cyclotomic character on the weight space of $\f$ (resp. $\g_\tau$), and $\mathbf{k}/2$ (resp. $\Bk'/2$) its unique square root.

\begin{theorem}[Reciprocity  laws for Beilinson--Flach elements] \label{Theorem_Reciprocity_for_BF}\label{reciprocity_BF}
  Let $\lambda_\Bh$ denote the Atkin--Lehner pseudo-eigenvalue on the primitive family $\Bh \in \{\Bf, \g_\tau\}$. Then there exists an explicit element $$\mathscr{R}(\g_\tau) \in \Frac(\Lambda_\Bf \hatotimes \mathcal{H}(\Gamma^\circ_\p) \hatotimes \mathcal{H}(\Gamma^\circ_{\rm cyc})) $$ 
  depending only on $\g_\tau$ so that the following hold:
   
   \item[i)] $\langle\mathfrak{Log}^{\Gamma}_{-+} ({_c{\rm{BF}}_\fp^{ \tau, -}}), \eta_\Bf \rangle_\DD = (-1)^{\Bj+\frac{\Bk}{2}} \lambda_{{\Bf}}^{-1} \cdot (c^2-c^{2\Bj -\Bk'} \cdot \varepsilon_\Bf (c)
^{-1}\varepsilon_{\g_\tau}(c)^{-1})\cdot L_p^{(\f)}(\Bf,{\g_\tau},\Bj+\frac{\Bk}{2}). $
    \item[ii)] 
    $\langle\mathfrak{Log}^{\Gamma}_{+-} ({_c{\rm{BF}}_\fp^{ \tau^c, +}}), \omega_\Bf \rangle_\DD = \frac{1}{\mathscr{R}(\g_\tau)}(-1)^{\Bj+\frac{\Bk}{2}} \lambda_{{\Bg_\tau}}^{-1} \cdot (c^2-c^{2\Bj -\Bk'} \cdot \varepsilon_\Bf (c)
^{-1}\varepsilon_{\g_\tau}(c)^{-1})\cdot L_p^{(\g_\tau)}(\Bf,{\g_\tau},\Bj+\frac{\Bk}{2} ).$
\end{theorem}
\begin{proof}
    Let us set 
    \[\mathscr{R}(\g_\tau):= \langle \lambda_{\g_\tau}^c,\eta_{\g_\tau}\rangle,\] 
    where we have abused notation to denote by $\lambda_{\g_\tau}^c$ the element of $\DD^-(V_{\g_{\tau}})$ corresponding to the fixed generator $\lambda_{\g_\tau}^c \in \mathcal{F}^-\TT_{\g_\tau}$. The result now follows directly as a consequence of \cite[Theorem 7.1.5]{LZ1} when \eqref{item_Reg'} holds (see also Remark $7.1.6$ of op. cit.) and \cite[Proposition 2.3]{BDV} when \eqref{item_Eis'} holds, together with the definition of $\mathfrak{Log}^\Gamma_{?}$ for $? \in \{-+, +-\}$.
\end{proof}
Recall the maps $\mathscr{L}^\tau_{\omega_\f}$ and $\mathscr{L}^\tau_{\eta_\f}$ from Definition \ref{definition_L_omega_and_L_eta}. We then have the following immediate consequence of Theorem \ref{Theorem_Reciprocity_for_BF}:

\begin{corollary} 
\label{corollary_reciprocity_for_anticyclotomic_BF} 
    For $\Bh \in \{\Bf, \g_\tau \}$, let us set $\mathscr{C}_{\rm ac}(\Bh) = (-1)^{\Bj+ \frac{\Bk}{2}} \lambda_{{\Bh}}^{-1} \cdot (c^2-c^{2\Bj -\Bk'} \cdot \varepsilon_\Bf (c)
^{-1}\varepsilon_{\g_\tau}(c)^{-1}) \ {\rm mod} \ (\gamma_+-1)$. Let us also put $\mathscr{R}_{\rm ac}(\g_\tau) := \mathscr{R}(\g_\tau) \ {\rm mod} \ (\gamma_+-1) \in \Frac(\Lambda_\Bf \hatotimes \mathcal{H}(\Gamma^\circ_{\rm ac}))$.  Then the following hold:
  \item[i)] $\mathscr{L}^\tau_{\eta_\Bf}({_c{\rm BF}^{\tau,-}_{\rm ac, \fp}}) = \mathscr{C}_{\rm ac}(\Bf)\cdot \mathfrak{L}^{\Bf}_{{\rm ac},0}(\Bf,\tau)$\,.
   
    \item[ii)] $\mathscr{L}^{\tau^{c}}_{\omega_\Bf}({_c{\rm BF}^{\tau^{-1},+}_{\rm ac, \fp}}) = \dfrac{\mathscr{C}_{\rm ac}(\g_\tau)}{\mathscr{R_{\rm ac}(\g_\tau)}}\cdot \mathfrak{L}^{\Bg_\tau}_{{\rm ac},0}(\Bf,\tau)\,.$
\end{corollary}

\begin{remark}\label{remark_R_ac_vs_L(g)}
\normalfont
    Suppose $\tau=\mathds{1}$ and let $\Lambda(\Gamma^\circ_\p)\xrightarrow{\kappa'_\circ} \Zp$ denote the augmentation map, so that $\kappa'_\circ$ corresponds to the weight one specialisation of $\g$: 
    \[g_{\rm Eis}:=\g_{\kappa'_\circ} = \sum_{(\mathfrak{a}, \p)=1} \mathds{1}(\mathfrak{a}) \,q^{\mathbf{N}(\mathfrak{a})}\,.\] 
    We recall the element $\mathcal{L}(g)$ defined at the start of \cite[\S4.6]{BDV} (and note that $g$ in op. cit. is $g_{\rm Eis}$ in this article). Setting $\mathscr{R} := \mathscr{R}_{\rm ac}(\g_{\kappa'_\circ})$, we then have that $\mathscr{R}= {1}/{\mathcal{L}(g)}$.
     \hfill $\triangle$
\end{remark}
%%%%%%%%%%%%%%%%%%%%%%%%%%%%%%%%
%%%%%%%%%%%%%%%%%%%%%%%%%%%%%%%%
%%%%%%%%%%%%%%%%%%%%%%%%%%%%%%%%

\section{A review of big Heegner classes}
\label{subsubsec_2026_01_07_1431}
Let $K$ be an imaginary quadratic field where $p=\p\p^c$ splits, and $N\geq 5$ an integer either coprime to $p$ or else $p\mid\mid N$, satisfying the Heegner hypothesis relative to $K$. 
%That there exists an ideal $\fN \subset \mathcal{O}_K$ such that $\mathcal{O}_K/\fN \simeq \ZZ/N\ZZ$. 
Let $f \in S_{k+2}^{\rm new}(\Gamma_0(N))$ be a newform as in our introduction (\S\ref{sec_2026_01_10_1700}). As in the introduction, let us fix a $p$-sabilisation $f^\alpha$ of $f$, where $\alpha$ is a root of the Hecke polynomial of $f$ at $p$. Note that if $p\mid N$, then $\alpha=\pm p^{\frac{k}{2}}$ by \cite[Theorem 4.6.17]{miyake06}. Finally, we let $\f$ denote the unique finite slope family through $f^\alpha$, as in \S\ref{subsec_Hida_families} or \S\ref{subsec_Coleman_families_revisited}.  
\subsection{Heegner cycles}
We recall the basic properties of Heegner cycles on Kuga--Sato varieties and associated Heegner classes, following the discussion in \cite[\S II]{nekovarGZ}. Let $w$ be a natural number and let $h\in S_{w+2}^{\rm new}(\Gamma_0(N))$ be a newform. 

Let $Y(N)$ denote the modular curve over $\QQ$, which is the moduli of elliptic curves with full level $N$ structure. We let $\mathfrak{j}:\, Y(N) \to X(N)$ denote its non-singular compactification. Since we assume $N \geq 3$, there is a universal generalised elliptic curve $\mathscr{E}\to X (N)$, which restricts to $\mathscr{E}\to Y(N)$. Assume first that $w>0$. Then the $w$-fold fibre product of $\mathscr{E}$ with itself over $Y(N)$ has a canonical non-singular compactification $W$, cf. \cite{Deligne1971Bourbaki, Scholl1990Inventiones}. We have natural maps
\begin{equation}
\label{eqn_KugaSatotoVg}
H^{w+1}_{\textup{\'{e}t}}(W_{\overline{\QQ}},\QQ_p)(1+w/2)\lra  H^{1}_{\textup{\'{e}t}}(X(N)_{\overline{\QQ}},\mathfrak{j}_*{\rm Sym}^{w}({\mathscr{H}_{\QQ_p}^\vee}))(1+w/2)\ra V_h^\dagger(\mathscr{H}_{\QQ_p}^\vee)\,.
\end{equation}
Scholl in \cite{Scholl1990Inventiones} defines a projector $\varepsilon$ and proves that there is a canonical isomorphism
$$H^{1}_{\textup{\'{e}t}}(X(N)_{\overline \QQ},\mathfrak{j}_*{\rm Sym}^{w}({\mathscr{H}_{\QQ_p}^\vee}))\stackrel{\sim}{\lra}\varepsilon H^{w+1}_{\textup{\'{e}t}}(W\times_{\overline\QQ},\QQ_p)\,.$$ 
We finally define 
$$\mathscr{B}:=\left\{\left(\begin{array}{cc}* & *\\ 0& * \end{array} \right) \right\}\Big{/}\{\pm 1\} \subset \GL_2(\ZZ/N\ZZ)\big{/}\{\pm 1\} $$
and the idempotent $\varepsilon_{\mathscr B}:=\frac{1}{|\mathscr{B}|}\sum_{g\in \mathscr{B}}g$, which acts on the modular curves $Y(N)$ and $X(N)$.

Let $H_K$ denote the Hilbert class field of $K$. We let ${\rm CH}^{\kappa/2}(W_{H_K})_0\otimes \QQ \xrightarrow{{\rm AJ}_h} H^1(H_K,V_h^\dagger(\mathscr{H}_{\QQ_p}^\vee))$ be the compositum of the map \eqref{eqn_KugaSatotoVg} with the Abel--Jacobi map. We then have a Heegner class
$$z_h:={\rm cor}_{H_K/K}\left({\rm AJ}_h(\varepsilon_{\mathscr B}\varepsilon Y)\right) \in H^1(K,V_h^\dagger(\mathscr{H}_{\QQ_p}^\vee))\,,$$
where the idempotent $\varepsilon_{\mathscr B}$ and the projector $\varepsilon$ are as above, and $Y$ is the cycle of dimension $w/2$ given as in \cite[\S II.3.6]{nekovarGZ} (so that $\varepsilon_{\mathscr B}\varepsilon Y$ is the Heegner cycle).

The case $w=0$ is elementary: One simply takes $W=X(N)$ above and $Y$ a point that has CM by $\cO_K$.

\subsubsection{} The discussion above generalises and permits one to define Heegner classes $z_h(c)\in H^1(K[c],V_h^\dagger(\mathscr{H}_{\QQ_p}^\vee))$
for all ring-class extensions of $K$ of conductor $c$. 

\subsubsection{}
\label{subsubsec_2026_1_11_1120}
Employing the isomorphism \eqref{eqn_2026_01_10_1340}, we have $\lambda_N(h)^{-1}W_N(z_h)\in H^1(K,V_h^\dagger)$. When we treat the isomorphism as an identification (as, for example, Nekov\'a\v{r} does, cf. \cite{nekovarGZ}, \S II.4.2), we will sometimes write $z_h$ in place of $\lambda_N(h)^{-1}W_N(z_h)$ to ease our notation. 

\subsubsection{} When $p\nmid N$, the varieties $X(N)$, $X_0(N)$ and $W$ have good reduction at $p$, and it follows from \cite[Theorem 3.1(i)]{Nekovar2000AJHeightsBanff1998} that $z_h\in H^1_{\rm f}(K,V_h^\dagger)$, namely the Heegner class falls within the Bloch--Kato Selmer group. If $p\mid \mid N$, this is more subtle, and this fact follows from the recent work of Nekovář and Nizioł~\cite{NekovarNiziol}.

\subsection{Big Heegner classes} 
The Heegner classes introduced above can be interpolated in $p$-adic families. The following result in this vein summarises the main results of \cite{howard2007Inventiones,Ota2020, JLZ}. Together with Corollary~\ref{cor_thm_bigheegnermain}, it will be used to upgrade Corollary~\ref{cor_2026_01_20_2022} (the $p$-adic Gross--Zagier theorem up to non-zero algebraic factors) to its precise form in Theorem~\ref{thm_2026_01_20_2021}.

Let $\f$ be the unique primitive family through $f^\alpha$ of tame level $N_\f=N/\gcd(N,p)$ and trivial tame nebentype. We say that a classical specialisation $\kappa$ is crystalline if the corresponding eigenform $\f_\kappa$ is $p$-old. In that situation, we put $\mathcal{E}(\f_\kappa):=\left(1-\frac{p^{w}}{\alpha(\kappa)^2}\right)$, and $\mathcal{E}^*(\f_\kappa):=\left(1-\frac{p^{w+1}}{\alpha(\kappa)^2}\right)\,.$ 
\begin{theorem}
\label{thm_bigheegnermain} 
There exists a unique class $\mathscr{Z}_{\f}\in H^1_{\rm f}(G_{K,S}, V_\f^\dagger)$ that is characterised by the following interpolative property: For any classical crystalline specialisation $\kappa$ of weight $w\geq 0$, so that $\f_\kappa\in S_{w+2}(\Gamma_0(Np))$ is $p$-old, we have 
\begin{equation}
\label{eqn_2026_01_bigheegnermain_2}
{\rm Pr}^{\alpha(\kappa)}_*\circ\,\kappa\,\circ \mathscr{Z}_{\f}=\dfrac{\left(1-\dfrac{p^{\frac{w}{2}}}{\alpha(\kappa)}\right)^{2}}{u_K(2\sqrt{-D_K})^{\frac{w}{2}}}\cdot \alpha(\kappa)^{-1}\, \lambda_N(\f_\kappa^\circ)^{-1}W_N(z_{\f_\kappa^\circ}) \in H^1_{\rm f}(G_{K,S},V_{\f_\kappa^\circ}^\dagger),
\end{equation}
where $u_K=|\cO_K^\times|/2$ and $-D_K$ is the discriminant of $K$. Moreover, 
\begin{equation}
\label{eqn_2026_01_bigheegnermain_1}
\kappa\,\circ\,\mathscr{Z}_{\f}\,=\,{\dfrac{\left(1-\dfrac{p^{\frac{w}{2}}}{\alpha(\kappa)}\right)^{2}}{u_K(2\sqrt{-D_K})^{\frac{w}{2}}}\frac{W_{Np}\circ({\rm Pr}^{\alpha(\kappa)})^*({z_{\f_\kappa^\circ}})}{\alpha(\kappa)\lambda_N(\f_\kappa^\circ)\mathcal{E}(\f_\kappa)\mathcal{E}^*(\f_\kappa)} }\in  H^1_{\rm f}(G_{K,S},V_{\f_\kappa}^\dagger)\,.
\end{equation}
If, on the other hand, if $p\mid N$ and $\f_\kappa=f$, then
\begin{equation}
\label{eqn_2026_01_bigheegnermain_3}
\kappa\,\circ\,\mathscr{Z}_{\f}\,=\,\alpha^{-1}z_{\etale}^{[f,\mathbf{N}^{\frac{k}{2}}]}\in H^1_{\rm f}(G_{K,S},V_{f}^\dagger)\,,
\end{equation}
where $\mathbf{N}$ the norm character and $z_{\etale}^{[f,\mathbf{N}^{\frac{k}{2}}]}$ is the generalised Heegner class given as in \cite[Definition 3.5.3]{JLZ}.
\end{theorem}

\begin{proof}
For $\kappa$ as in \eqref{eqn_2026_01_bigheegnermain_2}, it follows from the interpolation property of big generalised Heegner cycles (cf. \cite{JLZ}, Theorem 5.4.1) that 
\begin{equation}
\label{eqn_JLZ_541}
    {\rm Pr}^{\alpha(\kappa)}_*\circ\kappa\,\circ \mathscr{Z}_{\f}=\left(1-\dfrac{p^{\frac{w}{2}}}{\alpha(\kappa)}\right)^{2} \alpha(\kappa)^{-1}\,z_{\etale}^{[\f_\kappa^\circ,\mathbf{N}^{\frac{w}{2}}]}\in H^1_{\rm f}(G_{K,S},V_{\f_\kappa^\circ}^\dagger)\,,
\end{equation}
where $\mathbf{N}$ the norm character as above, and $z_{\etale}^{[\f_\kappa^\circ,\mathbf{N}^{\frac{w}{2}}]}$ is the generalised Heegner class; see also Remark~\ref{remark_2026_06_31_0939}. %$\alpha(\kappa)$ is because the class that fits in the family is the $p$-stabilised class, namely $\alpha(\kappa)^{-1} {Heeg}_{Np}=z_{\etale}^{[\f_\kappa^\circ,\mathbf{N}^{\frac{w}{2}}]}$$, where Heeg_{Np} is the Heegner class at level $Np$.
Moreover, by the comparison of generalised Heegner cycles and classical Heegner cycles from \cite[Proposition 4.1.2]{bertolinidarmonprasanna13} (see also the proof of \cite{CastellapadicvariationofHeegnerpoints}, Theorem~6.5), we have
\begin{equation}
\label{eqn_BDP_412}
z_{\etale}^{[\f_\kappa^\circ,\mathbf{N}^{\frac{w}{2}}]}=\dfrac{\lambda_N(\f_\kappa^\circ)^{-1}W_N(z_{\f_\kappa^\circ})}{u_K(2\sqrt{-D_K})^{\frac{w}{2}}}\,,
\end{equation}
where we remark that the appearance of the normalised Atkin--Lehner operators is explained in \S\ref{subsubsec_2026_1_11_1120}. The proof of \eqref{eqn_2026_01_bigheegnermain_2} follows on combining \eqref{eqn_JLZ_541} and \eqref{eqn_BDP_412}. 

The equality \eqref{eqn_2026_01_bigheegnermain_1} is a restatement of \cite[Proposition 4.15(i)]{BL-GHC}, whereas the equality \eqref{eqn_2026_01_bigheegnermain_3} follows from \cite[Proposition 5.3.1]{JLZ}. %omission of ${k-2 \choose k/2-1}$ is explained in the proof of Theorem 7.2.4 in op. cit.}
\end{proof}

\begin{remark}
\label{remark_2026_06_31_0939}
\normalfont
    The Galois representations considered in \cite{JLZ} are of the form $V_{f} \otimes \chi $ for some Galois character $\chi$ associated with a Hecke character of infinity type $(k-j, j)$. In this article, we consider Galois representations of the form $V_f \otimes \chi'=V_f^\dagger \otimes \mathbf{N}^{-\frac{k}{2}}\chi'$, where $\chi'$ corresponds to a Hecke character of infinity type $(k/2 -j, k/2 +j)$. We can obtain one from the other by twisting by the anticyclotomic Galois characters associated to the Hecke character of infinity type $(\frac{k}{2}, -\frac{k}{2})$ and trivial conductor. This amounts to a re-parametrisation of the two-parameter families considered here and in \cite{JLZ}.  \hfill $\triangle$
\end{remark}
\begin{defn}
    \label{defn_adhoc_Heegner}
    Suppose that $p\mid\mid N_f$. Let us define $z_f:=2\alpha^{-1}z_{\etale}^{[f,\mathbf{N}^{\frac{k}{2}}]}/(2\sqrt{-D_K})^{\frac{k}{2}}$, where $z_{\etale}^{[f,\mathbf{N}^{\frac{k}{2}}]}$ is the generalised Heegner class as in the statement of Theorem~\ref{thm_bigheegnermain}. 
\end{defn}

\begin{remark}
\label{rem_2026_01_13_1605}
\normalfont
We expect that when $p\mid\mid N_f$, the element $z_f$ of Definition~\ref{defn_adhoc_Heegner} coincides with the classical Heegner class. Note that $2=(1+p^{\frac{k}{2}}/\alpha)$ is the relevant Euler factor.
 \hfill $\triangle$
\end{remark}
\begin{corollary}
    \label{cor_thm_bigheegnermain}
In the situation of Theorem~\ref{thm_bigheegnermain}, we have the following identities of $p$-adic heights.
\item[i)] Suppose that the specialisation $\kappa$ is crystalline of weight $w\geq 0$. Then:  
$$\kappa\,\circ\, \left<\mathscr{Z}_{\f},\mathscr{Z}_{\f}\right>=\dfrac{\left(1-\frac{p^{\frac{w}{2}}}{\alpha(\kappa)}\right)^{4}}{(4|D_K|)^{\frac{w}{2}}}\dfrac{h_{\f_{\kappa}^\circ}(z_{\f_\kappa^\circ},z_{\f_\kappa^\circ})}{\alpha(\kappa)\lambda_{N_\f}(\f_\kappa^\circ)\mathcal{E}(\f_\kappa)\mathcal{E}^*(\f_\kappa)}\,.$$

\item[ii)] Suppose that $p\mid N$ and $\f_\kappa=f$. Then:
$$\kappa\,\circ\, \left<\mathscr{Z}_{\f},\mathscr{Z}_{\f}\right>=\dfrac{4\,h_{f}(z_{f},z_{f})}{\alpha(4|D_K|)^{\frac{k}{2}}}\,.$$
\end{corollary}

\begin{proof}  Let us first assume that we are in the situation of (i). We then have, by definition,
\begin{equation}
    \label{eqn_2026_05_30_1752}
    \kappa\,\circ\, \left<\mathscr{Z}_{\f},\mathscr{Z}_{\f}\right>=\left<\kappa\circ\mathscr{Z}_{\f}\,,\,\kappa\,\circ\,\mathscr{Z}_{\f}\right>_\kappa\,.
\end{equation}
By \eqref{eqn_2026_01_bigheegnermain_1}, 
\begin{equation}
    \label{eqn_2026_05_30_1753}
\left<\kappa\circ\mathscr{Z}_{\f}\,,\,\kappa\,\circ\,\mathscr{Z}_{\f}\right>_\kappa=\dfrac{\left(1-\frac{p^{\frac{w}{2}}}{\alpha(\kappa)}\right)^{2}}{(2\sqrt{-D_K})^{\frac{w}{2}}}\,\left\langle\kappa\circ \mathscr{Z}_{\f},\,\frac{W_{Np}\circ({\rm Pr}^{\alpha(\kappa)})^*({z_{\f_\kappa^\circ}})}{\alpha(\kappa)\lambda_{N_\f}(\f_\kappa^\circ)\mathcal{E}(\f_\kappa)\mathcal{E}^*(\f_\kappa)} \right\rangle_\kappa\,.
\end{equation}
Furthermore, using \eqref{eqn_2026_01_09_1534} and remembering the self-adjoint involution $\lambda_{N_\f}(\f_\kappa^\circ)^{-1}W_{N_\f}$ that is used to identify $V_{\f_{\kappa^\circ}}^\dagger(\mathscr{H}_{\QQ_p}^\vee)$ and $V_{\f_{\kappa^\circ}}^\dagger(\mathscr{H}_{\QQ_p})$  (cf. \S\ref{subsubsec_2026_1_11_1120}), we have
\begin{align}
    \label{eqn_2026_05_30_1754}
    \begin{aligned}
        \left\langle\kappa\circ \mathscr{Z}_{\f},\,\frac{W_{Np}\circ({\rm Pr}^{\alpha(\kappa)})^*({z_{\f_\kappa^\circ}})}{\alpha(\kappa)\lambda_{N_\f}(\f_\kappa^\circ)\mathcal{E}(\f_\kappa)\mathcal{E}^*(\f_\kappa)} \right\rangle_\kappa&=\alpha(\kappa)\cdot\dfrac{h_{\f_\kappa^\circ}\left({\rm Pr}^{\alpha(\kappa)}_*\circ\kappa\circ \mathscr{Z}_{\f},\, W_N(z_{\f_\kappa^\circ}) \right)}{\alpha(\kappa)\lambda_{N_\f}(\f_\kappa^\circ)^2\mathcal{E}(\f_\kappa)\mathcal{E}^*(\f_\kappa)}\\
        &=\dfrac{h_{\f_\kappa^\circ}\left({\rm Pr}^{\alpha(\kappa)}_*\circ\kappa\circ \mathscr{Z}_{\f},\, W_N(z_{\f_\kappa^\circ}) \right)}{\lambda_{N_\f}(\f_\kappa^\circ)^2\mathcal{E}(\f_\kappa)\mathcal{E}^*(\f_\kappa)}\,.
    \end{aligned}
\end{align}
We remark that the multiplier $\alpha(\kappa)$ on the right-hand side of the first row of \eqref{eqn_2026_05_30_1754} is present because of the modifications on the Poincar\'e duality pairings; cf. \S\ref{subsubsec_213_2025_12_09_1556}. We next use \eqref{eqn_2026_01_bigheegnermain_2} to compute the expression on the denominator of the second row of \eqref{eqn_2026_05_30_1754}:
\begin{equation}
    \label{eqn_2026_05_30_1807}
h_{\f_\kappa^\circ}\left({\rm Pr}^{\alpha(\kappa)}_*\circ\kappa\circ \mathscr{Z}_{\f},\, W_N(z_{\f_\kappa^\circ}) \right)=\dfrac{\left(1-\frac{p^{\frac{w}{2}}}{\alpha(\kappa)}\right)^{2}}{(2\sqrt{-D_K})^{\frac{w}{2}}}\,\dfrac{h_{\f_\kappa^\circ}\left(W_N(z_{\f_\kappa^\circ}),\, W_N(z_{\f_\kappa^\circ}) \right)}{\alpha(\kappa)\lambda_{N_\f}(\f_\kappa^\circ)}\,.
\end{equation}
Combining \eqref{eqn_2026_05_30_1753}, \eqref{eqn_2026_05_30_1754}, and \eqref{eqn_2026_05_30_1807}, we conclude that 
\begin{equation}
    \label{eqn_2026_05_30_1812}
\kappa\,\circ\, \left<\mathscr{Z}_{\f},\mathscr{Z}_{\f}\right>=\dfrac{\left(1-\frac{p^{\frac{w}{2}}}{\alpha(\kappa)}\right)^{4}}{(4|D_K|)^{\frac{w}{2}}}\dfrac{h_{\f_\kappa^\circ}\left(W_N(z_{\f_\kappa^\circ}),\, W_N(z_{\f_\kappa^\circ}) \right)}{\alpha(\kappa)\lambda_{N_\f}(\f_\kappa^\circ)^3\mathcal{E}(\f_\kappa)\mathcal{E}^*(\f_\kappa)}\,.
\end{equation}
Finally, using the fact that $\lambda_{N_\f}(\f_\kappa^\circ)^{-1}W_{N_\f}$ is a self-adjoint involution, we observe that
\begin{equation}
    \label{eqn_2026_05_30_1814}
h_{\f_\kappa^\circ}\left(\lambda_{N_\f}(\f_\kappa^\circ)^{-1} W_N(z_{\f_\kappa^\circ}),\, \lambda_{N_\f}(\f_\kappa^\circ)^{-1} W_N(z_{\f_\kappa^\circ}) \right)=h_{\f_\kappa^\circ}\left(z_{\f_\kappa^\circ},\, z_{\f_\kappa^\circ} \right)\,.
    \end{equation}
    The proof of Part (i) follows from \eqref{eqn_2026_05_30_1812} combined with \eqref{eqn_2026_05_30_1814}.

\begin{comment}
    the following chain of equalities:
    \begin{align*}
\kappa\,\circ\, \left<\mathscr{Z}_{\f},\mathscr{Z}_{\f}\right>&=\left<\kappa\circ\mathscr{Z}_{\f}\,,\,\kappa\,\circ\,\mathscr{Z}_{\f}\right>_\kappa\\
&=\dfrac{\left(1-\frac{p^{\frac{w}{2}}}{\alpha(\kappa)}\right)^{2}}{(2\sqrt{-D_K})^{\frac{w}{2}}}\,\left\langle\kappa\circ \mathscr{Z}_{\f},\,\frac{W_{Np}\circ({\rm Pr}^{\alpha(\kappa)})^*({z_{\f_\kappa^\circ}})}{\alpha(\kappa)\lambda_{N_\f}(\f_\kappa^\circ)\mathcal{E}(\f_\kappa)\mathcal{E}^*(\f_\kappa)} \right\rangle_\kappa\\
&=\alpha(\kappa)\cdot\dfrac{\left(1-\frac{p^{\frac{w}{2}}}{\alpha(\kappa)}\right)^{2}}{(2\sqrt{-D_K})^{\frac{w}{2}}}\dfrac{h_{\f_\kappa^\circ}\left({\rm Pr}^{\alpha(\kappa)}_*\circ\kappa\circ \mathscr{Z}_{\f},\, W_N(z_{\f_\kappa^\circ}) \right)}{\alpha(\kappa)\lambda_{N_\f}(\f_\kappa^\circ)^2\mathcal{E}(\f_\kappa)\mathcal{E}^*(\f_\kappa)}\\
&=\dfrac{\left(1-\frac{p^{\frac{w}{2}}}{\alpha(\kappa)}\right)^{4}}{(4|D_K|)^{\frac{w}{2}}}\dfrac{h_{\f_\kappa^\circ}\left(W_N(z_{\f_\kappa^\circ}),\, W_N(z_{\f_\kappa^\circ}) \right)}{\alpha(\kappa)\lambda_{N_\f}(\f_\kappa^\circ)^3\mathcal{E}(\f_\kappa)\mathcal{E}^*(\f_\kappa)}\\
&=\dfrac{\left(1-\frac{p^{\frac{w}{2}}}{\alpha(\kappa)}\right)^{4}}{(4|D_K|)^{\frac{w}{2}}}\dfrac{h_{\f_\kappa^\circ}\left(z_{\f_\kappa^\circ},\, z_{\f_\kappa^\circ} \right)}{\alpha(\kappa)\lambda_{N_\f}(\f_\kappa^\circ)\mathcal{E}(\f_\kappa)\mathcal{E}^*(\f_\kappa)}
\end{align*}
\end{comment} 

For $\kappa$ as in (ii), the proof follows from \eqref{eqn_2026_01_09_1406} combined with \eqref{eqn_2026_01_bigheegnermain_3}.
\end{proof}

\subsubsection{Anticyclotomic variation} One may vary the Heegner classes $\mathscr{Z}_\f$ along the anticyclotomic tower of $K$. This is, in fact, achieved by interpolating generalised Heegner cycles. The following is a restatement (in a vague form that suffices for our purposes) of the constructions in \cite[\S3.3]{howard2007Inventiones} combined with the main results of \cite{castella13} (when $\f$ has slope zero), and \cite[Theorem 5.4.1]{JLZ} (when $\f$ is of arbitrary finite slope).
\begin{proposition}
Under \eqref{item_Heeg}, there exists a non-trivial class 
$$\widetilde{\mathfrak{Z}}_\infty \in H^1_{\rm f}(G_{K,S}, V_\f^\dagger \hatotimes \mathcal{H}^\iota(\Gamma_{\rm ac}))$$ 
that interpolates the generalised Heegner cycles associated to the pairs 
$(\f_\kappa,\chi)$ where $\kappa$ is a classical specialisation of even weight $w\geq 0$ and $\chi$ is a Hecke character of $p$-power conductor and infinity type $(-j,j)$ with $|j|\leq\frac{w}{2}$, and such that $\chi^c=\chi^{-1}$. If the family $\f$ has slope zero, then in fact
$$\widetilde{\mathfrak{Z}}_\infty \in H^1(G_{K,S}, V_\f^\dagger\,\widehat\otimes\,\LL^\iota(\Gamma_\ac))\,.$$
\end{proposition}

Let $\mathfrak{Z}_\infty^\tau \in  H^1_{\rm f}(G_{K,S}, V_\f^\dagger \hatotimes \mathcal{H}^\iota_\tau(\Gamma_{\rm ac}^\circ)) $ be the image of $\widetilde{\mathfrak{Z}}_\infty$ under the map $\mathcal{H}(\Gamma_{\rm ac}) \twoheadrightarrow \mathcal{H}_\tau(\Gamma^\circ_{\rm ac})$, and let us write $\mathfrak{Z}_\infty$ in the case when $\tau=\mathds{1}$, then $\mathfrak{Z}_\infty$ specialises to the class $\mathscr{Z}_\f$ of Theorem~\ref{thm_bigheegnermain}  under the augmentation morphism $\mathcal{H}(\Gamma_{\rm ac}) \twoheadrightarrow \mathcal{H}(\Gamma^\circ_\ac)\to \ZZ_p$.

\subsubsection{Reciprocity law for $\mathfrak{Z}_\infty^\tau$}
Let us consider the $2$-variable $p$-adic L-function $\widetilde{\mathfrak{L}}_\fp(\Bf) \in \Lambda_\Bf \hatotimes \mathcal{H}(\Gamma_{\rm ac})$, which is constructed in  \cite[Theorem 7.3.3]{JLZ} (denoted by $L_\p^{\rm BDP}(\mathfrak{N}, \varepsilon)$ in op. cit.). We write \[\mathfrak{L}_{\fp}^\tau(\Bf) \in \Lambda_\Bf \hatotimes \mathcal{H}_{\tau}(\Gamma^\circ_{\rm ac}) \simeq \Lambda_\Bf \hatotimes \mathcal{H}(\Gamma^\circ_{\rm ac})\] for the image of $\widetilde{\mathfrak{L}}_{\fp}$ under the projection $\mathcal{H}(\Gamma_{\rm ac}) \twoheadrightarrow \mathcal{H}_{\tau}(\Gamma^\circ_{\rm ac})$.   
\begin{theorem}[Reciprocity law for $\mathfrak{Z}^\tau_\infty$] \label{reciprocity_Z_inf}
    The following equality holds in $\Lambda_\Bf \hatotimes \mathcal{H}(\Gamma^\circ_{\rm ac}):$
    \begin{align}
        \mathscr{L}_{\omega_\Bf}^{\tau}({\rm res}_\fp(\mathfrak{Z}^{\tau}_\infty)) = (-1)^{\frac{\Bk}{2}+\Bj}\mathfrak{L}_{\fp}^ {\tau}(\Bf)\,, \notag
    \end{align}
    where the universal anticyclotomic character $\Gamma^\circ_{\ac} \hookrightarrow (\Lambda_\f \hatotimes \mathcal{H}_\tau(\Gamma^\circ_\ac))^\times $ is considered to have ``infinity type'' $(\Bj, -\Bj)$.
 \end{theorem}
\begin{proof}
This is \cite[Theorem 8.2.4]{JLZ}, see also Remark \ref{remark_embedding_into_ac}.
\end{proof}

\section{\texorpdfstring{The $p$-adic Gross--Zagier formula}{The p-adic Gross--Zagier formula}}
\label{sec_main_GZ}
For the rest of this section we assume that $f \in S_{k+2}(\Gamma_0(N))$ for some integer $N\geq 5$. Let $\alpha, \beta$ denote the (distinct by assumption) roots of the Hecke polynomial of $f$, with $\text{ord}_p(\alpha)< {\rm ord}_p(\beta)$. If $p\nmid N$ (resp. $p\mid \mid N$) we let $f_\alpha = f(q) - \beta \cdot f(q^p)$ denote the $p$-stabilisation of $f$ at $\alpha$ (resp. let $f_\alpha = f$). Let $\Bf$ be the Coleman (or Hida) family passing through $f_\alpha$ and assume \eqref{item_Heeg} holds.
\subsection{A Rubin-style formula}
\label{subsec_2026_02_06_0827}
Our main goal in this subsection is to compute, \`a la Rubin, the $p$-adic heights of Beilinson--Flach elements. For the purposes of \S\ref{subsec_2026_02_06_0827}, the weaker assumption \eqref{item_Sign} suffices in place of \eqref{item_Heeg}. 
\begin{defn}
    For $\bullet = \{\}, \, c$,  let $$\bD^\dagger_{\rm rig}(\mathcal{H}_\tau(\bm{\Gamma})) \in \{\bD^\dagger_{\rm rig}(\mathcal{H}_\tau^\iota(\Gamma^\circ_{\rm ac}), \, \bD^\dagger_{\rm rig}(\mathcal{H}_\tau^\sharp(\Gamma^\circ_\fp)^\bullet \hatotimes \bD^\dagger_{\rm rig}(\mathcal{H}^\iota(\Gamma^\circ_{\rm cyc}) \}.$$
        We will write
        \begin{align}
            \pi_{-}^{(\fp)}: H^1\left(K_\fp, \Ddagrigf(V^\dagger_\Bf) \hatotimes \bD^\dagger_{\rm rig}(\mathcal{H}_\tau(\bm{\Gamma})) \right) \lra H^1 \left(K_\fp, \DD^{\dagger-}_\Bf \hatotimes \bD^\dagger_{\rm rig}(\mathcal{H}_\tau(\bm{\Gamma})) \right), \notag
        \end{align}
        for the natural map arising from the projection $\Ddagrigf(V^\dagger_\Bf) \twoheadrightarrow \DD^{\dagger-}_\Bf$.
    %\end{enumerate}
\end{defn}
\begin{lemma}\label{lemma_BF_vanish}
    $_c{\rm BF}^{\tau}_{\rm ac, \Bf} \in H^1_{\rm f}(G_{K,S},V^\dagger_\Bf\hatotimes \mathcal{H}_\tau^\iota(\Gamma^\circ_{\rm ac}))$.
\end{lemma}
\begin{proof}
    By Proposition \ref{proposition_BF_properties}(c) and the identification \eqref{--=fp--}, it suffices to show that 
    $${_c{\rm BF}^{\tau,-}_{\rm ac, \fp}}:=\pi_{-}^{(\fp)}(\rm{res}_\fp(_c{\rm BF}^{\tau}_{\rm ac, \Bf}))=0\,,$$
where ${_c{\rm BF}^{\tau,-}_{\rm ac, \fp}}$ is the class introduced in Definition~\ref{defn_2026_02_04_1426}. This follows by Corollary \ref{corollary_reciprocity_for_anticyclotomic_BF}(i) and \eqref{item_Sign}. Indeed, we first note that $\mathfrak{L}_{\rm ac, 0}^\Bf(\Bf_\kappa,\tau)$ is trivial for a dense set of specialisations $\kappa$ by \eqref{item_Sign} and the interpolation property of the $p$-adic $L$-function (cf. \cite[Theorem 2.7.4]{KLZ2}). This shows that $\mathfrak{L}_{\rm ac, 0}^\Bf(\Bf, \tau) = 0$, and the result follows since the map $\mathscr{L}^\tau_{\eta_\Bf}$ is injective.
    \end{proof}
\begin{remark} \label{remark_simpler_description_of_BF^+_p}
\normalfont
    Recall the isomorphisms \[\Gamma^\circ_{\rm cyc} \times \Gamma^\circ_\fp \simeq \Gamma^+\times \Gamma^- \simeq \Gamma^\circ_{\rm cyc} \times \Gamma^\circ_{\fp^c},\]
    from \S \ref{subsec_Iwasawa_Algebras}, which induce the isomorphisms
    \[\mathcal{H}_{\tau^c}^{\sharp}(\Gamma^\circ_\fp) \hatotimes \mathcal{H}^\iota(\Gamma^\circ_{\rm cyc})/(\gamma^+-1) \simeq \mathcal{H}_{\tau^{-1}}^\iota(\Gamma^\circ_{\rm ac}) \simeq \mathcal{H}^\sharp_{\tau}(\Gamma^\circ_\fp)^c \hatotimes 
    \mathcal{H}^\iota(\Gamma^\circ_{\rm cyc}) /(\gamma^+-1).\]
    In particular, we have that 
    \begin{align*}
        &{\rm im}\left(H^1\left(K_{\fp}, \DD^{\dagger+}_\Bf \hatotimes\bD^\dagger_{\rm rig}(\mathcal{H}^{\sharp}_{\tau}(\Gamma^\circ_\fp)^c)\hatotimes \bD^\dagger_{\rm rig}(\mathcal{H}^\iota(\Gamma^\circ_{\rm cyc})) \right)/(\gamma^+-1) \hookrightarrow H^1 \left (K_\fp, \DD^{\dagger+}_\Bf \hatotimes \bD^\dagger_{\rm rig}( \mathcal{H}_{\tau^{-1}}^\iota(\Gamma^\circ_{\rm ac}))\right) \right) \notag \\
        & = \ {\rm im}\left(H^1 \left(K_\fp, \DD^{\dagger+}_\f \hatotimes \bD^\dagger_{\rm rig}(\mathcal{H}^\sharp_{\tau^c}(\Gamma^\circ_\p)) \hatotimes \bD^\dagger_{\rm rig}(\mathcal{H}^\iota(\Gamma^\circ_{\rm cyc}))\right)/(\gamma^+-1)  \hookrightarrow H^1 \left(K_\fp, \DD^{\dagger+}_\Bf \hatotimes \bD^\dagger_{\rm rig}(\mathcal{H}_{\tau^{-1}}^\iota(\Gamma^\circ_{\rm ac}))\right) \right)\,.
    \end{align*} 
    Since ${_c{\rm BF}^{\tau^{-1},-}
    _{\rm ac, \fp}} = 0$, and recalling that $\tau \in \widehat{\Delta}_{\ac}$ satisfies $ \tau^c = \tau^{-1}$, we have the equality 
    \[_c{\rm BF}^{\tau^{-1}, +}_{\rm ac, \fp} ={\rm res}_\fp({_c{\rm BF}^{\tau^{-1}}_{\rm ac, \Bf}})\,,\]
 where ${_c{\rm BF}^{\tau^{-1},+}_{\rm ac, \fp}}$ is given as in Definition~\ref{defn_2026_02_04_1426}.  \hfill $\triangle$
    \end{remark}

We recall the fixed a topological generator $\gamma_\cyc$ of $\Gamma^\circ_\cyc$.
\begin{lemma}
\label{lemma_2026_05_28_0843}
    There exists an element 
    $$\mathfrak{d}^{(\fp)}_{\rm{cyc}} {\rm{\bf BF}}^\tau_{\Bf} \in H^1\left(K_\fp, \DD^{\dagger-}_\Bf\hatotimes \bD^\dagger_{\rm rig}(\mathcal{H}_\tau^{\sharp}(\Gamma^\circ_\fp)) \hatotimes \bD^\dagger_{\rm rig}(\mathcal{H}^\iota(\Gamma^\circ_{\rm cyc})) \right)$$ 
    satisfying the equality
    \[\pi_{-}^{(\fp)}({\rm res}_\fp(_c{\rm{\bf BF}}^\tau_{\Bf})) = \frac{\gamma_{\rm cyc}-1}{\log_p(\chi_{\rm cyc}(\gamma_{\rm cyc}))} \cdot \mathfrak{d}^{(\fp)}_{\rm{cyc}} {\rm{\bf BF}}^\tau_{\Bf}\]
    in $H^1\left(K_\fp, \DD^{^\dagger -}_\Bf\otimes \bD^\dagger_{\rm rig}(\mathcal{H}_\tau^{\sharp}(\Gamma^\circ_{\fp})) \hatotimes \bD^\dagger_{\rm rig}(\mathcal{H}^\iota(\Gamma^\circ_{\rm cyc})) \right)$.\\
\end{lemma}

We remark that, thanks to the denominator $\log_p(\chi_{\rm cyc}(\gamma_{\rm cyc}))$ in the statement above, the derived class $\mathfrak{d}^{(\fp)}_{\rm{cyc}} {\rm{\bf BF}}^\tau_{\Bf}$ given as above does not depend on the choice of $\gamma_{\rm cyc}$.
Similar discussion also applies in \eqref{eq_expansion_of_l_function} above, pertaining to the definition of $\mathfrak{L}^{\Bh}_{{\rm ac},1}(\Bf,\tau)$.

\begin{proof}[Proof of Lemma~\ref{lemma_2026_05_28_0843}]
    This is a special case of Theorem \ref{thm_RSformula_cyclo_height_2025_12_14}(i), with $[z_f]$ the element corresponding to $_c\rm{BF}^\tau_{{\rm ac},\Bf}$ and $[\widetilde{z}] = {_c{\rm{\bf BF}}_\Bf^\tau}$.
\end{proof}

\begin{defn}
\label{def_Tate_pairing}
    Let $D$ denote some $(\varphi, \Gamma_{\rm cyc})$--module over $\mathcal{R}_A$ for some affinoid algebra $A$, and let $D^*(1) := \Hom(D, \bD^\dagger_{\rm rig}(A)(1))$.
    \item[i)] We define the cyclotomic Tate pairing 
    \begin{align}
    \begin{aligned}
    \label{eqn_2026_01_24_1009}
                \langle\,,\,\rangle_n \,:\, H^1_{\rm Iw} \left(K_{\fp, \infty}, D \hatotimes \bD^\dagger_{\rm rig}(\mathcal{H}^\sharp_{n, \tau}(\Gamma^\circ_{\p}))  \right)\, &\otimes_{\Lambda_\Bf}\, H^1_{\rm Iw} \left(K_{\fp, \infty}, D^*(1) \hatotimes \bD^\dagger_{\rm rig}(\mathcal{H}^\iota_{n, \tau^{-1}}(\Gamma^\circ_{\p})) \right)  \\
        &\hspace{4.2cm}\lra \Lambda_\Bf \hatotimes \mathcal{H}_n(\Gamma^\circ_\p) \hatotimes \mathcal{H}(\Gamma^\circ_{\cyc})
            \end{aligned}
    \end{align}
    as in \cite[Definition 4.2.8]{KPX}, for $M = D \hatotimes \bD^\dagger_{\rm rig}(\mathcal{H}^\sharp_{n,\tau}(\Gamma^\circ_\p))$.
    \item[ii)] Let  us denote by
    \begin{align*}
        \langle\,,\,\rangle^{\rm ac}_{\rm Tate} \,:\, H^1_{\rm Iw} \left(K_{\fp},D \hatotimes \bD^\dagger_{\rm rig}(\mathcal{H}^\sharp_{\tau}(\Gamma^\circ_{\p})) \right)/(\gamma^+-1) \,&\otimes_{\Lambda_\Bf}\, H^1_{\rm Iw} \left(K_{\fp},D^*(1) \hatotimes \bD^\dagger_{\rm rig}(\mathcal{H}^\iota_{\tau^{-1}}(\Gamma^\circ_{\p})) \right)/(\gamma^+-1) \notag \\
        &\hspace{5.2cm}\lra \Lambda_\Bf \hatotimes \mathcal{H}(\Gamma^\circ_{\rm ac}) \notag
    \end{align*}
     the ``anticyclotomic'' Tate pairing, which is the inverse limit $($over $n$$)$ of the pairings \eqref{eqn_2026_01_24_1009} modulo $(\gamma^+-1)$.
\end{defn}
\begin{remark}\label{remark_self_dual}
\normalfont
    \item[i)] We will primarily be interested in the case when $D= \DD^{\dagger-}_{\f}$. Note that $(\DD^{\dagger-}_\f)^*(1) \simeq \DD^{\dagger+}_\f$ since ${V_\Bf^\dagger}$ is self-dual.
    \item [ii)] Reasoning as in Remark \ref{remark_embedding_into_ac}, we may view $\langle\,,\,\rangle^{\rm ac}_{\rm Tate}$ as a morphism on an appropriate submodule of 
    $$H^1 \left(K_{\fp},D \hatotimes \bD^\dagger_{\rm rig}(\mathcal{H}^\iota_{\tau}(\Gamma^\circ_{\ac})) \right) \,\otimes_{\Lambda_\Bf}\, H^1 \left(K_{\fp},D^*(1) \hatotimes \bD^\dagger_{\rm rig}(\mathcal{H}^\iota_{\tau^{-1}}(\Gamma^\circ_{\ac})) \right)\,.$$    
\end{remark}
Since $\mathfrak{Log}^\Gamma_{?}$ interpolates Nakamura's logarithm and dual exponential maps, we have the following:

\begin{lemma}\label{lemma_comparing_pairings}  Suppose that 
    $$\mathfrak{x} \otimes \mathfrak{y} \in H^1 \left(K_{\fp}, \DD^{\dagger-}_{\Bf} \hatotimes \bD^\dagger_{\rm rig}(\mathcal{H}^\iota_{\tau}(\Gamma^\circ_\ac))\right) \otimes H^1 \left(K_{\fp}, \DD^{\dagger+}_{\Bf} \hatotimes \bD^\dagger_{\rm rig}(\mathcal{H}^\iota_{\tau^{-1}}(\Gamma^\circ_\ac))\right)$$ 
    live in the submodule specified in Remark~\ref{remark_self_dual}. Then we have the equality 
\begin{equation}
    \langle \mathfrak{Log}^{\rm ac}_{-+}(\mathfrak{x}) , \mathfrak{Log}^{\rm ac}_{+-}(\mathfrak{y})\rangle_{\DD} = -(-1)^{\frac{\Bk}{2} + \Bj} \cdot \langle \mathfrak{x} , \mathfrak{y} \rangle_{\rm Tate}^{\rm ac}. \notag
\end{equation}

\end{lemma}
\begin{proof}
    
    It suffices to prove that the claimed identity holds when it is specialised to $E$-valued classical points $\kappa\in \cX$, by the density of such points.
    
    To that end, let $D=\cR_E(\tau\cdot\delta_\cyc^{r})$ be a $(\varphi, \Gamma_\cyc)$-module of rank one over $\cR_E$ (cf. \cite{Colmez_2008} for a classification of such), where $\QQ_p^\times \xrightarrow{\delta_\cyc} E^\times$ is the character $x\mapsto x|x|$, $r$ is an integer (to be specified later, in terms of our $\kappa$), and $\ZZ_p^\times \xrightarrow{\tau} E^\times$ has finite order. %(in the case of interest, this comes from the branch character of $\g_\tau$, but here we absorb it into $D$). 
    Using the identifications \eqref{eq_isom_+_-} and \eqref{eq_isom_star_bullet}, we oberve that specialising $\Lambda_{\g} \hatotimes \mathcal{H}^\iota(\Gamma_\cyc^\circ) \cong \mathcal{H}(\Gamma_\p^\circ) \hatotimes \mathcal{H}^\iota(\Gamma_\cyc^\circ)$ at $( \kappa'_{2j}, \chi_\cyc^j)$, where $\kappa'_{2j}$ is a crystalline weight-$2j$ specialisation of $\LL_\g$ for some integer $j$, factors through $\mathcal{H}^\iota(\Gamma_\ac^\circ)$. Moreover, such specialisations are dense in ${\rm Spm}\,\mathcal{H}(\Gamma^\circ_\ac)$. 
    
    By the discussion above, we infer that for $\mathfrak{c} \in H^1 \left(K_{\fp}, D \hatotimes \bD^\dagger_{\rm rig}(\mathcal{H}^\iota(\Gamma^\circ_\ac))\right)$ and any lift $$\widetilde{\mathfrak{c}} \in H^1 \left(K_{\fp}, D \hatotimes \bD^\dagger_{\rm rig}(\mathcal{H}^{\sharp}(\Gamma^\circ_{\fp})) \hatotimes \bD^\dagger_{\rm rig}(\mathcal{H}^\iota(\Gamma^\circ_{\rm cyc})) \right)$$
    of $\mathfrak{c}$, the composition
    $$ \mathfrak{c} \longmapsto \mathfrak{Log}_D^\ac(\mathfrak{c})\longmapsto \mathfrak{Log}_D^\ac(\mathfrak{c})(j):= (\kappa'_{2j}, \chi_\cyc^j)\circ \mathfrak{Log}_D^\Gamma(\widetilde{\mathfrak{c}}), $$
    is well-defined, 
 where $\mathfrak{Log}_D^\Gamma$ is induced by the Perrin-Riou logarithm map defined in \cite[Theorem 7.1.4]{LZ1} (see Definition~\ref{defn_2_var_log}), and $\mathfrak{Log}_D^\ac$ is, as in Definition~\ref{def_2026_06_03_1051} above, $\mathfrak{Log}_D^\Gamma \pmod{\gamma_+-1}$. Let $\mathfrak{c}_{0,j}$ denote the image of $\mathfrak{c}$ under the canonical (augmentation) map, composed with the twisting morphism ${\rm Tw}_{\chi^j_\cyc}$. Then from \cite[Theorem 7.1.4]{LZ1} (see also \cite[Propositions 8.11 \& 8.14]{venjakob_epsilon}, \cite[Proposition 4.16]{Nakamuraepsilon} and \cite[\S B.5]{LZ0}), we have
    \begin{equation}
        \mathfrak{Log}_D^\ac(\mathfrak{c})(j) = (*)\begin{cases}
            \frac{(-1)^{r-j-1}}{(r-j-1)!}\cdot \log(\mathfrak{c}_{0,j}), \quad &\text{if }j-r <0\,, \\
            (j-r)! \cdot \exp^*(\mathfrak{c}_{0,j}),\quad &\text{if }j-r \geq 0\,,
        \end{cases}
        \label{eq_Log_interpolation_property}
    \end{equation}
    where $$(*) = \begin{cases}
(1 - \varphi p^{j})\left(1 -{\varphi^{-1}}{p^{-1-j}}\right)^{-1}, & \text{if } \tau \text{ is unramified},  \\
({p^{1+j}}{\varphi})^{\rm{cond}(\tau)} G(\tau)^{-1}, & \text{if } \tau \text{ is ramified}.
\end{cases}$$
Here $G(\tau)$ is the Gauss sum as in op. cit. and ${\rm cond}(\tau)$ is the conductor of $\tau$. Then \eqref{eq_Log_interpolation_property}, together with the same point-wise computation as in the proof of \cite[Theorem II.14]{berger03} (since $\exp$ and $\exp^*$  are adjoints, by definition)
%is \cite[Proposition 2.16]{nakamura2014Jussieu}) 
yields the equality 
$$\langle \mathfrak{Log}^\ac_D(\mathfrak{x}) , \mathfrak{Log}^\ac_{D^*(1)}(\mathfrak{y}) \rangle_\DD = -(-1)^{r- \Bj-1} \cdot \langle \mathfrak{x}, \mathfrak{y}\rangle_{\rm Tate}^\ac.$$
The specialisation of the claimed statement of our lemma to classical points $\kappa$, hence the lemma itself, follows on setting $D = \DD^{\dagger -}_{\f_\kappa}\otimes \tau$ and $D^*(1) = \DD^{\dagger +}_{\f_\kappa} \otimes \tau^{-1}$ (where, denoting by $k$ the weight of $\kappa$, we set $r = 1+\frac{k}{2}$ in this case).
\end{proof}

\begin{proposition}[Rubin-Style formula]\label{Rubin_formula}
    We set $\mathfrak{d}^{(\fp)}_{\rm{cyc}} {\rm BF}^{\tau}_{\rm ac, \Bf} \in H^1 \left(K_{\fp}, \DD^{\dagger-}_\f \hatotimes \bD^\dagger_{\rm rig}(\mathcal{H}^\iota_\tau(\Gamma^\circ_{\rm ac})) \right)$ to be the image of $\mathfrak{d}^{(\fp)}_{\rm{cyc}} {\rm{\bf BF}}_\Bf^\tau$ under the map induced on cohomology by \eqref{eq_ses_gamma_gamma_gammaac_first}. We have the equality 
    \[\left\langle _c{\rm BF}^{\tau}_{\rm ac, \Bf}, {_c{\rm BF}^{\tau^{-1}}_{\rm ac, \Bf}} \right\rangle = -\left\langle \mathfrak{d}^{(\fp)}_{\rm{cyc}}{\rm BF}^{\tau}_{\rm ac, \Bf} , {_c{\rm BF}_{\rm ac, \fp}^{\tau^{-1}, +}} \right\rangle_{\rm Tate}^{\rm ac}\]
    in $\Lambda_\Bf \hatotimes \mathcal{H}(\Gamma^\circ_{\rm ac})$.
\end{proposition}
\begin{proof}
    This is a direct consequence of Corollary \ref{cor_thm_RSformula_cyclo_height_2025_12_15}, with the choices 
    $$[z_f] = {_c{\rm BF}^\tau_{\ac, \f}},\quad [y_f] = {_c{\rm BF}^{\tau^{-1}}_{\ac, \f}},\quad [\widetilde{z}] = {_c\mathbf{BF}^\tau_\f}\,.$$ 
Note that $S_{\rm cris}=\{\fp^c\}$, and the additional condition that is stated at the beginning of Corollary \ref{cor_thm_RSformula_cyclo_height_2025_12_15} is satisfied in our case by  Proposition \ref{proposition_BF_properties}(c) and the identification \eqref{--=fp--}.
\end{proof}

\subsection{\texorpdfstring{$p$-adic Gross--Zagier formula}{The ordinary p-adic Gross--Zagier formula}} \label{subsec_GZ_formula}
We will now put the pieces together and prove the main results of our paper. We will assume that  \eqref{item_Disc}, \eqref{item_Heeg}, and \eqref{item_BI} hold until the end of this section.
\subsubsection{Comparison of big Generalised Heegner cycles and big Beilinson--Flach elements} 
For notational convenience, let us set 
$$\mathcal{H}_{\rm ac, \Bf}:= \Lambda_\Bf \hatotimes \mathcal{H}(\Gamma^\circ_{\rm ac})\,, \quad \mathcal{H}^{\tau, \iota}_{{\rm ac},\Bf}:= \Lambda_\Bf \hatotimes \mathcal{H}_\tau^\iota(\Gamma^\circ_{\rm ac})\,,\quad V^{\dagger, \tau}_{\Bf, \, \rm ac}:= V^\dagger_\Bf \hatotimes \mathcal{H}^\iota_\tau(\Gamma^\circ_{\rm ac})\,,\quad  \DD^{\pm, \dagger,\tau}_{\f, \, \rm ac} := \DD^{\pm, \dagger}_\f \hatotimes \bD^\dagger_{\rm rig}(\mathcal{H}^\iota_{\tau}(\Gamma^\circ_{\ac}))\,.$$ 
We also define $\mathfrak{Z}^\tau_{\infty,\fp} :=\rm{res}_\fp(\mathfrak{Z}^\tau_\infty) \in H^1(K_\fp, \DD^{+,\dagger, \tau}_{\f,\, \ac})$.
\begin{proposition} \label{proposition_Selmer_rank_1}
    Suppose \eqref{item_Disc}, \eqref{item_Heeg}, and \eqref{item_BI} hold. Assume also that $p\nmid 6h_K$, then 
    $$\rm{rank}_{\mathcal{H}_{\rm ac, \f}} \left(H^1_{\rm f}(G_{K,S}, V^{\dagger, \tau}_{\f, \ac}) \right)=1\,,\qquad \forall\,\tau \in \widehat{\Delta}_\ac\,. $$
    \end{proposition}
\begin{proof}
    Since $\mathfrak{Z}^\tau_\infty \in H^1_{\rm f} (G_{K,S}, V^{\dagger, \tau}_{\f, \ac})$ in non-torsion under the current assumptions, we have that 
    \begin{equation}
        \rm{rank}_{\mathcal{H}_{\rm ac, \f}} \left(H^1_{\rm f}(G_{K,S}, V^{\dagger, \tau}_{\f, \ac}) \right)\geq 1. \label{equation_inequality_>1}
    \end{equation}
    On the other hand, let $\f_\kappa \in S_w(\Gamma_0(N_\f))$, with $w \in 2\ZZ^{\geq 1}$, be some crystalline specialisation of $\f$ such that $\mathfrak{Z}^\tau_{\infty, \, \kappa}$ is a non-torsion specialisation of $\mathfrak{Z}^\tau_{\infty}$. We will prove that $\rm{rank}_{\mathcal{H}(\Gamma^\circ_{\rm ac})}\left(H^1_{\rm f} (G_{K,S}, V^{\dagger, \tau}_{\f_\kappa, \ac}  \right) \leq 1$, which suffices to show the reverse inequality to \eqref{equation_inequality_>1} (after possibly shrinking $U_\f$). 
    
    Let $S^\bullet(V^{\dagger,\tau}_{\f_\kappa, \ac}, \bD_{\p^c})$ denote the Selmer complex constructed by considering no local conditions at $\p$, and the local condition given by $\DD^{\dagger+,\tau}_{\f, \ac}$ at $\p^c$. We also set $S^\bullet(V^{\dagger, \tau}_{\f_\kappa, \ac}, \bD_{\emptyset, 0})$ to be the Selmer complex constructed by considering no condition at $\p$ and the ``strict'' local condition at $\p^c$ given by the trivial submodule of $\bD^\dagger_{\rm rig}(V^{\dagger, \tau}_{\f_\kappa, \ac})$. Let us write $H^i_{\fp^c}(G_{K,S},V^{\dagger, \tau}_{\f_\kappa, \ac} )$ (resp. $H^i_{\emptyset, 0}(G_{K,S},V^{\dagger, \tau}_{\f_\kappa, \ac})$) for the $i$'th cohomology group of $S^\bullet(V^{\dagger,\tau}_{\f_\kappa, \ac}, \bD_{\p^c})$ (resp. of $S^\bullet(V^{\dagger, \tau}_{\f_\kappa, \ac}, \bD_{\emptyset, 0})$). It clearly suffices to show that the rank of $H^1_{\p^c}(G_{K,S}, V^{\dagger, \tau}_{\f_\kappa, \ac})$ is bounded above by one.
    
    From the definitions of $S^\bullet(V^{\dagger,\tau}_{\f_\kappa, \ac}, \bD_{\p^c})$ and $S^\bullet(V^{\dagger, \tau}_{\f_\kappa, \ac}, \bD_{\emptyset, 0})$, we obtain the embedding
    $$ \frac{H^1_{\fp^c}(G_{K,S},V^{\dagger, \tau}_{\f_\kappa, \ac} )}{H^1_{\emptyset, 0} (G_{K,S},V^{\dagger, \tau}_{\f_\kappa, \ac} )} \hookrightarrow H^1(K_{\p^c}, \DD^{+, \dagger, \tau}_{\f_\kappa, \rm ac} ) \,.$$
    We are therefore reduced to showing that $H^1_{\emptyset, 0}(G_{K,S},V^{\dagger, \tau}_{\f_\kappa, \ac} )$ is a torsion $\mathcal{H}(\Gamma^\circ_\ac)$-module. 
    This follows under our current assumptions from the main theorem of \cite{kobayashiota}, which tells us that 
    \begin{equation}
        \rm{rank}_{\Lambda(\Gamma^\circ_\ac)} \left(H^1_{
        0, \emptyset}(\Gal(K_S/K_{\ac}), W_{\f_\kappa}^{\dagger, \tau} )^\vee\right) =0, \label{eq_Kobayashi_Ota}
    \end{equation}
    where $K_S$ is the maximal unramified outside $S$ extension of $K$, $K_{\ac}$ is the field corresponding to $\Gamma_\ac^\circ$, $W_{\f_\kappa}^{\dagger, \tau} = V_{\f_\kappa}^{\dagger, \tau}/T^{\dagger, \tau}_{\f_\kappa}$ for some fixed $G_\Q$-stable lattice $T_{\f_\kappa}^{\dagger, \tau} \subset V^{\dagger, \tau}_{\f_\kappa}$, and $(-)^\vee$ denotes the Pontryagin duality functor. By Nekov\'a\v{r}'s global duality \cite[Theorem 8.9.12]{nekovar06}, we have an isomorphism
    \begin{equation}
        H^1_{0,\emptyset}(\Gal(K_S/K_{\ac}), W_{\f_\kappa}^{\dagger, \tau})^\vee \cong H^2_{\emptyset,0}(G_{K,S}, T_{\f_\kappa}^{\dagger, \tau, \star}(1)\hatotimes \Lambda(\Gamma^\circ_\ac)) \cong H^2_{\emptyset, 0}(G_{K,S}, T_{\f_\kappa}^{\dagger, \tau}\hatotimes \Lambda(\Gamma^\circ_\ac)), \label{eq_isom_selmer}
    \end{equation}
    where the last isomorphism follows by self-duality of $T^\dagger_{\f_\kappa}$
    (note that the local conditions $\{\emptyset, 0 \}$ and $\{0, \emptyset\}$ make sense in Nekov\'a\v{r}'s formalism). For the Selmer complex given by the local condition $\{\emptyset, 0\}$ at $p$, the global Euler-Poincar\'e characteristic formula (\cite[Theorem 8.9.15]{nekovar06}) reads 
    $$\sum_{i=0}^2 (-1)^i{\rm rank}_{\Lambda(\Gamma^\circ_\ac)}\left(H^i_{\emptyset, 0}(G_{K,S}, T^{\dagger,\tau}_{\f_\kappa} \hatotimes \Lambda(\Gamma^\circ_\ac)\right) = 0,$$ 
    where we recall that our local conditions at $\p$ (resp. at $\p^c$) are relaxed (resp. strict) local conditions. Since $H^0_{\emptyset, 0}(G_{K,S}, T^{\dagger, \tau}_{\f_\kappa} \hatotimes \Lambda(\Gamma^\circ_\ac)) = 0$, and $\rm{rank}_{\Lambda(\Gamma^\circ_\ac)}\left (H^2_{\emptyset, 0}(G_{K,S}, T^{\dagger, \tau}_{\f_\kappa} \hatotimes \Lambda(\Gamma^\circ_\ac))\right) = 0$ by \eqref{eq_Kobayashi_Ota} and \eqref{eq_isom_selmer}, we conclude that 
    $$\rm{rank}_{\Lambda(\Gamma^\circ_\ac)}\left (H^1_{\emptyset, 0}(G_{K,S}, T^{\dagger, \tau}_{\f_\kappa} \hatotimes \Lambda(\Gamma^\circ_\ac))\right) = 0.$$
   The result now follows after base change to $\mathcal{H}(\Gamma^\circ_\ac)$ by \cite[Theorem 1.9]{pottharst_analytic_families}.  
\end{proof}
\begin{proposition}
\label{prop_BF=Z}  Under the same hypotheses as in the statement of Proposition \ref{proposition_Selmer_rank_1}, we have the following equality in $H^1_{\rm f}(G_{K,S}, {V^{\dagger, \tau}_{{\rm ac}, \Bf}}) \otimes_{\mathcal{H}_{{\rm ac},\Bf}}\Frac\left((\mathcal{H}_{\rm ac, \Bf})\right)$:
    \[\mathfrak{Z}^\tau_\infty = (-1)^{\frac{\Bk}{2}+\Bj}\cdot\frac{\mathscr{R_{\rm ac}}(\Bg_{\tau^{-1}})}{\mathscr{C}_{\rm ac}(\Bg_{\tau^{-1}})}\cdot \frac{\mathfrak{L}_\fp^{\tau}(\Bf)}{\mathfrak{L}^{\g_{\tau^{-1}}}_{{\rm ac},0}(\Bf, \tau^{-1})}\cdot{_c{\rm BF}^{\tau}_{\rm ac, \Bf}}.\]
    
\end{proposition}
\begin{proof}
     By lemma \ref{lemma_BF_vanish}, we have that $_c{\rm BF}^{\tau}_{\rm ac, \Bf} \in H^1_{\rm f}(G_{K,S}, {V^{\dagger, \tau}_{{\rm ac}, \Bf}})$ and hence, by Propsition \ref{proposition_Selmer_rank_1},  we have the following equality in $H^1_{\rm f}(G_{K,S}, {V^{\dagger, \tau}_{{\rm ac}, \Bf}} )\otimes_{\mathcal{H}_{{\rm ac}, \Bf}} \text{Frac}(\mathcal{H_{\rm ac, \Bf}})$: \[\mathfrak{Z}^\tau_\infty = \lambda \cdot {_c{\rm BF}^{\tau}_{\rm ac, \Bf}}, \ \ \ \lambda \in \text{Frac}(\mathcal{H}_{{\rm ac}, \Bf}). \]
    Applying $\mathscr{L}^{\tau}_{\omega_\Bf} \circ {\rm res}_\p$ on both sides and solving for $\lambda$ then yields the desired result by Corollary \ref{corollary_reciprocity_for_anticyclotomic_BF} and Theorem \ref{reciprocity_Z_inf}.
\end{proof}
\begin{remark}
\normalfont
    If $\f$ is a slope-zero family, we may replace \eqref{item_BI} with the condition that the residual representation of $\f$ is absolutely irreducible. In this case, Proposition \ref{proposition_Selmer_rank_1} is an immediate consequence of \cite[Theorem 12.9.11]{nekovar06} (see also \cite{howard06}, Remark 3.3.2).   \hfill$\triangle$
\end{remark}
\subsubsection{The $p$-adic Gross--Zagier formula}
\begin{theorem}[Universal Gross--Zagier Formula] 
\label{Theorem_Big_GZ}
Suppose \eqref{item_Disc}, \eqref{item_Heeg}, and \eqref{item_BI} hold\footnote{See Footnote~\eqref{footnote_BI} for a comment on the condition \eqref{item_BI}.}, and assume that $p\nmid 6h_K$, then we have the equality
    \[(-1)^{\frac{\Bk}{2}+\Bj}  \cdot \langle \mathfrak{Z}^\tau_\infty, \mathfrak{Z}^{\tau^{-1}}_\infty \rangle = \frac{\lambda_{\Bg_{\tau^{-1}}}}{\lambda_\Bf} \cdot \mathscr{R}_{\rm ac}(\Bg_{\tau^{-1}}) \cdot\frac{\mathfrak{L}_\fp^{\tau} (\Bf) \cdot \mathfrak{L}^{\tau^{-1}}_\fp(\Bf)}{\mathfrak{L}^{\g_{\tau^{-1}} }_{{\rm ac},0}(\Bf,\tau^{-1})} \cdot \mathfrak{L}^{\Bf}_{\rm ac, 1}(\Bf,\tau)\]
\end{theorem}
\begin{proof}
    Let us put 
    $$L=\frac{\mathscr{R_{\rm ac}}(\Bg_{\tau}) \cdot \mathscr{R_{\rm ac}}(\Bg_{\tau^{-1}})}{\mathscr{C}_{\rm ac}(\g_{\tau})\cdot \mathscr{C}_{\rm ac}(\g_{\tau^{-1}})}\cdot \frac{\mathfrak{L}_\fp^{\tau}(\Bf) \cdot \mathfrak{L}_\fp^{\tau^{-1}}(\Bf)}{\mathfrak{L}^{\g_\tau}_{{\rm ac},0}(\Bf,\tau)\cdot \mathfrak{L}^{\g_{\tau^{-1}}}_{{\rm ac},0}(\Bf,\tau^{-1})}\,.$$
    Combining Propositions \ref{Rubin_formula} and \ref{prop_BF=Z}, we infer that \begin{align}
        (-1)^{\frac{\Bk}{2}+\Bj}\cdot\langle \mathfrak{Z}^\tau_\infty, \mathfrak{Z}^{\tau^{-1}}_\infty \rangle &=  -(-1)^{\frac{\Bk}{2}+\Bj} \cdot L\cdot\langle \mathfrak{d}^{(\fp)}_{\rm{cyc}}{\rm BF}^{\tau}_{\rm ac, \Bf} , {_c\rm{BF}^{\tau^{-1}, +}_{\ac, \p}}\rangle_{\rm Tate}^{\rm ac} \notag \\
        &=  L \cdot \mathscr{L}^{\tau^c}_{\omega_\Bf}(_c{\rm BF}^{\tau^{-1}, +}_{\rm ac,\fp})\cdot \langle \mathfrak{Log}^{\rm ac}_{-+} (\mathfrak{d}^{(\fp)}_{\rm{cyc}}{\rm BF}^{\tau}_{\rm ac, \Bf}), \eta_\Bf \rangle_\DD \notag \\
        &= \frac{\mathscr{R}_{\rm ac}(\Bg_{\tau^{-1}})}{\mathscr{C}_{\rm ac}(\Bg_{\tau^{-1}})} \cdot \frac{\mathfrak{L}_\fp^{\tau}(\Bf)\cdot \mathfrak{L}_\fp^{\tau^{-1}}(\Bf)}{\mathfrak{L}^{\g_\tau^{-1}}_{{\rm ac},0}(\Bf, \tau^{-1})} \cdot \mathscr{L}^{\tau}_{\eta_{\Bf}}(\mathfrak{d}^{(\fp)}_{\rm{cyc}}{\rm BF}^{\tau}_{\rm ac, \Bf})  \notag \\
        &= \frac{\lambda_{\Bg_{\tau^{-1}}}}{\lambda_\Bf} \cdot \mathscr{R}_{\rm ac}(\Bg_{\tau^{-1}}) \cdot \frac{\mathfrak{L}_\fp^{\tau}(\Bf)\cdot \mathfrak{L}_\fp^{\tau^{-1}}(\Bf)}{\mathfrak{L}^{\g_\tau^{-1}}_{{\rm ac},0}(\Bf, \tau^{-1})} \cdot \mathfrak{L}^{\Bf}_{\rm ac, 1}(\Bf, \tau) . \notag
    \end{align}
    as required. Here, the second equality is Lemma \ref{lemma_comparing_pairings}, the third equality is Corollary \ref{corollary_reciprocity_for_anticyclotomic_BF}(ii), and the fourth equality follows from Theorem \ref{Theorem_Reciprocity_for_BF}(i).
\end{proof}
\subsubsection{} 
\label{subsubsec_623_2026_02_06}
We now calculate explicitly the specialisations of the formula proved in Theorem~\ref{Theorem_Big_GZ}. From the sequence of morphisms 
$$ \varprojlim_n \mathscr{O}_{\g_{\tau, n}} \simeq \mathcal{H}_\tau(\Gamma^\circ_\p) \lra \mathcal{H}_\tau(\Gamma^\circ_\p) \hatotimes \mathcal{H}(\Gamma^\circ_\cyc)/(\gamma^+-1) \simeq \mathcal{H}_\tau(\Gamma^\circ_\ac)$$
(cf. \eqref{eq_ses_gamma_gamma_gammaac_first}, \eqref{eq_ses_gamma_gamma_gammaac} and the surrounding discussion for details on the last isomorphism; we also recall the convention for $\tau$ adopted there) we see that the map $\chi_\tau:\mathcal{H}_\tau (\Gamma^\circ_\ac) \to E$ induced by an anticyclotomic character (of $p$-power conductor, satisfying $\chi_\tau|_{\Delta_\ac} = \tau$) can be pulled back to a specialisation of our CM family. Henceforth, all specialisations of $\g_\tau$ we shall work with are assumed to be of this type.

We set $\tau = \mathds{1}^{(p)} \in \widehat{\Delta}_\ac$. For $\Bh \in \{\f, \g\}$, we let $\kappa_h$ denote a specialisation of $\Bh$ (so $\kappa_{g}$ is induced by some anticyclotomic character $\chi =\chi_{\mathds{1}^{(p)}} $ as above). We write 
$$\dfrac{d^i}{ds^i}L_p^{(\f)}(\f_{\kappa_f},\chi,s)|_{s=\frac{k}{2}+1}\,, \qquad L_\p^{\rm BDP}(\f_{\kappa_f} \otimes \,  \chi, k/2+1)$$ 
for the respective specialisations of $\mathfrak{L}^\f_{\ac, i}(\f,\mathds{1}^{(p)})$ and $\mathfrak{L}^{\mathds{1}^{(p)}}_\p(\f)$ at $(\kappa_f, \kappa_g)$.

In the notation of \cite[\S 4.5.2]{BDV}, we recall the function $\mathscr{C}(l) = \mathscr{C}_2^{-1}(l) \cdot \mathscr{C}_3^{-1}(l)$ for $l \in U_\g$. From the definitions, we may write $\mathscr{C}(1):= \mathscr{C}_2^{-1} \cdot \mathscr{C}_3^{-1} $ for some algebraic constants $\mathscr{C}_2, \mathscr{C}_3$, and moreover $\mathscr{C}_3 \in K^\times$ depends only on the field $K$. On the other hand, since $f$ has trivial nebentype, we have the equality $\mathscr{C}_2 = i \cdot |D_K|^{\frac{1}{2}}$, under our running hypotheses \eqref{item_Disc}, \eqref{item_Heeg}. Thus $\mathscr{C}_2$ also depends only on $K$ in this case. We then define 
$$\mathscr{B} := \frac{2p}{p-1} \cdot \mathscr{C}^{-1}_2 \cdot \mathscr{C}_3^{-1}.$$
Furthermore, let $u_\p \in \mathcal{O}_K[\frac{1}{p}]^\times $ denote a generator for the ideal $\p^{h_k}$ and $\lambda_{g_{\rm Eis}} $ the specialisation of $\lambda_\g$ corresponding to $g_{\rm Eis}$. We define
$$\mathscr{A}:= \log_p(u_\p) \cdot \mathscr{R} \cdot \lambda_{\rm Eis}\in \overline{\QQ}^\times\,, $$
where the algebraicity is explained in \cite[Lemma 4.6]{BDV} (cf. Remark \ref{remark_R_ac_vs_L(g)}).
\begin{corollary}[$p$-adic Gross--Zagier theorem up to non-zero algebraic factors]
    \label{cor_2026_01_20_2022}
    Let $f_\alpha$ be as in the beginning of \S\ref{sec_main_GZ}. Under the hypotheses of Theorem \ref{Theorem_Big_GZ}, the following equality holds: 
    \begin{equation}
   \label{eq_first_GZ_Equality}
        \dfrac{h_f( z_{f}, z_{f})}{(4|D_K|)^{\frac{k}{2}}} =(\star) \cdot (-1)^{\frac{k}{2}} \cdot \frac{\mathscr{A}}{\mathscr{B}}\cdot \frac{d}{ds}L_p^{(\f)}(f_{\alpha},\mathds{1}^{(p)} ,s)|_{s=\frac{k}{2}+1}\, ,\quad  (\star) = \begin{cases}
    {\alpha \cdot \mathcal{E}(f_\alpha)\cdot \mathcal{E}^*(f_\alpha)}{\left(1-\frac{p^{\frac{k}{2}}}{\alpha}\right)^{-4}} &\hbox{ if } p\nmid N\,,\\
 \dfrac{\alpha}{4}  & \hbox{ if } p\mid \mid N \,.
    \end{cases}
    \end{equation}
\end{corollary}
\begin{proof}
     We present the argument in the case $p\nmid N$; the case $p\mid \mid N$ is entirely analogous. Let $\kappa$ denote the specialisation of $\f$ such that $\f_\kappa = f_\alpha$ and let $\varepsilon: \mathcal{H}(\Gamma^\circ_\ac) \to E$ denote the augmentation map. Setting $\tau =\mathds{1}^{(p)}$ in Theorem \ref{Theorem_Big_GZ} and specialising at $(\kappa, \varepsilon)$, and using Corollary~\ref{cor_thm_bigheegnermain}, we deduce that
    \begin{align}
    \begin{aligned}
    \label{eqn_2026_02_04_1211}
        \dfrac{\left(1-\frac{p^{\frac{k}{2}}}{\alpha}\right)^{4}}{(4|D_K|)^{\frac{k}{2}}}\cdot \dfrac{(-1)^{\frac{k}{2}}h_{f}(z_{f},z_{f})}{\alpha \cdot (-1)^{k/2} \cdot \mathcal{E}(f_\alpha)\cdot \mathcal{E}^*(f_\alpha)}  &= (-1)^{k/2}\cdot {\lambda_{g_{\rm Eis}}}\cdot\mathscr{R} \cdot \frac{L_\fp^{\rm BDP}(f_\alpha \otimes \mathds{1}^{(p)}, k/2+1)^2}{L _p^{(g_{\rm{Eis}})}(f_\alpha,\mathds{1}^{(p)},k/2+1)
        }\\
        &\hspace{3cm} \times \frac{d}{ds}L^{(f_\alpha)}_p(f_{\alpha},\mathds{1}^{(p)},s )|_{s=\frac{k}{2}+1}\,.
    \end{aligned}
    \end{align}
    According to \cite[Lemma 4.4]{BDV}, we have the equality $$\frac{L_\fp^{\rm BDP}(f_\alpha \otimes \mathds{1}^{(p)}, k/2+1)^2}{L _p^{(g_{\rm{Eis}})}(f_\alpha,\mathds{1}^{(p)},k/2+1)} = \frac{\log_p(u_\fp)}{\mathscr{B}}.$$ The result now follows after rearranging \eqref{eqn_2026_02_04_1211}.
\end{proof}

\subsubsection{} 
Corollary~\ref{cor_2026_01_20_2022} is readily sufficient to, among other things, conclude the proof of Perrin-Riou's conjecture (cf. \cite{BPSI}). We can also improve it to the following ``usual form'' of the $p$-adic Gross--Zagier formula in terms of Heegner cycles.
On setting 
\begin{equation}
    \label{eqn_2026_02_04_1051}
    L_p(f_{\alpha/K},s) := ( \clubsuit)\cdot L_p^{(\f)}(f_\alpha, \mathds{1}^{(p)},s)\,, 
\end{equation}
where
$$( \clubsuit) = \begin{cases}
        {\alpha \cdot \mathcal{E}(f_\alpha)\cdot \mathcal{E}^*(f_\alpha)} & \rm{if }\ p\nmid N, \\
        {\alpha}  &\rm{if }\ p\mid  \mid N,
    \end{cases}$$
    we obtain the following, more familiar version of  Corollary \ref{cor_2026_01_20_2022}.
\begin{theorem}
    \label{thm_2026_01_20_2021} Suppose that \eqref{itemH} holds and $p\nmid 6h_K$. We have
    \begin{equation}
    \label{eqn_2026_02_05_1024}
        \frac{\mathscr A}{\mathscr{B}}\cdot \frac{d}{ds}L_p(f_{\alpha/K},s)|_{s=\frac{k}{2}+1}= (-1)^{\frac{k}{2}}\times 
    \begin{cases}
    \left(1-\frac{p^{\frac{k}{2}}}{\alpha}\right)^{4}\cdot\dfrac{h_f( z_{f}, z_{f})}{(4|D_K|)^{\frac{k}{2}}} &\hbox{ if } p\nmid N\,,\\
    4\cdot\dfrac{h_f( z_{f}, z_{f})}{(4|D_K|)^{\frac{k}{2}}} & \hbox{ if } p\mid \mid N \,.
    \end{cases}
    \end{equation}
   Moreover, if \eqref{itemNV} holds\footnote{\label{footnote_Thm_Main} See Remark~\ref{remark_2026_02_06_1016} concerning \eqref{itemNV}, and Footnote~\eqref{footnote_BI} for a comment on the condition \eqref{item_BI}. In particular, if $f_\alpha$ is slope-zero, then we may relax \eqref{item_BI} to the requirement that the residual representation of $f_\alpha$ be absolutely irreducible.} as well, then $\mathscr{A/B}=1$.
\end{theorem}
\begin{proof}
    In view of Corollary~\ref{cor_2026_01_20_2022}, it suffices to prove that ${\mathscr{A}}/{\mathscr{B}}=1$ under \eqref{itemH} and \eqref{itemNV}. Note also that ${\mathscr{A}}/{\mathscr{B}}$ does not depend on $f_\alpha$. We will prove that it equals to $1$ by making a convenient choice of $f_\alpha$. Let $h$ be a $p$-stabilised eigenform satisfying \eqref{itemNV} (in the situation of \eqref{item_NV1}, we let $h=f_{E,\alpha}$ be the $p$-stabilisation of the newform $f_E$ associated to $E$). Comparing the $p$-adic Gross--Zagier formulae\footnote{We note that the Heegner class considered in \cite{perrinriou87} is equal to $h_K \cdot z_{h^\circ}$.} proved in \cite{perrinriou87, nekovarGZ, kobayashi13, kobayashi2014_GZ} for this particular choice of $h$ to \eqref{eqn_2026_02_05_1024}, we conclude the proof of our theorem. 
    \end{proof}
\label{appendix_padicL}
\bibliographystyle{amsalpha}
\bibliography{references}
\end{document}